\documentclass[leqno]{prims}
\usepackage{hyperref}
\usepackage{amssymb,amsmath}
\usepackage{enumerate}

\usepackage[dvips]{graphicx}
\usepackage{pst-all} 
\usepackage[all]{xy} 

\newtheorem{thm}{Theorem}[section] 
\newtheorem{prop}[thm]{Proposition}

\newtheorem{lem}[thm]{Lemma}

\theoremstyle{definition}
\newtheorem{defin}[thm]{Definition}
\newtheorem{rem}[thm]{Remark}

\newtheorem{ass}[thm]{Assumption}

\numberwithin{equation}{section}  

\def\p{\partial}
\def\tx{\tilde{x}}
\def\tt{\tilde{t}}
\def\tr{\tilde{r}}
\def\ta{\tilde{a}}
\def\tz{\tilde{z}}
\def\ts{\tilde{s}}
\def\tQ{\tilde{Q}}
\def\tA{\tilde{A}}
\def\tS{\tilde{S}}
\def\tE{\tilde{E}}

\def\tF{\tilde{F}}
\def\tU{\tilde{U}}
\def\tV{\tilde{V}}
\def\tlambda{\tilde{\lambda}}
\def\tnu{\tilde{\nu}}
\def\talpha{\tilde{\alpha}}
\def\tbeta{\tilde{\beta}}
\def\tdelta{\tilde{\delta}}
\def\tgamma{\tilde{\gamma}}
\def\tGamma{\tilde{\Gamma}}
\def\tphi{\tilde{\phi}}
\def\tPhi{\tilde{\Phi}}

\def\PJ{(P_{J})}
\def\HJ{(H_{J})}
\def\SLJ{(SL_{J})}
\def\DJ{(D_{J})}
\def\PII{(P_{\rm II})}
\def\HII{(H_{\rm II})}
\def\SLII{(SL_{\rm II})}
\def\DII{(D_{\rm II})}
\def\PIII{(P_{{\rm III'}(D_7)})}
\def\HIII{(H_{{\rm III'}(D_7)})}
\def\SLIII{(SL_{{\rm III'}(D_7)})}

\def\III{{\rm III'}(D_7)}

\newcommand{\Res}{\mathop{\rm Res}}

\VolumeNo{4x}           
\YearNo{201x}          
\communication{S.Mochizuki. 
Received December 5, 2013. 
Revised July 19, 2014.} 

\begin{document}



\TitleHead{Transformations on Stokes segments}
\title{On WKB theoretic transformations for 
Painlev\'e transcendents on degenerate Stokes segments}


\AuthorHead{Kohei Iwaki}
\author{Kohei \textsc{Iwaki}*
\footnote{$*$~Research Institute for Mathematical Sciences,
Kyoto University, Kyoto 606-8502, Japan;
\email{iwaki@kurims.kyoto-u.ac.jp}} }


\maketitle
\begin{abstract}
The WKB theoretic transformation theorem established 
in \cite{KT98} implies that the first Painlev\'e equation 
gives a normal form of Painlev\'e equations 
with a large parameter near a simple $P$-turning point. 
In this paper we extend this result and show that 
the second Painlev\'e equation $\PII$ and the 
third Painlev\'e equation $\PIII$ of type $D_7$ 
give a normal form of Painlev\'e equations 
on a degenerate $P$-Stokes segments connecting 
two different simple $P$-turning points and 
on a degenerate $P$-Stokes segment of loop-type, respectively.
That is, any 2-parameter formal solution of a Painlev\'e 
equation is reduced to a 2-parameter formal solution of 
$\PII$ or $\PIII$ on these degenerate $P$-Stokes segments
by our transformation.
\end{abstract}


\section{Introduction}

\begin{table}[h]
{\small
\begin{eqnarray*}
\hspace{-1.em} 
(P_{\rm I}) & : & 
\frac{d^{2} \lambda}{dt^{2}} = 
\eta^{2} (6\lambda^{2} + t), \\[+.2em]  %
\hspace{-1.em} 
(P_{\rm II}) & : & 
\frac{d^{2} \lambda}{dt^{2}} = 
\eta^{2} (2\lambda^{3} + t \lambda + c), \\[+.2em] %
\hspace{-1.em}  
(P_{{\rm III'}(D_6)}) & : & 
\frac{d^{2} \lambda}{dt^{2}} = 
\frac{1}{\lambda} \biggl( 
\frac{d \lambda}{dt} \biggr)^{2} - 
\frac{1}{t} \frac{d \lambda}{dt} + 
\eta^{2} \Bigl[ 
\frac{\lambda^{3}}{t^{2}} - 
\frac{c_{\infty} \lambda^{2}}{t^{2}} + 
\frac{c_{0}}{t} - \frac{1}{\lambda} \Bigr], \\[+.5em] %
\hspace{-1.em}  
(P_{{\rm III'}(D_7)}) & : & 
\frac{d^{2} \lambda}{dt^{2}} = \frac{1}{\lambda} 
\biggl(\frac{d \lambda}{dt} \biggr)^{2} 
- \frac{1}{t} \frac{d \lambda}{dt} + 
\eta^{2} \Bigl[ \frac{-2\lambda^{2}}{t^{2}} + 
\frac{c}{t} - \frac{1}{\lambda} \Bigr], \\[+.5em] %
\hspace{-1.em}  
(P_{{\rm III'}(D_8)}) & : & 
\frac{d^{2} \lambda}{dt^{2}} = \frac{1}{\lambda} 
\biggl(\frac{d \lambda}{dt} \biggr)^{2} - \frac{1}{t} 
\frac{d \lambda}{dt} + \eta^{2} \Bigl[ \frac{\lambda^{2}}{t^{2}} 
- \frac{1}{t} \Bigr], \\[+.5em] %
\hspace{-1.em}  
(P_{\rm IV}) & : & 
\frac{d^{2} \lambda}{dt^{2}} = \frac{1}{2\lambda} 
\biggl(\frac{d \lambda}{dt} \biggr)^{2} 
+\eta^{2} \Bigl[ \frac{3}{2}\lambda^3 + 4t\lambda^2 +
(2t^2 - 2c_{\infty})\lambda - 
\frac{2c_0^2}{\lambda} \Bigr], \\[+.5em] %
\hspace{-1.em}  
(P_{\rm V}) & : & 
\frac{d^{2} \lambda}{dt^{2}} = 
\left(\frac{1}{2\lambda}-\frac{1}{\lambda-1} \right)
\left(\frac{d\lambda}{dt}\right)^2-\frac{1}{t}\frac{d\lambda}{dt}
\\
& & \hspace{-.5em}
+\eta^{2} \frac{2\lambda(\lambda-1)^2}{t^2}
\Bigl[
\frac{c_{\infty}^2}{4}-\frac{c_0^2}{4} 
\frac{1}{\lambda^2}-\frac{c_1 \hspace{+.1em} t}{(\lambda-1)^2}
-\frac{t^2}{4}\frac{\lambda+1}{(\lambda-1)^3}
\Bigr], \\[+.5em] %
\hspace{-1.em} 
(P_{\rm VI}) & : & 
\frac{d^{2} \lambda}{dt^{2}} = 
\frac{1}{2} \biggl( 
\frac{1}{\lambda} + \frac{1}{\lambda - 1} 
+ \frac{1}{\lambda - t} \biggr) 
\biggl(\frac{d \lambda}{dt} \biggr)^{2} - 
\biggl(\frac{1}{t} + \frac{1}{t-1} + \frac{1}{\lambda - t}  
\biggr) \frac{d \lambda}{dt} \\[+.2em]
&   & \hspace{-2.em}
+ \frac{\lambda(\lambda-1)}{2t(t-1)(\lambda - t)} 
+ \eta^{2} \frac{2 \lambda (\lambda - 1) (\lambda - t)}
{t^{2}(t - 1)^{2}} 
\biggl[ 
\frac{c_{\infty}^2}{4}
- \frac{c_{0}^2}{4} \frac{t}{\lambda^{2}}  
+ \frac{c_{1}^2}{4} \frac{t - 1}{(\lambda - 1)^{2}} 
- \frac{c_{t}^2}{4} \frac{t (t - 1)}{(\lambda - t)^{2}}
 \biggr].%
\end{eqnarray*}
\caption{Painlev\'e equations with a large parameter $\eta$. }
\label{table:PJ}
 } 
\end{table}

{\em Painlev\'e transcendents} are remarkable special functions which 
appear in many areas of mathematics and physics. These are solutions 
of certain non-linear ordinary differential equations known as 
Painlev\'e equations. Since the work of Painlev\'e and Gambier there 
have been many works which investigate mutual relationships 
(mainly on the formal level) between different Painlev\'e equations, 
often called the degeneration or confluence procedure,  
or (double) scaling limits of Painlev\'e equations.
More recently, relations of solutions of different Painlev\'e equations 
have been also discussed; see \cite{Kit92, Kit94, Kapaev-Kitaev, 
Kit-Var, Kit06, GIL} and references therein. For example, 
\cite{Kapaev-Kitaev} describes solutions of the first Painlev\'e equation 
in terms of those of the second Painlev\'e equation using infinite 
times iteration of B\"acklund transformations. \cite{GIL} also
succeeds in giving a relation between solutions of different
Painlev\'e equations through their explicit expressions of 
$\tau$-functions and computations of the limit 
in the degeneration procedure.

Now, in this paper we discuss a different kind of relations between
solutions of Painlev\'e equations containing a large parameter $\eta$
(cf. Table \ref{table:PJ}) called a ``{\em WKB theoretic transformation}". 
Here a WKB theoretic transformation is an invertible formal coordinate
transformation which relates formal solutions of different Painlev\'e
equations. (See a series of papers \cite{KT96}, \cite{AKT96} and 
\cite{KT98} by Aoki, Kawai and Takei for more details of 
WKB theoretic transformations.)
%
%
%
%
%
The main result of this paper is the construction of new 
WKB theoretic transformations. 
That is, for any ``2-parameter (formal) solution" 
of a general Painlev\'e equation $\PJ$, we can find a formal 
invertible coordinate transformation which reduces 
the 2-parameter solution to a 2-parameter solution of 
$\PII$ or $\PIII$, when the configuration of 
``$P$-Stokes curves" of $\PJ$ degenerates and contains a 
$P$-Stokes curve connecting two ``$P$-turning points" 
(we call such a special $P$-Stokes curve a ``{\em $P$-Stokes segment}").




We explain 
the motivation of our study. 
Some of the important results by Aoki, Kawai and Takei 
are summarized as follows (see \cite{KT96}, \cite{AKT96} and \cite{KT98}): 
\begin{itemize}
\item %
notions of {\em $P$-turning points} and 
{\em $P$-Stokes curves} are introduced for $\PJ$, %
\item %
{\em 2-parameter (formal) solutions}
$\lambda_J(t,\eta;\alpha,\beta)$ of $\PJ$ containing 
two free parameters $\alpha$ and $\beta$ are constructed 
by the multiple-scale method, %
\item %
the {\em WKB theoretic transformation theory} 
near a simple $P$-turning point is established, that is, 
any 2-parameter solution of $\PJ$ 
can be reduced to that of the first Painlev\'e equation 
\begin{equation}
(P_{\rm I}) : \frac{d^{2} \lambda}{dt^{2}} = 
\eta^{2} (6\lambda^{2} + t)
\end{equation}
on a $P$-Stokes curve emanating from a {\em simple} 
$P$-turning point.
\end{itemize}
In this paper, for the sake of clarity, we call turning points 
(resp., Stokes curves) of Painlev\'e equations ``$P$-turning 
points" (resp., ``$P$-Stokes curves"), following the 
terminology used in \cite{KT06} for example.
The precise statement of the last claim is that, 
for any 2-parameter solution 
$\tlambda_J(\tt,\eta;\talpha,\tbeta)$ of $\PJ$, 
there exist formal coordinate transformation series 
$x(\tx,\tt,\eta)$ and $t(\tt,\eta)$ of dependent and 
independent variables and a 2-parameter solution 
$\lambda_{\rm I}(t,\eta;\alpha,\beta)$ of $(P_{\rm I})$ 
such that 
\begin{equation} \label{eq:transform-P1}
x(\tlambda_J(\tt,\eta;\talpha,\tbeta),\tt,\eta) = 
\lambda_{\rm I}(t(\tt,\eta),\eta;\alpha,\beta)
\end{equation}
holds in a neighborhood of a point $\tt=\tt_{\ast}$ 
which lies on a $P$-Stokes curve emanating 
from a simple $P$-turning point.
Here we put the symbol $\sim$ on the variables relevant to
$\PJ$ to distinguish them from those of $(P_{\rm I})$.  
In this sense the first Painlev\'e equation 
$(P_{\rm I})$ is a canonical equation of Painlev\'e 
equations near a simple $P$-turning point.

The above result can be considered as a non-linear analogue 
of the transformation theory of linear ordinary differential 
equations near a simple turning point. In the case of linear  
equations of second order, a canonical equation is given by 
the {\em Airy equation}:
\begin{equation} \label{eq:Airy}
\left(\frac{d^2}{dx^2} - \eta^2 x \right) \psi(x,\eta) = 0.
\end{equation} 
See \cite{AKT91} for the precise statement. 
The transformation gives an equivalence between 
{\em WKB solutions} of a general Schr{\"o}dinger equation 
and those of the Airy equation \eqref{eq:Airy} 
near a simple turning point, and consequently 
the explicit form of the connection formula 
on a Stokes curve for a general equation 
is determined in a 
{\em ``generic"} situation (\cite{KT iwanami}). 

The above genericity assumption means that 
the Stokes graph of the equation does not contain 
any {\em (degenerate) Stokes segments} 
(i.e., Stokes curves connecting simple turning points). 
We say that the Stokes geometry {\em degenerates} 
if such a Stokes segment appears. When a Stokes segment 
appears in the Stokes geometry, the connection formula 
does not make sense on the Stokes segment 
(cf.~\cite[Section 7]{Voros}).  

Typically two types of Stokes segments appear 
for Stokes geometry of linear equations in a 
generic situation: A Stokes segment of the first type 
connects two {\em different} simple turning points, while
a Stokes segment of the second type 
(sometimes called a {\em loop-type} Stokes segment)
emanates from and returns to the {\em same} simple 
turning point and hence forms a closed loop. 

To analyze the degenerate situation where 
a Stokes segment connects two {\em different} simple turning 
points for a general Schr{\"o}dinger equation, 
\cite{AKT09} constructs a transformation which brings 
WKB solutions of the general equation to that 
of the {\em Weber equation} when $x$ lies on a Stokes segment. 
Here the Weber equation they discussed 
has the form 
\begin{equation} \label{eq:Weber}
\left(\frac{d^2}{dx^2} - \eta^2 \Bigl( c-\frac{x^2}{4} \Bigr) 
\right) \psi(x,\eta) = 0.
\end{equation} 
To be more precise, we need 
to replace the constant $c$ by a formal power series 
$c=c(\eta)$ in $\eta^{-1}$ with constant coefficients 
in discussing the transformation. The Stokes geometry of the 
equation \eqref{eq:Weber} when $c\in{\mathbb R}_{\ne 0}$ 
(where ${\mathbb R}_{\ne 0}$ is the set of non-zero 
real numbers) has two simple turning points and 
a Stokes segment connects the two simple turning points. 
In this sense the Weber equation gives 
a canonical equation on a Stokes segment 
which connects two different simple turning points. 

On the other hand, recently Takahashi \cite{Takahashi}
constructs a similar kind of formal transformation 
which brings a general  Schr{\"o}dinger equation
having a loop-type Stokes segment to the
{\em Bessel-type equation} of the form 
\begin{equation} \label{eq:Bessel}
\left(\frac{d^2}{dx^2} - \eta^2 
\Bigl( \frac{x - c^2}{x^2} \Bigr) 
\right) \psi(x,\eta) = 0.
\end{equation} 
When $c\in{\mathbb R}_{\ne 0}$, 
the Stokes geometry of the equation \eqref{eq:Bessel} 
has one simple turning point and a Stokes curve 
emanating from the turning point turns around 
the double-pole $x=0$ of the potential and 
returns to the original simple turning point. 
This gives a loop-type Stokes segment. 
In this sense the Bessel-type equation gives 
a canonical equation on a loop-type Stokes segment.

The transformation constructed in \cite{AKT09} and 
\cite{Takahashi} are expected to play important roles 
in the analysis of {\em parametric} Stokes phenomena. 
Actually, if we vary the constant $c$, 
WKB solutions of \eqref{eq:Weber} may enjoy a Stokes 
phenomenon, that is, the correspondence between 
WKB solutions and their {\em Borel sums} 
changes discontinuously before and after the appearance 
of Stokes segments (cf.~\cite{SS}, \cite{Takei08}). 
We call such Stokes phenomena ``parametric" 
since the Stokes phenomena 
occur when we vary the parameter $c$ which is not the 
independent variable. 
Due to parametric Stokes 
phenomena, the transformation to the Airy equation does 
not work when a Stokes segment appears. 
Actually, a Stokes segment yields the so-called 
{\em fixed singularities} 
(cf.~\cite{DP99}, \cite{AKT09}) for the Borel transform 
of WKB solutions. Parametric Stokes phenomena are caused 
by such fixed singularities. 
The analysis of these fixed singularities is done 
in \cite{AKT09} through the transformation 
to the Weber equation. If the Borel summability 
of the transformation series constructed 
in \cite{AKT09} and \cite{Takahashi} is established, 
then the explicit form of the connection formula 
describing the parametric Stokes phenomena 
will be derived.

Here a natural question aries: 
What happens to 2-parameter solutions of $\PJ$ 
when the $P$-Stokes geometry {\em degenerates}, 
that is, when a $P$-Stokes segment appears 
in the $P$-Stokes geometry of $\PJ$.

It is shown in the author's papers 
\cite{Iwaki}, \cite{Iwaki-Bessatsu} and 
\cite{Iwaki-PIII} that the parametric Stokes phenomena 
also occur to {\em 1-parameter solutions} 
(which belongs to a subclass of 2-parameter solutions)
of the Painlev\'e equations when 
a $P$-Stokes segment appears. 
For example, when the parameter $c$ contained 
in the second Painlev\'e equation 
\begin{equation}
\PII : \frac{d^{2} \lambda}{dt^{2}} = 
\eta^{2} (2\lambda^{3} + t \lambda + c)
\end{equation}
is pure imaginary, $P$-Stokes segments appear 
in the $P$-Stokes geometry of $\PII$. 
In this case three $P$-Stokes segments appear
simultaneously and each of them connects two 
{\em different} simple $P$-turning points (see Section 
\ref{section:degeneration-of-PStokes-geometry}).
It is shown in \cite{Iwaki} that 
1-parameter solutions of $\PII$ 
enjoy Stokes phenomena when 
the $P$-Stokes segments appear. 
Similarly, a {\em loop-type} $P$-Stokes segment 
also appears in the $P$-Stokes geometry of 
the degenerate third Painlev\'e equation
\begin{equation}
\PIII : \frac{d^{2} \lambda}{dt^{2}} = \frac{1}{\lambda} 
\biggl(\frac{d \lambda}{dt} \biggr)^{2} 
- \frac{1}{t} \frac{d \lambda}{dt} + 
\eta^{2} \Bigl[ \frac{-2\lambda^{2}}{t^{2}} + 
\frac{c}{t} - \frac{1}{\lambda} \Bigr]
\end{equation}
of type $D_7$ (in the sense of \cite{OKSO})
when $c \in i \hspace{+.1em} {\mathbb R}_{\ne 0}$ 
(see Section \ref{section:degeneration-of-PStokes-geometry}).

Motivated by these results, in this paper we construct 
a transformation of the form \eqref{eq:transform-P1} 
when the $P$-Stokes geometry of $\PJ$ degenerates. 
That is, as is described below, 
(under some geometric assumptions for the 
Stokes geometry of isomonodromy systems,)
when a $P$-Stokes segment which connects 
two {\em different} simple $P$-turning points 
(resp., a {\em loop-type} $P$-Stokes segment) 
appears, then any 2-parameter solution of $\PJ$ is 
reduced to a 2-parameter solution of $\PII$ 
(resp., $\PIII$) on the $P$-Stokes segment
(see Section \ref{section:main-results} 
and Section \ref{section:main-theorem2} 
for the precise statements and assumptions). 

\begin{thm}[Thorem \ref{thm:main-theorem1}]
Assume that $\PJ$ has a $P$-Stokes segment 
connecting two different simple $P$-turning points of $\PJ$. 
Then, for any 2-parameter solution 
$\tlambda_J(\tt,\eta;\talpha,\tbeta)$ of $\PJ$, 
we can find
\begin{itemize}
\item formal coordinate transformation series 
$x(\tx,\tt,\eta)$ and $t(\tt,\eta)$ of dependent 
and independent variables, %
\item a 2-parameter solution 
$\lambda_{\rm II}(t,\eta;\alpha,\beta)$ of $\PII$ 
with a suitable choice of the constant $c$ 
in the equation, %
\end{itemize}
satisfying
\begin{equation}
x \bigl( \tlambda_J(\tt,\eta;\talpha,\tbeta),\tt,\eta 
\bigr) = \lambda_{\rm II} 
\left( t(\tt,\eta),\eta; \alpha,\beta \right)
\end{equation}
in a neighborhood of a point $\tt=\tt_{\ast}$ 
which lies on the $P$-Stokes segment. 
\end{thm}

\begin{thm}[Theorem \ref{thm:main-theorem2}]
Assume that $\PJ$ has a $P$-Stokes segment of loop-type.
Then, for any 2-parameter solution 
$\tlambda_J(\tt,\eta;\talpha,\tbeta)$ of $\PJ$, 
we can find
\begin{itemize}
\item formal coordinate transformation series 
$x(\tx,\tt,\eta)$ and $t(\tt,\eta)$ of dependent 
and independent variables, %
\item a 2-parameter solution 
$\lambda_{\III}(t,\eta;\alpha,\beta)$ of $\PIII$ 
with a suitable choice of the constant $c$ 
in the equation, %
\end{itemize}
satisfying
\begin{equation}
x \bigl( \tlambda_J(\tt,\eta;\talpha,\tbeta),\tt,\eta 
\bigr) = \lambda_{\III} \left( t(\tt,\eta),c,\eta;
\alpha,\beta \right)
\end{equation}
in a neighborhood of a point 
$\tt=\tt_{\ast}$ which lies on 
the $P$-Stokes segment of loop-type. 
\end{thm}

In this sense the equations $\PII$ and $\PIII$ 
give canonical equations of Painlev\'e equations 
on a $P$-Stokes segment connecting different 
simple $P$-turning points and a loop-type $P$-Stokes 
segment, respectively. Our main results can be 
considered as non-linear analogues of the transformation 
theory of \cite{AKT09} (to the Weber equation)
and \cite{Takahashi} (to the Bessel-type equation). 
We expect that, together with 
the previous results \cite{Iwaki}, 
\cite{Iwaki-Bessatsu} and \cite{Iwaki-PIII}, 
our transformation theory plays an important role 
in the analysis of parametric Stokes phenomena 
for Painlev\'e equations. 


This paper is organized as follows. In Section 
\ref{section:PJ-and-SLJ-and-2-parameter-solutions}
we briefly review some results of WKB analysis of 
Painlev\'e equations $\PJ$ and a role of isomonodromy 
systems $\SLJ$ and $\DJ$ associated with $\PJ$. 
Section \ref{section:Stokes-geometry} is devoted to 
descriptions of properties of the $P$-Stokes geometry 
of $\PJ$ and the Stokes geometry of $\SLJ$. 
Our main results together with assumptions are stated 
and proved in Section \ref{section:main-results} 
and Section \ref{section:main-theorem2}.

\section{Review of the exact WKB analysis of 
Painlev\'e transcendents with a large parameter}
\label{section:PJ-and-SLJ-and-2-parameter-solutions}

In this section we prepare some notations and review 
some results of \cite{KT96}, \cite{AKT96} and \cite{KT98}
that are relevant to this paper.

\subsection{2-parameter solution 
$\lambda_J(t,\eta;\alpha,\beta)$ of $\PJ$}

In \cite{AKT96} a 2-parameter family of formal solutions 
of $\PJ$, called a {\em 2-parameter solution}, is constructed 
by the so-called multiple-scale method. Here we introduce 
some notations to describe the solutions explicitly 
and to make our discussion smoothly. 
Most notations introduced here are 
consistent with those used in \cite{KT98}.

As is clear from Table \ref{table:PJ}, each 
$(P_{J})$ has the following form,
\[
(P_{J}) : \frac{d^2\lambda}{dt^2}=
G_{J}\left(\lambda,\frac{d\lambda}{dt},t\right)+
\eta^2F_{J}(\lambda,t),
\]
where $F_{J}$ is a rational function in $t$ and $\lambda$, 
and $G_{J}$ is a polynomial in $d\lambda/dt$ with degree 
equal to or at most 2, and rational in $\lambda$ and $t$.
Define the set ${\rm Sing}_J \subset {\mathbb P}^1$ 
of {\em singular points} of $\PJ$ by 
\begin{eqnarray}
\label{eq:singular-points}
{\rm Sing}_{\rm I} & = & {\rm Sing}_{\rm II} = \{\infty \},~
{\rm Sing}_{{\rm III'}(D_6)} = {\rm Sing}_{{\rm III'}(D_7)} 
= {\rm Sing}_{{\rm III}'(D_8)} = \{0,\infty \}, \hspace{-.5em} \\
{\rm Sing}_{\rm IV} & = & \{ \infty \},~
{\rm Sing}_{\rm V} = \{0,\infty \},~
{\rm Sing}_{{\rm VI}} = \{ 0,1,\infty\}, 
\hspace{-.5em} \nonumber
\end{eqnarray}
and the set $\Delta_J$ of {\em branch points} of $\PJ$ by 
\begin{equation}
\Delta_J = \{r\in{\mathbb P}^1\setminus{\rm Sing}_J~|~
F_J(\lambda,r) = ({\p F_J}/{\p \lambda})
(\lambda,r) = 0~\text{for some $\lambda$}\}.
\end{equation}
We also set $\Omega_J={\mathbb P}^1\setminus
({\rm Sing}_{J}\cup\Delta_J)$.

Fix a holomorphic function $\lambda_0(t)$ 
that satisfies 
\begin{equation}\label{eq:lam0}
F_{J}(\lambda_0(t),t)=0
\end{equation}
near a point $t_{\ast} \in \Omega_J$. 
The 2-parameter solutions are formal solutions 
of $\PJ$ defined in a neighborhood $V$ of $t_{\ast}$ 
of the following form:
\begin{equation} \label{eq:lambdaJ}
\lambda_J(t,\eta;\alpha,\beta) = \lambda_0(t) + 
\eta^{-1/2}\sum_{j=0}^{\infty}\eta^{-j/2}\Lambda_{j/2}
(t,\eta;\alpha,\beta).
\end{equation}
Here $(\alpha,\beta)=(\sum_{n=0}^{\infty}\eta^{-n}\alpha_{n}, 
\sum_{n=0}^{\infty}\eta^{-n}\beta_{n})$ is a pair of 
formal power series whose coefficients 
$\{(\alpha_{n},\beta_{n})\}_{n=0}^{\infty}$ 
parametrize the formal solution, and the functions 
\begin{equation} \label{eq:Lambdaj/2}
\Lambda_{j/2}(t,\eta;\alpha,\beta) = 
\sum_{m=0}^{j+1}a_{j+1-2m}^{(j/2)}(t) 
\exp((j+1-2m)\Phi_{J}(t,\eta))
\end{equation}
labeled by half-integers possess the following properties
(see \cite{AKT96}, \cite{KT98}).
\begin{itemize} 
\item For any $j \ge 0$ and $\ell = j+1-2m$ 
($m=0,\dots,j+1$), $a^{(j/2)}_{\ell}(t)$ 
is a holomorphic function of $t$ on $V$ and 
free from $\eta$. %
\item The functions $a^{(0)}_{\pm1}(t)$ contain 
the free parameters $(\alpha_0,\beta_0)$ as 
\begin{equation}
a^{(0)}_{+1}(t)=\frac{\alpha_0}
{\sqrt[4]{F^{(1)}_{J}(t)~C_{J}(\lambda_0(t),t)^2}},~~
a^{(0)}_{-1}(t)=\frac{\beta_0}
{\sqrt[4]{F^{(1)}_{J}(t)~C_{J}(\lambda_0(t),t)^2}},
\end{equation} 
where the function $F_J^{(1)}(t)$ is given by 
\begin{equation} \label{eq:F1}
F_J^{(1)}(t) = \frac{\p F_J}{\p\lambda}(\lambda_0(t),t),
\end{equation}
and $C_J(\lambda,t)$ is given in 
Table \ref{table:CJ-and-DJ}. %
\item The function $\Phi_{J}(t,\eta)$, 
which is also holomorphic 
in $t \in V$, is given by
\begin{equation}
\Phi_{J}(t,\eta) = \eta \phi_{J}(t) + \alpha_0 \beta_0 
\log (\theta_{J}(t) \eta^2),
\end{equation}
where
\begin{eqnarray} \label{eq:phaseJ}
\phi_{J}(t) & = & \int^t \sqrt{F^{(1)}_{J}(t)}~dt, 
\label{eq:phaseJ} 
\end{eqnarray}
and $\theta_J(t)$ is determined from $F_J$, $G_J$ 
and $\lambda_0(t)$ (cf.~\cite[Section 1]{KT98}).
We will fix the lower end point of $\phi_J(t)$ later. %
\item The functions $a^{(j/2)}_{\ell}(t)$ 
and ($\ell \ne \pm 1$) 
are determined recursively from \\
$\{a^{(j'/2)}_{j'+1-2m}(t) \}_{j'<j,~0 \le m \le j'+1}$. %
\item The functions $a^{(j/2)}_{\pm1}(t)=0$ for an odd integer 
$j$ while $a^{(j/2)}_{+1}(t)$ and $a^{(j/2)}_{-1}(t)$ for an 
even integer $j \ge 2$ satisfy a certain system of linear 
inhomogeneous differential equations of the following form:
\begin{eqnarray}
\Biggl\{\frac{d}{dt} + \frac{1}{4}\frac{d}{dt}
\log{F^{(1)}_{J}(t)}+
\frac{1}{2}\log{C_{J}(\lambda_0(t),t)} 
~~\qquad\qquad \\ -
\begin{pmatrix}
\alpha_0\beta_0 & \alpha_0^2 \\ \nonumber
 -\beta_0^2 & -\alpha_0\beta_0
\end{pmatrix}
\frac{d}{dt}\log\theta_{J}(t) \Biggr\}
\left( \begin{array}{cc} 
a^{(n)}_{+1} \\ a^{(n)}_{-1}\\ 
\end{array} \right) = 
\left( \begin{array}{cc} 
R^{(n)}_{+1} \\ R^{(n)}_{-1}\\ 
\end{array} \right), \hspace{-2.em}
\end{eqnarray}
where $R^{(n)}_{\pm1}$ is determined by 
$\{a^{(j'/2)}_{j'+1-2m}(t) \}_{j'<j=2n,~0 \le m \le j'+1}$.
The free parameters $(\alpha_n, \beta_n)$ ($n \ge 1$: integer) 
capture the ambiguity of solutions of the differential 
equation for $j=2n$.
\end{itemize}
Therefore, 2-parameter solutions are formal power series 
in $\eta^{-1/2}$ whose coefficients $\Lambda_{j/2}$ may contain 
$\eta$-dependent terms of the form of $\exp(\ell \Phi(t,\eta))$ 
for some $\ell \in {\mathbb Z}$, called the
{\em $\ell$-instanton term} in \cite{KT98}. In this paper 
``{formal series}" means such a series, and we say that 
``$\Lambda_{j/2}(t,\eta)$ is holomorphic in $t$" if 
coefficients of each instanton term in $\Lambda_{j/2}(t,\eta)$ 
are holomorphic in $t$. Note that $\Lambda_{j/2}(t,\eta)$ 
contains instanton terms in such a way that, 
if $j$ is odd (resp., even), 
then $\Lambda_{j/2}(t,\eta)$ contains only even (resp., odd) 
instanton terms. We call this property the 
{\em alternating parity} of 2-parameter solutions. 
In order to avoid some degeneracy, we assume the condition
\begin{equation} \label{eq:genericity}
\alpha_0 \beta_0 \ne 0
\end{equation}
throughout this paper. 

\begin{table}[h]
\begin{eqnarray*}
C_{\rm I}(\lambda,t) & = & C_{\rm II}(\lambda,t) = 1,~~
C_{{\rm III'}(D_6)}(\lambda,t) =  C_{{\rm III'}(D_7)}(\lambda,t) 
= C_{{\rm III'}(D_8)}(\lambda,t) = \frac{t}{2\lambda^2},\\
C_{\rm IV}(\lambda,t) & = & \frac{1}{4\lambda},~~
C_{\rm V}(\lambda,t) = \frac{t}{2\lambda(\lambda-1)^2}, ~~
C_{\rm VI}(\lambda,t) = \frac{t(t-1)}
{2\lambda(\lambda-1)(\lambda-t)}. 
\end{eqnarray*}
\caption{$C_J(\lambda,t)$.}
\label{table:CJ-and-DJ}
\end{table}

It is well-known that the Painlev\'e equation $(P_{J})$ 
is equivalent to the following Hamiltonian system 
(e.g.,~\cite{Okamoto}):
\[
(H_{J}) : \frac{d\lambda}{dt}=\eta \frac{\p K_{J}}{\p\nu},~~
\frac{d\nu}{dt}=-\eta \frac{\p K_{J}}{\p\lambda}.
\]
Here the explicit form of Hamiltonians 
$K_{J} = K_{J}(t,\lambda,\nu,\eta)$ are tabulated in 
Table \ref{table:KJ}. 
From 2-parameter solutions of $\PJ$, we can also construct 
2-parameter solutions of the Hamiltonian system $(H_J)$. 
From the explicit form of the Hamiltonians $K_J$, we see that 
$\nu$ is written by $\lambda$ and its first order derivative. 
Consequently, $\nu_J = \nu_J(t,\eta;\alpha,\beta)$ has 
the following form:
\begin{equation}
\nu_J(t,\eta;\alpha,\beta) = \eta^{-1/2} 
\sum_{j=0}^{\infty}\eta^{-j/2}N_{j/2}
(t,\eta;\alpha,\beta),
\end{equation}
where $N_{j/2}$ has a similar form as $\Lambda_{j/2}$; 
that is, it contains instanton terms and enjoys
the alternating parity. 

\begin{rem}
If we put $\alpha = 0$ or $\beta = 0$ or $\alpha=\beta=0$, 
then 2-parameter solutions are reduced to 
{\em 1-parameter solutions} or {\em 0-parameter solutions}. 
Here $\alpha = 0$ etc., mean that all $\alpha_n$ are set 
to be $0$ etc. 1-parameter solutions are also called 
{\em trans-series solutions}. We can expect that 1-parameter
solutions and 0-parameter solutions interpreted as
analytic solutions through the {\em Borel resummation method} 
(see \cite{Kamimoto-Koike} 
for example). 
\end{rem}


%
\begin{table}[h]
\begin{eqnarray*}
K_{\rm I} & = & \frac{1}{2}\left[\nu^2 - 
(4\lambda^3+2t\lambda)\right],\\[+.5em] %
K_{\rm II} & = & \frac{1}{2}\left[\nu^2 - 
(\lambda^4+t\lambda^2+2c\lambda)\right],\\[+.5em] %
K_{{\rm III'}(D_6)} & = &
\frac{\lambda^2}{t}\left[\nu^2 - 
\eta^{-1}\frac{\nu}{\lambda} - 
\left( \frac{t^2}{4\lambda^4} - 
\frac{c_0 t}{2\lambda^3} 
- \frac{c_{\infty}}{2\lambda} + 
\frac{1}{4} \right) \right],\\[+.5em]%
K_{{\rm III'}(D_7)} & = &
\frac{\lambda^2}{t}\left[\nu^2 - 
\eta^{-1}\frac{\nu}{\lambda} - 
\left( \frac{t^2}{4\lambda^4} - 
\frac{c t}{2\lambda^3} 
- \frac{c^2}{4\lambda^2} - 
\frac{1}{\lambda} \right) \right],\\[+.5em]%
K_{{\rm III'}(D_8)} & = &
\frac{\lambda^2}{t}\left[\nu^2 - 
\eta^{-1}\frac{\nu}{\lambda} - 
\left( \frac{t}{2\lambda^3}+
\frac{1}{2\lambda} \right) \right],\\[+.5em]%
K_{{\rm IV}} & = &
2\lambda\left[\nu^2 - 
\eta^{-1}\frac{\nu}{\lambda} - 
\left( 
\frac{c_0^2 - \eta^{-2}}{4\lambda^2}-\frac{c_{\infty}}{4} 
+ \left(\frac{\lambda+2t}{4}\right)^2 
\right) \right], \\[+.5em]%
K_{{\rm V}} & = &
\frac{\lambda(\lambda-1)^2}{t}\biggl[\nu^2 - 
\eta^{-1}\left(\frac{1}{\lambda}+
\frac{1}{\lambda-1}\right)\nu \\[+.2em] 
& & -
\left( 
\frac{c_0^2 - \eta^{-2}}{4\lambda^2}
+ \frac{t^2}{4(\lambda-1)^4}
+\frac{c_1 t}{(\lambda-1)^3} 
+ 
\frac{c^2_{\infty}-c_0^2 - 3\eta^{-2}}{4(\lambda-1)^2} 
\right) \biggr], \\[+.5em]%
K_{{\rm VI}} & = & \frac{\lambda(\lambda-1)(\lambda-t)}{t(t-1)}
\biggl[ \nu^2 - \eta^{-1}\left(\frac{1}{\lambda}+
\frac{1}{\lambda-1}\right)\nu -
\biggl( \frac{c_0^2 - \eta^{-2}}{4\lambda^2} \\[+.2em]
& & 
+ \frac{c_1^2-\eta^{-2}}{4(\lambda-1)^2} + 
\frac{c_t^2 - \eta^{-2}}{4(\lambda-t)^2} +
\frac{c_{\infty}^2-(c_0^2+c_1^2+c_t^2)-\eta^{-2}}
{4\lambda(\lambda-1)} 
\biggr) \biggr].
\end{eqnarray*}
\caption{Hamiltonians of $\HJ$.}
\label{table:KJ}
\end{table}

\subsection{Isomonodromy system for $\PJ$ and WKB solutions}
\label{section:SLJ}

The Hamiltonian system $(H_J)$ arises when we 
consider {\em isomonodromic deformations} 
(see \cite{Jimbo-Miwa}, \cite{Okamoto}) of a certain 
Schr{\"o}dinger equation of the form
\[
(SL_{J}):\left(\frac{\p^2}{\p x^2}-
\eta^{2}Q_{J}(x,t,\eta)\right)\psi(x,t,\eta)=0.
\]
More precisely, there exists another differential equation 
\[
(D_{J}): \frac{\p}{\p t}\psi(x,t,\eta) = 
\left( A_J(x,t,\eta) \frac{\p}{\p x} 
- \frac{1}{2} \frac{\p A_J}{\p x}(x,t,\eta) 
\right) \psi(x,t,\eta),
\]
called {\em deformation equation}, such that $(H_J)$ describes 
the compatibility condition of the system of linear differential 
equations $(SL_{J})$ and $(D_J)$. See Table \ref{table:QJ} and 
\ref{table:AJ} for the explicit forms of $Q_J$ and $A_J$. 

\begin{table}[h]
\begin{eqnarray*}
Q_{\rm I} & = & 4x^3+2tx+2K_{\rm I}-\eta^{-1}\frac{\nu}{x-\lambda}
+\eta^{-2}\frac{3}{4(x-\lambda)^2}, \\[+.5em]%
Q_{\rm II} & = & x^4+tx^2+2cx+2K_{\rm II}-
\eta^{-1}\frac{\nu}{x-\lambda}+
\eta^{-2}\frac{3}{4(x-\lambda)^2},  \\[+.5em]%
Q_{{\rm III'}(D_{6})}
& = & \frac{t^{2}}{4x^{4}} - \frac{c_{0} t}{2 x^{3}} - 
\frac{c_{\infty}}{2x} + \frac{1}{4} + 
\frac{t \hspace{+.1em} {K}_{{\rm III'}(D_{6})}}{x^{2}} - 
\eta^{-1} \frac{\lambda \nu}{x(x-\lambda)} + 
\eta^{-2} \frac{3}{4(x-\lambda)^{2}}, \\[+.5em]%
Q_{{\rm III'}(D_{7})} 
& = & \frac{t^{2}}{4x^{4}} - \frac{c \hspace{+.1em} t}{2 x^{3}} + 
\frac{c^{2}}{4x^{2}} - \frac{1}{x} + 
\frac{t \hspace{+.1em} {K}_{{\rm III'}(D_{7})}}{x^{2}} - 
\eta^{-1} \frac{\lambda \nu}{x(x-\lambda)} + 
\eta^{-2} \frac{3}{4(x-\lambda)^{2}}, \\[+.5em] %
Q_{{\rm III'}(D_{8})} 
& = & \frac{t}{2x^{3}} + \frac{1}{2x} + 
\frac{t \hspace{+.1em} {K}_{{\rm III'}(D_{8})}}{x^{2}} - 
\eta^{-1} \frac{\lambda \nu - \eta^{-1}}{x(x-\lambda)} + 
\eta^{-2} \frac{3}{4(x-\lambda)^{2}}, \\[+.5em] %
Q_{{\rm IV}}
& = & 
\frac{c_0^2 - \eta^{-2}}{4x^{2}} - \frac{c_{\infty}}{4} +  
\left(\frac{x+2t}{4}\right)^2 + 
\frac{{K}_{{\rm IV}}}{2x} - 
\eta^{-1} \frac{\lambda \nu}{x(x-\lambda)} + 
\eta^{-2} \frac{3}{4(x-\lambda)^{2}}, \\[+.5em]%
Q_{{\rm V}}
& = & 
\frac{c_0^2 - \eta^{-2}}{4x^{2}} + \frac{t^2}{4(x-1)^4}
+ \frac{c_{1} t}{(x-1)^3} - 
\frac{c_{\infty}^2-c_0^2 - 3\eta^{-2}}{4(x-1)^2} \\[+.3em]
& &  + 
\frac{t \hspace{+.1em} {K}_{{\rm V}}}{x(x-1)^{2}} - 
\eta^{-1} \frac{\lambda(\lambda-1) \nu}{x(x-1)(x-\lambda)} 
+ \eta^{-2} \frac{3}{4(x-\lambda)^{2}}, \\[+.5em]%
Q_{{\rm VI}}
& = & \frac{c_0^2 - \eta^{-2}}{4x^2} + 
\frac{c_1^2-\eta^{-2}}{4(x-1)^2} + 
\frac{c_t^2 - \eta^{-2}}{4(x-t)^2} +
\frac{c_{\infty}^2-(c_0^2+c_1^2+c_t^2)-\eta^{-2}}
{4x(x-1)}\\[+.3em]%
& & 
+ \frac{t(t-1) \hspace{+.1em} {K}_{{\rm VI}}}
{x(x-1)(x-t)} - 
\eta^{-1} \frac{\lambda(\lambda-1) \nu}{x(x-1)(x-\lambda)} 
+ \eta^{-2} \frac{3}{4(x-\lambda)^{2}}
\end{eqnarray*}
\caption{Coefficient of $(SL_J)$.}
\label{table:QJ}
\end{table}
\begin{table}[h]
\begin{eqnarray*}
A_{\rm I} & = & A_{\rm II} = \frac{1}{2(x-\lambda)},~~
A_{{\rm III'}(D_6)} = A_{{\rm III'}(D_7)} =
A_{{\rm III'}(D_8)} =  \frac{\lambda x}{t(x-\lambda)}, \\[+.5em] %
A_{\rm IV} & = &   \frac{2x}{x-\lambda},~~
A_{\rm V} = \frac{\lambda-1}{t}\frac{x(x-1)}{x-\lambda},~~
A_{\rm VI} = \frac{\lambda-t}{t(t-1)}
\frac{x(x-1)}{x-\lambda}.
\end{eqnarray*}
\caption{Coefficient of $(D_J)$.}
\label{table:AJ}
\end{table}


Substituting 2-parameter solutions 
$(\lambda_J(t,\eta;\alpha,\beta),\nu_J(t,\eta;\alpha,\beta))$
into $(\lambda,\nu)$ that appears in $Q_J$ and $A_J$, 
we find that they have the same type formal series 
expansion as
\begin{equation} \label{eq:QJ-and-AJ}
Q_J(x,t,\eta) = \sum_{j=0}^{\infty}Q_{j/2}(x,t,\eta), \quad
A_J(x,t,\eta) = \sum_{j=0}^{\infty}A_{j/2}(x,t,\eta). 
\end{equation}
Here we omit writing explicitly 
the dependence on $\alpha$ and $\beta$ for 
simplicity. The top term $Q_0 = Q_{J,0}(x,t)$ 
is independent of $\eta$ 
(i.e., it does not contain instanton terms), 
and can be written in the form 
\begin{equation} \label{eq:Q0}
Q_{J, 0}(x,t) = C_{J}(x,t)^2(x-\lambda_0(t))^2 R_{J}(x,t).
\end{equation}
Thus, $Q_{J, 0}(x,t)$ has a double zero at 
$x=\lambda_0(t)$ in general. 
(Here we have used the fact that $\lambda_0(t)$ 
is defined by the algebraic equation \eqref{eq:lam0}.)
Here $R_J(x,t)$ is a polynomial in $x$ which satisfies
\begin{equation} \label{eq:RJ-at-lambda0}
R_J(\lambda_0(t),t) = F^{(1)}_J(t).
\end{equation}
We can verify that $R_{\rm I}(x,t)$,  $R_{\III}(x,t)$ 
and $R_{{\rm III}'(D_8)}(x,t)$ are polynomial in $x$ 
of degree 1, while $R_{J}(x,t)$ for other $J$ 
are polynomial in $x$ of degree 2. 


In what follows, we always assume that 
a 2-parameter solution 
$(\lambda_J, \nu_J)$ of $(H_J)$ is substituted 
into $(\lambda,\nu)$ which appears 
in the coefficients of $(SL_J)$ and $(D_J)$. 
For such a Schr{\"o}dinger equation $(SL_J)$, we can 
construct {\em WKB solutions} of the following form: 
\begin{equation} \label{eq:WKBsol}
\psi_{J,\pm}(x,t,\eta) = 
\frac{1}{\sqrt{S_{J,{\rm odd}}(x,t,\eta)}}
\exp\left(\pm \int^x S_{J,{\rm odd}}(x,t,\eta)~dx\right). 
\end{equation}
Here $S_{J,{\rm odd}}(x,t,\eta)$ is the 
{\em odd part} of 
a formal series solution $S_J(x,t,\eta)$ of 
\begin{equation} \label{eq:Riccati}
S^2 + \frac{\p S}{\p x} = \eta^2 Q_J(x,t,\eta),
\end{equation}
which is called the {\em Riccati equation} associated 
with $(SL_J)$. Here the odd part $S_{J,{\rm odd}}(x,t,\eta)$
is defined as follows (see \cite{AKT96} for details). 
We can find two formal series solutions 
\begin{equation} \label{eq:S-plus-minus}
S^{(\pm)}_J(x,t,\eta) = \eta S^{(\pm)}_{-1}(x,t) + 
\sum_{j=0}^{\infty}\eta^{-j/2} S^{(\pm)}_{j/2}(x,t,\eta)
\end{equation}
of \eqref{eq:Riccati} starting from 
\begin{equation} \label{eq:S-1}
S^{(\pm)}_{-1}(x,t) = \pm \sqrt{Q_{J,0}(x,t)}. 
\end{equation}
Once we fix the sign in \eqref{eq:S-1} 
(i.e., the branch of square root),
the subsequent terms are determined 
by a recursion relation.
Then, $S_{J, \rm odd}(x,t,\eta)$ is given by
\begin{eqnarray}
S_{J,{\rm odd}}(x,t,\eta) & = & \frac{1}{2} 
\left(S^{(+)}_J(x,t,\eta)-S^{(-)}_J(x,t,\eta) \right) \\
& = & \nonumber
\eta S_{-1}(x,t) + \sum_{j=0}^{\infty} \eta^{-j/2}
S_{{\rm odd}, j/2}(x,t,\eta).
\end{eqnarray}
The integral of $S_{J,\rm odd}(x,t,\eta)$ appeared in 
\eqref{eq:WKBsol} is defined by the term-wise integral
of formal series. We discuss the choice of lower end 
point of \eqref{eq:WKBsol} later.

The formal series $S_{J,\rm odd}(x,t,\eta)$ etc. are 
constructed in the above manner for a fixed $t$ and 
have several good properties as a function of $t$. 
Firstly, $S_{J,\rm odd}(x,t,\eta)$ 
also has the property of alternating parity; if $j$ is odd 
(resp., even), then $S_{{\rm odd}, j/2}(x,t,\eta)$ contains 
only odd (resp., even) instanton terms. Secondly, the derivative 
of $S_{J,\rm odd}(x,t,\eta)$ with respect to $t$ satisfies 
the following equation.
\begin{prop}[{\cite[Proposition 2.1]{AKT96}}] 
\label{prop:t-derivative-of-Sodd}
The formal solutions $S^{(\pm)}_J(x,t,\eta)$ satisfy 
\begin{equation}
\frac{\p}{\p t}S^{(\pm)}_J(x,t,\eta) = \frac{\p}{\p x} 
\left( S_J^{(\pm)}(x,t,\eta) A_J(x,t,\eta) - \frac{1}{2}
\frac{\p A_J}{\p x}(x,t,\eta) \right)
\end{equation}
and hence we have
\begin{equation} \label{eq:t-derivative-of-Sodd}
\frac{\p}{\p t}S_{J,{\rm odd}}(x,t,\eta) = \frac{\p}{\p x} 
\left( S_{J,{\rm odd}}(x,t,\eta) A_J(x,t,\eta) \right).
\end{equation}
\end{prop}
Proposition \ref{prop:t-derivative-of-Sodd} is proved by using 
the isomonodromic property of $(SL_{J})$, that is, 
the compatibility of $\SLJ$ and $\DJ$. As a corollary, 
we obtain the following important (formal series valued) 
first integral of $\PJ$ from $\SLJ$. 
\begin{lem}[{\cite[Section 3]{AKT96}}]
The formal series $E(\eta)$ defined by 
\begin{equation} \label{eq:EJ}
E_J(\eta) = 4 \hspace{-.6em}
{\Res_{~~x=\lambda_0(t)}} 
S_{J,{\rm odd}}(x,t,\eta)~dx
\end{equation}
is independent of $t$. 
\end{lem}
The independence of $t$ implies that $E_J(\eta)$ 
must be a formal power series in $\eta^{-1}$; 
$E_J(\eta) = \sum_{n=0}^{\infty} \eta^{-n} E_n$ 
with some constants $E_n$. The free parameters 
$\alpha_n$ and $\beta_n$ of a 2-parameter solution 
are contained in $E_n$ in the following manner.
\begin{lem} [{\cite[Lemma 3.2]{KT98}}] \label{lemma:E}
\begin{enumerate}[\upshape (i)]
\item %
The top term $E_0$ of $E_J(\eta)$ is given by
\begin{equation}
E_0 = - 8 \alpha_0 \beta_0.
\end{equation} %
\item 
The coefficient $E_n$ of $\eta^{-n}$ in $E_J(\eta)$
depends only on $\{\alpha_{n'},\beta_{n'} \}_{0 \le n' \le n}$.
Furthermore, $E_n + 8(\alpha_0\beta_n+\alpha_n\beta_0)$ 
is independent of $(\alpha_n,\beta_n)$.
\end{enumerate}
\end{lem}

\begin{rem}
Let us take a generic point $t_{\ast}$ such that 
$Q_{J,0}(x,t)$ has a simple zero $x=a(t)$ at any 
point $t$ in a neighborhood 
of $t_{\ast}$. It is known that each coefficient 
$S_{{\rm odd},j/2}(x,t,\eta)$ of 
$S_{J,{\rm odd}}(x,t,\eta)$ 
has a square root type singularity at a simple zero 
of $Q_{J,0}$ (e.g., \cite[Section 2]{KT iwanami}). Due to 
this property we can define the WKB solution of $\SLJ$
which is ``well-normalized" at $x=a(t)$ as follows:
\begin{equation} \label{eq:WKBsol-tp}
\psi_{J,\pm}(x,t,\eta) = 
\frac{1}{\sqrt{S_{J,{\rm odd}}(x,t,\eta)}}
\exp\left(\pm \int_{a(t)}^x 
S_{J,{\rm odd}}(x,t,\eta)~dx\right). 
\end{equation}
Here the integral in \eqref{eq:WKBsol-tp} is defined 
as a contour integral; that is, 
\[
\int_{a(t)}^x S_{J,{\rm odd}}(x,t,\eta)~dx = 
\frac{1}{2} \int_{\delta_x} 
S_{J,{\rm odd}}(x,t,\eta)~dx,
\]
where the path $\delta_x$ is depicted in Figure 
\ref{fig:normalized-at-TP}. In Figure 
\ref{fig:normalized-at-TP} the wiggly line
is a branch cut to determine the branch of 
$\sqrt{Q_{J,0}(x,t)}$, and the solid (resp., the dashed) 
line is a part of the path $\delta_x$ on the first 
(resp., the second) sheet of the Riemann surface of 
$\sqrt{Q_{J,0}(x,t)}$. 
Then, we can show that the well-normalized WKB solutions 
\eqref{eq:WKBsol-tp} satisfy both $(SL_{J})$ and $(D_J)$ 
by using \eqref{eq:t-derivative-of-Sodd} 
(cf. \cite[Lemma 1]{Takei00}). 
\end{rem}
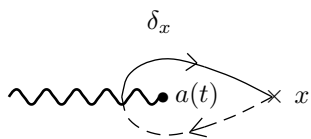
\begin{figure}[h]
\begin{center}
\begin{pspicture}(0,0)(4,2)
%
\psset{fillstyle=none}
\rput[c]{0}(3.53,0.51){$\times$}
\rput[c]{0}(2.05,0.48){$\bullet$}
\rput[c]{0}(2,1.5){$\delta_x$}
\rput[c]{0}(3.9,0.5){$x$}
\rput[c]{0}(2.5,0.52){$a(t)$}
\psset{linewidth=1pt}
\pscurve(2,0.5)(1.9,0.4)(1.8,0.5)(1.7,0.6)(1.6,0.5)
(1.5,0.4)(1.4,0.5)(1.3,0.6)(1.2,0.5)(1.1,0.4)(1,0.5)
(0.9,0.6)(0.8,0.5)(0.7,0.4)(0.6,0.5)(0.5,0.6)(0.4,0.5)
(0.3,0.4)(0.2,0.5)(0.1,0.6)(0,0.5)
%
%
%
\psset{linewidth=0.5pt}
%
\psecurve(1.45,0.2)(1.5,0.5)(1.6,0.8)(2,1)(2.6,0.9)(3.5,0.5)
\pscurve(2.6,0.9)(3,0.75)(3.5,0.5)
\psline(2.5,0.95)(2.37,1.08)
\psline(2.5,0.95)(2.35,0.85)
\psline(2.4,0.025)(2.57,0.22)
\psline(2.4,0.025)(2.63,-0.03)
\psset{linewidth=0.5pt, linestyle=dashed}
\psecurve(1.45,0.8)(1.5,0.5)(1.6,0.2)(2,0)(2.6,0.1)(3.5,0.5)
\pscurve(2.6,0.1)(3,0.25)(3.5,0.5)
\end{pspicture}
\end{center}
\caption{Path of integration $\delta_x$.} 
\label{fig:normalized-at-TP}
\end{figure}

The following proposition will play an important role 
in the proof of our main theorems.
\begin{prop}\label{prop:residue}
Let $p$ be an even order pole of $Q_{J,0}(x,t)$ 
(hence, a singular point of $\SLJ$), and set 
\begin{equation}
{\rm Res}({SL_J},p) = \Res_{x=p}S_{J, {\rm odd}}(x,t,\eta)~dx.
\end{equation}
Then, the list of ${\rm Res}(SL_J,p)$ for all $J$ and $p$ 
is given in Table \ref{table:residues}, up to the sign. 
\begin{table}[h]
\begin{eqnarray*}
{\rm Res}(SL_{\rm II},\infty) & = & c \eta, \\[+.5em]
{\rm Res}(SL_{{\rm III'}(D_7)},0) & = & \frac{c\eta}{2} \\[+.5em]
{\rm Res}(SL_{{\rm III'}(D_6)},0) & = & 
\frac{c_{0} \eta}{2},~~ 
{\rm Res}(SL_{{\rm III'}(D_6)},\infty) = 
\frac{c_{\infty} \eta}{2}, \\[+.5em]
{\rm Res}(SL_{\rm IV},0) & = & \frac{c_0 \eta}{2},~~ 
{\rm Res}(SL_{\rm IV},\infty) = 
\frac{c_{\infty}\eta}{2},\\[+.5em]
{\rm Res}(SL_{\rm V},\infty) & = & 
{\frac{c_{\infty} \eta}{2}},~~
{\rm Res}(SL_{\rm V},0) = 
\frac{c_{0} \eta}{2}, ~~ 
{\rm Res}(SL_{\rm V},1) = 2 c_1 \eta, \\[+.7em]
{\rm Res}(SL_{\rm VI},\infty) & = & 
\frac{c_{\infty} \eta}{2},~~
{\rm Res}(SL_{\rm VI},0) = 
\frac{c_{0} \eta}{2},  \\
{\rm Res}(SL_{\rm VI},1) & = & \frac{c_1\eta}{2},~~~
{\rm Res}(SL_{\rm VI},t) = \frac{c_t\eta}{2}. 
\end{eqnarray*}
\caption{The list of 
${\rm Res}(SL_J,p)$ at singular points of $\SLJ$.}
\label{table:residues}
\end{table}
\end{prop}
\begin{proof}
Let us show the claim when $J={\rm II}$ and 
$p = \infty$. The coefficients 
$\{S^{(\pm)}_{j/2}(x,t,\eta) \}_{j \ge -2}$
of the formal series $S^{(\pm)}_J(x,t,\eta)$ 
in \eqref{eq:S-plus-minus} must satisfy the 
recursion relations 
\begin{eqnarray}  
\label{eq:recursion-1}
S^{(\pm)}_{-1}(x,t) & = & \pm \sqrt{Q_{J,0}(x,t)},~~
S^{(\pm)}_{-1/2}(x,t) = 0, \\[+.5em]
S^{(\pm)}_{(j+2)/2} & = & \frac{1}{2 S^{(\pm)}_{-1}} 
\Biggl( Q_{(j+4)/2} - \frac{\p S^{(\pm)}_{j/2}}{\p x} - 
\sum_{ \scriptsize
\begin{array}{ll}
j_{1}+j_{2} = j  \\ ~0 \le j_{1}, j_{2} \le j
\end{array} } S^{(\pm)}_{j_1/2}S^{(\pm)}_{j_2/2} \Biggr) 
\label{eq:recursion} \\
& &  \hspace{+19.em} (j\ge-2) \nonumber
\end{eqnarray}
since $S^{(\pm)}_J(x,t,\eta)$ solve the Riccati equation 
\eqref{eq:Riccati}. We can then directly compute 
the asymptotic behavior of $S^{(\pm)}_{j/2}(x,t,\eta)$ 
near $x=\infty$ from the recursion relations 
\eqref{eq:recursion-1} and \eqref{eq:recursion} 
and the explicitform of the potential $Q_{J}$ 
in Table \ref{table:QJ}; for example, 
when $J={\rm II}$, those are given by 
\begin{eqnarray} \label{eq:QII0-residue}
S^{(\pm)}_{j/2}(x,t,\eta) = 
\begin{cases}  
\displaystyle 
\pm \left( x^2 + \frac{t}{2} - c \hspace{+.1em}
x^{-1} + O(x^{-2}) \right) & 
\text{if $j=-2$},  \\
O(x^{-2}) & \text{if $j \ge -1$}.
\end{cases}
\end{eqnarray}
Thus we have 
\begin{equation}
\Res_{x=\infty} S_{\rm II, odd}(x,t,\eta)~dx = 
\pm c \hspace{+.1em} \eta.
\end{equation}
In a similar manner we cam compute residues of 
$S_{J, \rm odd}(x,t,\eta)~dx$ at each singular point 
for the other $J$'s by straightforward computations.  
Actually, when $p$ is a regular singular point 
of $\SLJ$, we need more careful computation 
since $S^{(\pm)}_{j/2}(x,t,\eta)$ may have 
first order poles at regular singular points 
in view of \eqref{eq:recursion-1} and 
\eqref{eq:recursion}. However, by the same technique 
used in the proof of \cite[Proposition 3.6]{KT iwanami}
we can check that the residues of 
$S^{(\pm)}_{j/2}(x,t,\eta)~dx$ at 
regular singular points vanish for $j \ge 0$. 
Thus we obtain Table \ref{table:residues}. 
\end{proof}
Especially, we can find that the residues tabulated 
in Table \ref{table:residues} are genuine constants 
multiplied by $\eta$, which implies that the residue 
of $S_{J, \rm odd}(x,t,\eta)~dx$ only come form 
the top term $\eta \hspace{+.1em} S_{-1}(x,t)~dx$: 
\begin{equation} \label{eq:resSodd=res-1}
\Res_{x=p} S_{J, \rm odd}(x,t,\eta)~dx = 
\eta \hspace{+.1em} \Res_{x=p} S_{-1}(x,t)~dx. 
\end{equation}
This fact will make our construction of 
transformations of Painlev\'e transcendents easy. 

\subsection{Local transformation near the double turning point}
\label{section:zJ-and-sJ}

In the theory of (exact) WKB analysis, zeros of $Q_{J,0}(x,t)$ 
play important roles. They are called {\em turning points} 
of $\SLJ$ (see Definition \ref{def:turning-points} below).
In view of \eqref{eq:Q0}, the point $x=\lambda_0(t_{\ast})$ 
is a {\em double} turning point 
(i.e., a double zero of $Q_{J,0}(x,t_{\ast})$)
when $t_{\ast}$ is a generic point. 
This double turning point is particularly important 
in the WKB analysis of Painlev\'e transcendents.

Let us fix a generic point $t_{\ast}$ and take 
a sufficiently small neighborhood $V$ of $t_{\ast}$
such that $x=\lambda_0(t)$ is a double zero of 
$Q_{J,0}(x,t)$ at any point $t \in V$. 
It is shown in \cite{KT98}
that the isomonodromy system 
$\SLJ$ and $\DJ$ can be reduced to the system 
\begin{eqnarray*}
(Can) & : & \left(\frac{\p^2}{\p z^2}-
\eta^{2}Q_{\rm can}(z,s,\eta)\right)\varphi(z,s,\eta)=0 \\[+.3em]
(D_{\rm can}) & : & \frac{\p}{\p s}\varphi(z,x,\eta) = 
\left( A_{\rm can}(z,s,\eta) \frac{\p}{\p z} - 
\frac{\p A_{\rm can}}{\p z}(z,s,\eta) \right)
\varphi(z,s,\eta)
\end{eqnarray*}
on $U_0 \times V$, where $U_0$ is a neighborhood 
of the double turning point $x=\lambda_0(t)$. 
Here $Q_{\rm can}$ and $A_{\rm can}$ 
are given by
\begin{eqnarray}\label{eq:Qcan}
Q_{\rm can}(z,s,\eta) & = & 4z^2 + \eta^{-1} E(s,\eta) \\
& & + \eta^{-1/2}\frac{\eta^{-1}\rho(s,\eta)}
{z-\eta^{-1/2}\sigma(s,\eta)} + \eta^{-2}
\frac{3}{4(z-\eta^{-1/2}\sigma(s,\eta))^2}, \nonumber \\
A_{\rm can}(z,s,\eta) & = & 
\frac{1}{2(z-\eta^{-1/2}\sigma(s,\eta))},
\label{eq:Acan}
\end{eqnarray}
with 
\begin{equation} \label{eq:E}
E(s,\eta) = \rho(s,\eta)^2 - 4\sigma(s,\eta)^2.
\end{equation}
The system $(Can)$ and $(D_{\rm can})$ is compatible if 
$\rho$ and $\sigma$ satisfy the Hamiltonian system
\[
(H_{\rm can}) : \frac{d \rho}{ds} = - 4\eta\sigma, \quad
\frac{d \sigma}{ds} = - \eta\rho.
\]
As a solution of $(H_{\rm can})$, we take
\begin{eqnarray} \label{eq:sol-of-Hcan}
\sigma(s,\eta;A,B) = A e^{2\eta s}+
B e^{-2\eta s},  \quad
\rho(s,\eta;A,B) = -2A e^{2\eta s}+
2B e^{-2\eta s}, \hspace{-2.em}
\end{eqnarray} 
where $A$ and $B$ are complex constants, and 
\eqref{eq:E} becomes independent of $s$:
\begin{equation} \label{eq:E-2}
E(s,\eta;A,B) = 
\rho(s,\eta;A,B)^2 - 
4\sigma(s,\eta;A,B)^2 = -16 A B.
\end{equation}
Denote by $Q_{\rm can}(z,s,\eta;A,B)$ the potential
\eqref{eq:Qcan} with the solution \eqref{eq:sol-of-Hcan} 
of $(H_{\rm can})$ being substituted into $(\sigma,\rho)$ 
in its expression. Then, the precise statement of the local 
reduction theorem of \cite{KT98} is stated as follows. 
\begin{thm}[{\cite[Theorem 2.1, Lemma 3.3]{KT98} 
(cf.~\cite[Theorem 3.1]{AKT96})}]
\label{thm:transformation-at-double-turning-point}
Let $t_{\ast}$ be a generic point as above. 
Then, there exist a neighborhood $U_0 \times V$ of the 
point $(\lambda_0(t_{\ast}),t_{\ast})$ and a formal series
\begin{eqnarray}
z_J(x,t,\eta) & = & \sum_{j=0}^{\infty} \eta^{-j/2} 
z_{j/2}(x,t,\eta), \\
s_J(t,\eta) & = & \sum_{j=0}^{\infty} \eta^{-j/2} 
s_{j/2}(t,\eta), \\
A_J(\eta) & = & 
\sum_{n=0}^{\infty} \eta^{-n} A_n, \quad
B_J(\eta) = 
\sum_{n=0}^{\infty} \eta^{-n} B_n, 
\end{eqnarray}
satisfying the following conditions.
\begin{enumerate}[\upshape (i)]
\item For each $j \ge 0$, 
$z_{j/2}(x,t,\eta)$ and $s_{j/2}(t,\eta)$ are 
holomorphic functions in $(x,t) \in U_0 \times V$ 
and in $t \in V$, respectively. %
\item 
For each $n \ge 0$, 
$A_n$ and $B_n$ are genuine constants.
\item $z_0(x,t)$ is free from $\eta$, $(\p z_0/\p x)$ never 
vanishes on $U_0 \times V$, and $z_0(\lambda_0(t),t) = 0$. %
\item $s_0(t)$ is also free from $\eta$ and $\p s_0/d t$ 
never vanishes on $V$. %
\item $z_{1/2}(x,t)$ and $s_{1/2}(t)$ vanish identically.
\item The $\eta$-dependence of 
$z_{j/2}(x,t,\eta)$ and $s_{j/2}(t,\eta)$ ($j \ge 2$)
is only through instanton terms $\exp(\ell \Phi_J(t,\eta))$ 
for $\ell = j-2-2j'$ with $0 \le j' \le j-2$ that appear 
in the 2-parameter solution $\lambda(t,\eta;\alpha,\beta)$
of $\PJ$. Thus $z_J(x,t,\eta)$ and $s_J(t,\eta)$ have the 
property of alternating parity. %
\item The following equality holds.
\begin{eqnarray}
Q_{J}(x,t,\eta)  & = &
\left(\frac{\p z_J(x,t,\eta)}{\p x}\right)^2
Q_{\rm can}(z_J(x,t,\eta),s_J(t,\eta),\eta;A_J(\eta),B_J(\eta)) 
\hspace{-3.em}  \\
&  & - \frac{1}{2}\eta^{-2}\{ 
z_J(x,t,\eta);x\}, \nonumber
\end{eqnarray}
where $\{z_J(x,t,\eta);x\}$ denotes 
the Schwarzian derivative:
\begin{equation} \label{eq:Schwarzian-derivative}
\hspace{-2.em}
\{z_J(x,t,\eta);x\} = 
\left(\frac{\p^3 z_J(x,t,\eta)}{\p x^3} {\bigg \slash} 
\frac{\p z_J(x,t,\eta)}{\p x} \right) - \frac{3}{2}
\left(\frac{\p^2 z_J(x,t,\eta)}{\p x^2} {\bigg \slash} 
\frac{\p z_J(x,t,\eta)}{\p x} \right)^2.
\end{equation}
\end{enumerate}
\end{thm}

The proof of \cite{KT98} also tells us that 
the formal series appearing in Theorem 
\ref{thm:transformation-at-double-turning-point} 
are determined by the following process. First, the formal 
series $z_J(x,t,\eta)$ is fixed by \cite[Theorem 3.1]{AKT96}. 
Especially, the top term $z_0(x,t)$ is given 
with a suitable choice of the square root as follows:
\begin{equation} \label{eq:zJ0}
z_0(x,t) = \left[ \int_{\lambda_0(t)}^x 
\sqrt{Q_{J,0}(x,t)}~dx \right]^{1/2}. 
\end{equation}
Next, in view of \eqref{eq:E-2}, 
we find formal power series $A_J(\eta)$ and $B_J(\eta)$ 
(which are not unique) satisfy
\begin{equation} \label{eq:correspondemce1}
-16 A_J(\eta) B_J(\eta) = E_{J}(\eta).
\end{equation}
Fixing $(A_J(\eta),B_J(\eta))$ thus found, we can 
find the formal series $s_J(t,\eta)$ so that 
\begin{eqnarray} \label{eq:correspondemce2}
\sigma(s_J(t,\eta),\eta;A_J(\eta),B_J(\eta)) = \eta^{1/2} 
z_J(\lambda_J(t,\eta;\alpha,\beta),t,\eta)
\end{eqnarray}
holds. Here $\lambda_J$ is the 2-parameter solution 
of $\PJ$ substituted into the coefficients of $\SLJ$ 
and $\DJ$. The top term $s_0(t)$ in $s_J(t,\eta)$ 
is given by
\begin{equation} \label{eq:sJ0}
s_0(t) = \frac{1}{2} \phi_J(t) = 
\frac{1}{2} \int^t \sqrt{F^{(1)}_J(t)}~dt.
\end{equation}
Then the set of formal series 
$(z_J(x,t,\eta),s_J(t,\eta),A_J(\eta),B_J(\eta))$
satisfies the conditions in Theorem 
\ref{thm:transformation-at-double-turning-point}.
Note that there is an ambiguity in the above choice 
of formal series; if a set of formal series 
\[
\bigl( z_J(x,t,\eta),s_J(t,\eta),A_J(\eta),B_J(\eta) \bigr)
\] 
satisfies the conditions in Theorem
\ref{thm:transformation-at-double-turning-point}, then 
\begin{equation} \label{eq:ambiguity}
\bigl( z_J(x,t,\eta),s_J(t,\eta) + G(\eta), 
A_J(\eta) \exp\bigl( -2\eta \hspace{+.1em} 
G(\eta) \bigr), 
B_J(\eta) \exp\bigl( 2\eta \hspace{+.1em} 
G(\eta) \bigr) \bigr)
\end{equation}
also satisfies the same conditions. Here 
\begin{equation} \label{eq:free-parameter-G}
G(\eta) = \sum_{n=1}^{\infty} \eta^{-n} G_n
\end{equation}
is an arbitrary formal power series with constant 
coefficients $G_n$. Here we have assumed that 
the formal power 
series \eqref{eq:free-parameter-G} has no constant term $G_0$. 
If we allow the constant term $G_0 \ne 0$, then 
$A_J(\eta) \exp\bigl( -2\eta \hspace{+.1em} G(\eta) \bigr)$ 
is no longer formal power series in $\eta^{-1}$, 
and hence we set $G_0 = 0$. 
The existence of this ambiguity 
corresponds to the fact that the relation between parameters 
$(\alpha,\beta)$ and $(A,B)$ is given by essentially one 
relation, i.e., \eqref{eq:correspondemce1}. 

We will regard the coefficients $G_n$ in 
\eqref{eq:free-parameter-G} as free parameters. As is clear 
from \eqref{eq:ambiguity}, such free parameters are contained 
in the transformation series $s_J(t,\eta)$ additively, and 
it is shown in \cite[Proposition 3.2]{KT98} that the formal 
series $s_J(t,\eta)$ is unique up to these additive free 
parameters (see \cite[Remark 3.3]{KT98}). Once the free 
parameters $G_n$ are fixed, then the transformation 
from $\SLJ$ and $\DJ$ to $(Can)$ and $D_{\rm can}$ 
is fixed, and hence the correspondence between 
the solutions of $(H_J)$ and $(H_{\rm can})$ is also fixed. 
These free parameters 
will be fixed when we discuss the transformation theory 
between Painlev\'e transcendents in Section 
\ref{section:main-results} 
and \ref{section:main-theorem2}. 

\section{Stokes geometries of Painlev\'e equations 
and isomonodromy systems} \label{section:Stokes-geometry}

In \cite{KT96}, \cite{KT98} etc. the relationship between 
$P$-turning points, $P$-Stokes curves of $\PJ$ and 
turning points, Stokes curves of $\SLJ$
plays an important role in the
construction of WKB theoretic 
transformations. In this section we review 
these geometric properties of Stokes geometries 
of $\PJ$ and $\SLJ$.

\subsection{$P$-Stokes geometry of $(P_J)$}

First, we review the definition of $P$-turning points 
and $P$-Stokes curves of $\PJ$ introduced by Kawai and Takei. 
Here we recall that ${\rm Sing}_J$ is the set of singular 
points of $\PJ$ defined in \eqref{eq:singular-points}.
\begin{defin}[{\cite[Definition 2.1]{KT96}}]
Let $\lambda_J = \lambda_J(t,\eta;\alpha,\beta)$ 
be a 2-parameter solution of $\PJ$ and $\lambda_0(t)$ 
be its top term.
\begin{itemize}
\item A point $t=r \notin {\rm Sing}_J$ 
is said to be a {\em $P$-turning point} 
of $\lambda_J$ if
\begin{equation}
F_J^{(1)}(r) = 0,
\end{equation} 
where $F_J^{(1)}(t)$ is defined by \eqref{eq:F1}.
\item A $P$-turning point $t=r$ of $\lambda_J$
is called {\em simple} if 
\begin{equation}
\frac{\p^2 F_J}{\p \lambda^2}(\lambda_0(r),r) \ne 0.
\end{equation} %
\item For a $P$-turning point $t=r$ of $\lambda_J$, 
a {\em $P$-Stokes curve} of $\lambda_J$ 
(emanating from $t=r$) is an integral curve defined by 
\begin{equation}
{\rm Im}\int_{r}^{t}\sqrt{F^{(1)}_J(t)}~dt = 0.
\end{equation} %
\end{itemize}
\end{defin}

$P$-turning points and $P$-Stokes curves of 
$\lambda_J$ are defined 
in terms of only the top term $\lambda_0(t)$ 
of the 2-parameter solution in question. 
Although they are defined for a fixed branch of the 
algebraic function $\lambda_0(t)$, we may regard them 
as objects on the Riemann surface of $\lambda_0(t)$. 
By ``a $P$-turning point 
(resp., a $P$-Stokes curve)" we may mean ``a $P$-turning point 
(resp., a $P$-Stokes curve) of some 2-parameter solution 
$\lambda_J$", simply. 
Note also that $P$-turning points and $P$-Stokes curves 
are nothing but zeros and {\em horizontal trajectories} 
(see \cite{Strebel}) of the quadratic differential 
$F^{(1)}(t)~dt^2$ defined on the Riemann surface of 
$\lambda_0(t)$.

As is pointed out by \cite{Wakako-Takei} and \cite{Takei10}, 
a point $s \in {\rm Sing}_J$ contained in the following list 
may play a role similar to $P$-turning points: 
\begin{itemize}
\item $s=0$ for $(P_{{\rm III'}(D_6)})$, $(P_{{\rm III'}(D_7)})$, 
$(P_{{\rm III'}(D_8)})$, $(P_{{\rm V}})$ and $(P_{\rm VI})$,
\item $s=1$ for $(P_{\rm VI})$, 
\item $s=\infty$ for $(P_{\rm VI})$.
\end{itemize}
At a singular point $s$ in the above list, there exists 
{\em a simple-pole type} 2-parameter solution; 
that is, the top term $\lambda_0(t)$ of a 2-parameter 
solution has a branch point at $t=s$ satisfing 
\begin{equation} \label{eq:simple-pole}
F_J^{(1)}(t) = O\bigl( (t-s)^{-3/2} \bigr) \quad
\text{as $t \rightarrow s$}.
\end{equation}
Note that the condition \eqref{eq:simple-pole} guarantees 
that the corresponding quadratic differential 
$F^{(1)}(t)~dt^2$ has a simple pole type singularity 
at $t=s$ after taking a new independent variable 
$T=(t-s)^{1/2}$, which is a local parameter of the 
Riemann surface of $\lambda_0(t)$ near $t=s$.  
On the Riemann surface of $\lambda_0(t)$ we distinguish 
such singular points from usual singular points, 
and call them {\em $P$-turning points of simple-pole type}. 
A $P$-turning point of simple-pole type is denoted 
by $r_{sp}$. A $P$-Stokes curve emanating from $r_{sp}$ 
is also defined by 
\begin{equation}
{\rm Im} \int_{r_{sp}}^t \sqrt{F_J^{(1)}(t)}~dt = 0.
\end{equation}

By the {\em $P$-Stokes geometry} (of $\PJ$) we mean the 
configuration of $P$-turning points, $P$-turning points 
of simple-pole type, singular points and $P$-Stokes curves
(of $\PJ$). Figure \ref{fig:P-Stokes-curves-of-PJ} depicts 
examples of $P$-Stokes geometries. Five $P$-Stokes curves 
emanate from each simple $P$-turning point. 
Figure \ref{fig:P-Stokes-curves-of-PJ} (b) shows  
an example of $P_{{\rm III'}(D_6)}$ which has a 
$P$-turning point of simple-pole type at the origin, 
and one $P$-Stokes curve 
emanates from the $P$-turning point of simple-pole type. 
Since $\lambda_0(t)$ is a multi-valued function 
of $t$, $P$-Stokes curves intersect each other, as observed 
in the figures. Such ``apparent" intersections are resolved 
if we take a lift of $P$-Stokes curves onto the Riemann 
surface of $\lambda_0(t)$ (see Section 
\ref{section:degeneration-of-PStokes-geometry} below).

\begin{rem}
$P$-Stokes curves are used to describe the 
criterion of {\em Borel summability} 
of 0-parameter solutions (i.e., formal power 
series solutions of the form 
$\lambda(t,\eta)=\sum_{n=0}^{\infty}\eta^{-n}\lambda_n(t)$) 
of $\PJ$ by \cite{Kamimoto-Koike}. 
It is known that certain {\em non-linear} Stokes phenomena 
occur to such a formal solution of $\PJ$ on $P$-Stokes curves. 
Takei discussed such Stokes phenomena for $(P_{\rm I})$ 
in \cite{Takei95}. Moreover, it is also expected that 
non-linear Stokes phenomena also occur to the 
2-parameter solutions (see \cite{Takei00}). 
\end{rem}

  \begin{figure}[h]
  \begin{minipage}{0.25\hsize}
  \begin{center}
  \includegraphics[width=35mm]{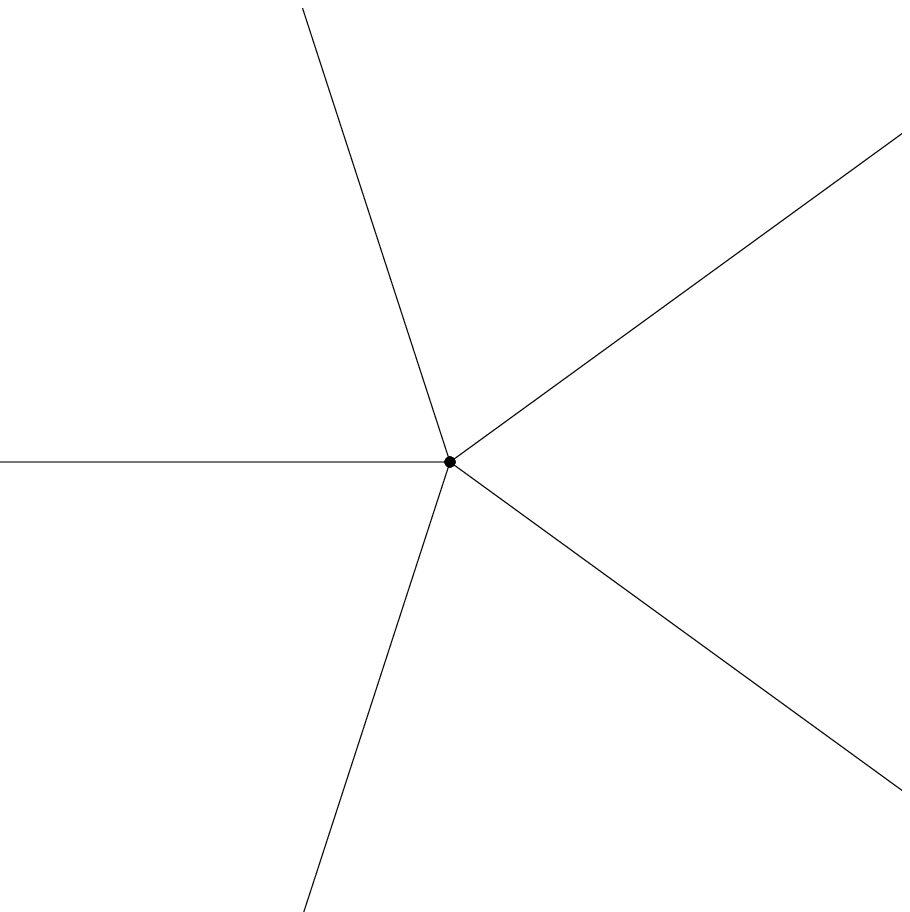} \\
  {(a): $(P_{\rm I})$.}
  \end{center}
  \end{minipage} \hspace{+.3em}
  \begin{minipage}{0.35\hsize}
  \begin{center}
  \includegraphics[width=35mm]{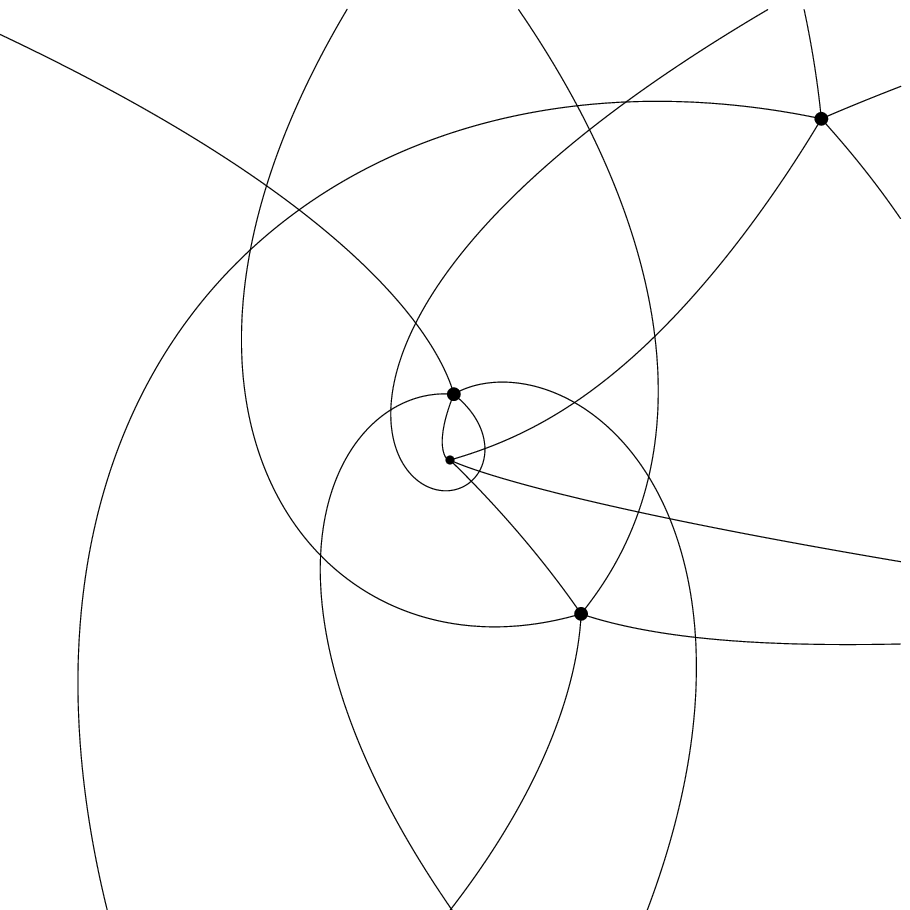} \\
  {(b): $(P_{{\rm III'}(D_6)})$ with 
  $(c_{0},c_{\infty}) = (3,2+i)$}
  \end{center} 
  \end{minipage} \hspace{+.5em}
  \begin{minipage}{0.35\hsize}
  \begin{center}
  \includegraphics[width=38mm]{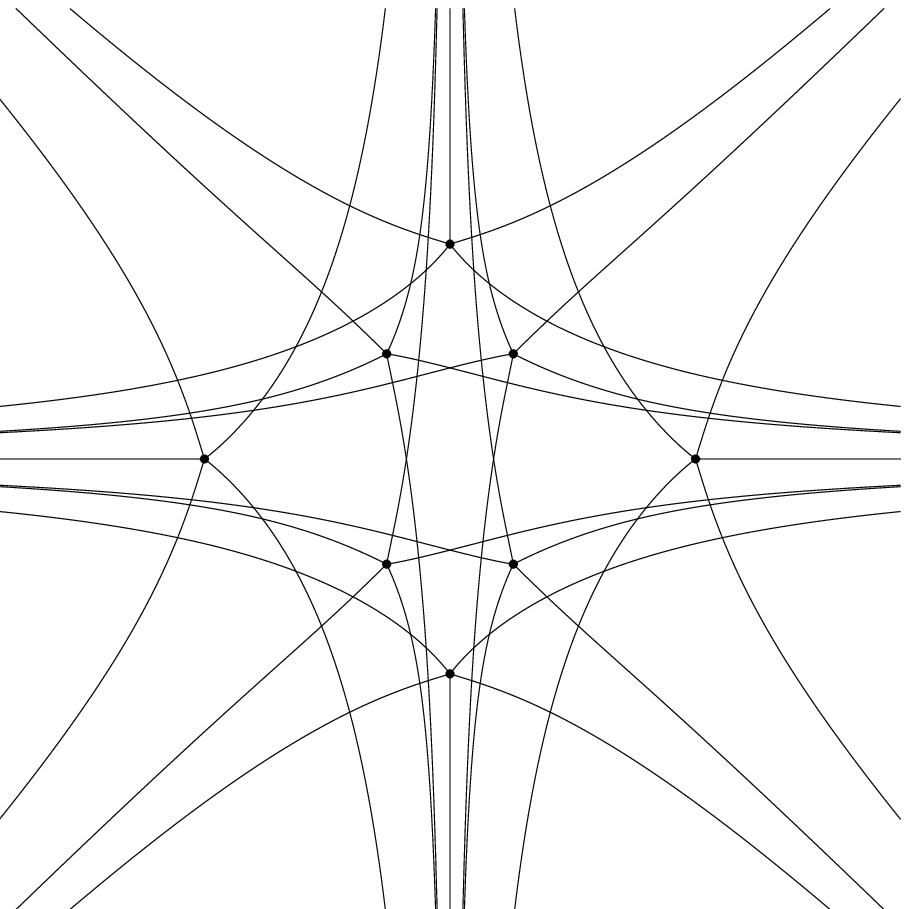} \\[+.3em]
  {(c): $(P_{\rm IV})$ with $(c_0,c_{\infty})=(1,2)$.}
  \end{center}
  \end{minipage} 
  \caption{Examples of $P$-Stokes geometries.}
  \label{fig:P-Stokes-curves-of-PJ}
  \end{figure}

\subsection{Stokes geometry of $\SLJ$}

Next, we recall the definition of turning points and 
Stokes curves for the linear differential equation $\SLJ$,
and explain their relationship with the $P$-Stokes geometry 
defined in the previous subsection. Recall that, we consider 
the situation that a 2-parameter solution 
$(\lambda_J,\nu_J) = $
$ (\lambda_J(t,\eta;\alpha,\beta),
\nu_J(t,\eta;\alpha,\beta))$ of $(H_J)$ is substituted 
into $(\lambda, \nu)$ which appears in the coefficients of 
$\SLJ$ and $\DJ$, as explained in Section \ref{section:SLJ}. 
Here we assume that the 2-parameter solution is defined 
in a neighborhood $V$ of a point $t_{\ast} \in \Omega_J$, 
and the branch of $\lambda_0(t)$, which is the top term of
$\lambda_J$, is fixed on $V$. 

\begin{defin}[{\cite[Definition 2.4 and 2.6]
{KT iwanami}}] \label{def:turning-points}
Fix a point $t$ contained in $V$.
\begin{itemize}
\item A point $x=a(t)$ is called a {\em turning point} of 
$(SL_{J})$ (at $t$) if it is a zero of $Q_{J,0}(x,t)$. %
\item A {\em Stokes curve} of $(SL_{J})$ is an integral curve 
emanating from a turning point $x=a(t)$ defined by 
\begin{equation}
{\rm Im}\int_{a(t)}^x \sqrt{{Q_{J,0}}(x,t)}~dx=0. 
\end{equation} %
\end{itemize}
\end{defin}

\begin{rem} \label{rem:Stokes-geometry-at-t}
Note that, locations of turning points and 
Stokes curves for $\SLJ$ depend on $t$. 
More precisely, they depend also on the branch 
of $\lambda_0$ at $t$, which is the 
top term of 2-parameter solution substituted. 
Therefore, by ``turning points (resp., Stokes curves)
of $\SLJ$ at $t \in V$" we mean 
``turning points (resp., Stokes curves) of $\SLJ$ 
at $t$ with the fixed branch of $\lambda_0$ on $V$". 
\end{rem}

Turning points and Stokes curves of $\SLJ$ are 
nothing but zeros and horizontal trajectories 
of the quadratic differential $Q_{J,0}(x,t)~dx^2$. 
We say that a turning point is of {\em order} $m$ if 
it is a zero of $Q_{J,0}$ of order $m$. Especially, 
turning points of order 1 and 2 are called 
{\em simple} and {\em double} turning points, 
respectively. In view of \eqref{eq:Q0}, 
in a generic situation $(SL_{J})$ has 
a double turning point at $x=\lambda_0(t)$ 
and one simple turning point 
(resp., two simple turning points)
when $J={\rm I}$, ${\III}$ and ${\rm III'}(D_8)$ 
(resp., $J = {\rm II}, {\rm III'}(D_6), 
{\rm IV}, {\rm V}$ and ${\rm VI}$). 
In the case of a linear equation, $(m+2)$ Stokes curves 
emanate from a turning point of order $m$ ($m \ge 1$). 
By the {\em Stokes geometry} (of $\SLJ$) we mean the 
configuration of turning points, singular points and 
Stokes curves (for a fixed $t$). 
Actually, if $Q_{J,0}(x,t)$ has simple poles, we need to 
regard them as turning points 
similarly to $P$-turning points of simple-pole type 
of $\PJ$ (see \cite{Koike00}). 
However, in view of \eqref{eq:Q0}, such a simple pole 
does not appear in a generic situation, and we will 
only consider situations where a simple pole never 
appears in the Stokes geometry of $\SLJ$.

\begin{figure}[h]
\begin{center}
\begin{pspicture}(6,0)(10,3.5)
%
\psset{fillstyle=none}
\rput[c]{0}(8.0,0.2){$P$-Stokes curves of $(P_{\rm I})$.} %
\rput[c]{0}(6.7,2.0){$\times$} \rput[c]{0}(6.3,2.0){$t_{2}$}
\rput[c]{0}(6.7,2.5){$\times$} \rput[c]{0}(6.3,2.5){$t_{1}$}
\rput[c]{0}(6.7,1.5){$\times$} \rput[c]{0}(6.3,1.5){$t_{3}$}
\rput[c]{0}(8.8,-2.5){$\lambda_0(t)$} 
\rput[c]{0}(6.45,-2.1){$-2\lambda_0(t)$}
\psset{linewidth=1pt}
\psline(8,2)(6.5,2)
\psline(8,2)(7.5,0.5)
\psline(8,2)(7.5,3.5)
\psline(8,2)(9.5,1.0)
\psline(8,2)(9.5,3.0)
%
\end{pspicture}
\end{center}
  \begin{minipage}{0.31\hsize}
  \begin{center}
  \includegraphics[width=37mm]
  {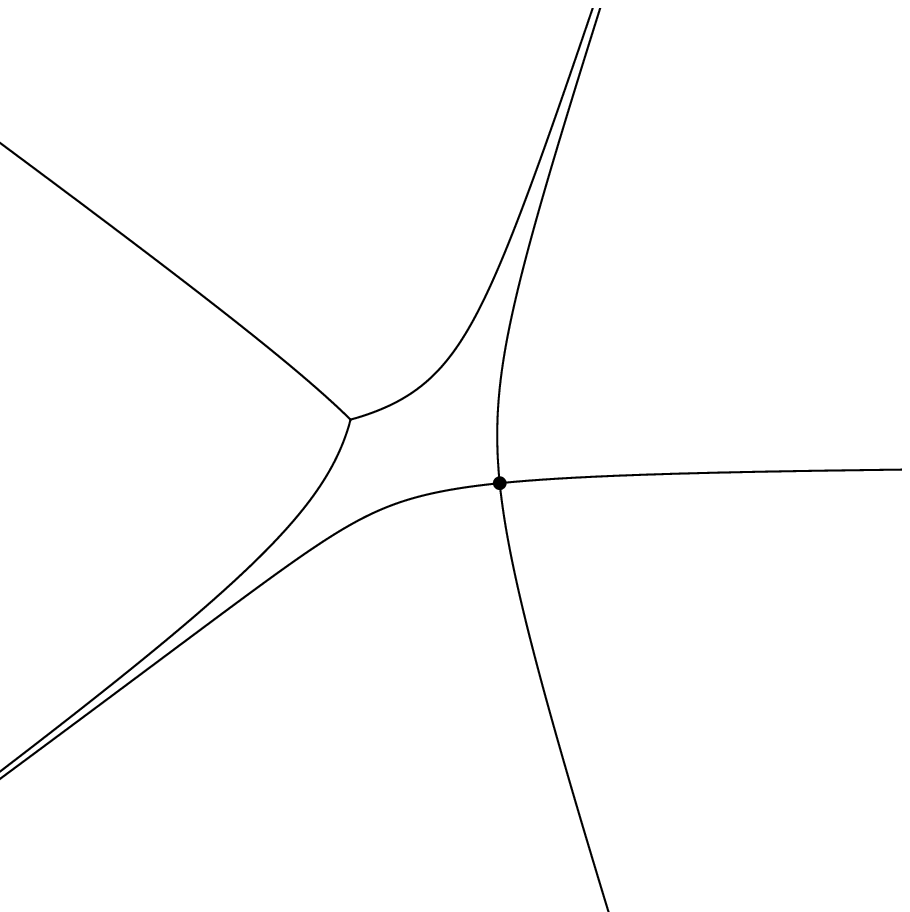} \\
  {(a): At $t=t_1$. \\~}
  \end{center}
  \end{minipage} \hspace{+.8em}
  \begin{minipage}{0.31\hsize}
  \begin{center}
  \includegraphics[width=37mm]
  {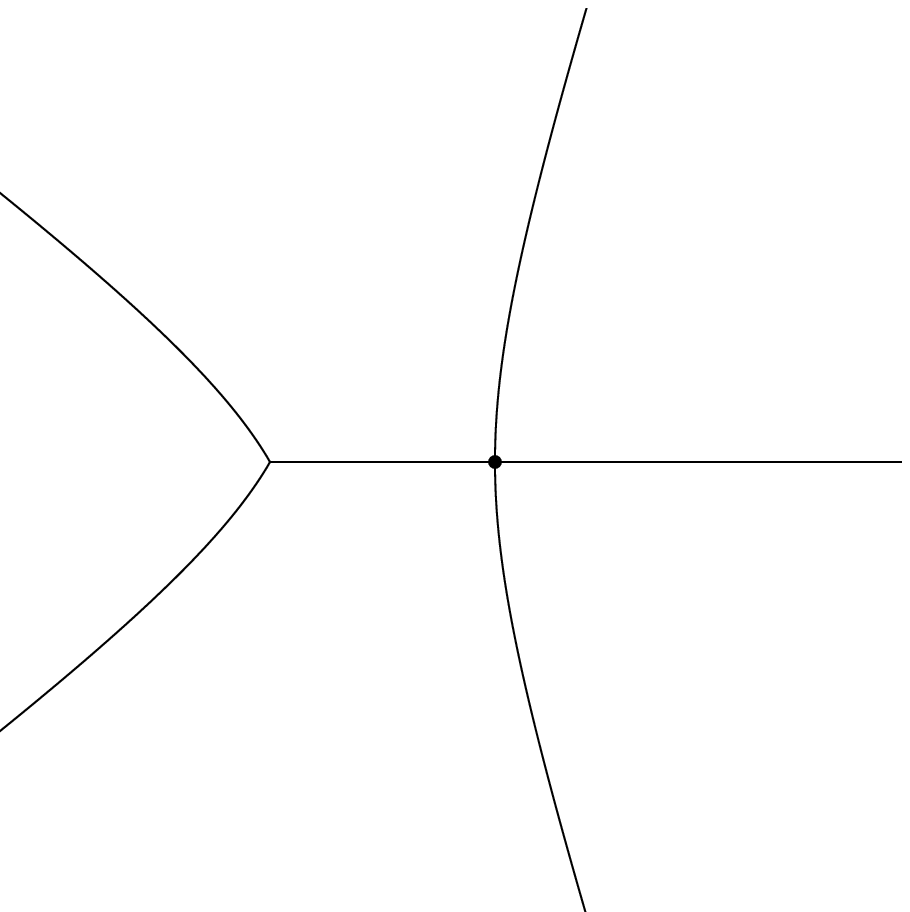} \\
  {(b): At $t=t_2$ \\(on a $P$-Stokes curve).}
  \end{center}
  \end{minipage} \hspace{+.5em}
  \begin{minipage}{0.31\hsize}
  \begin{center}
  \includegraphics[width=37mm]
  {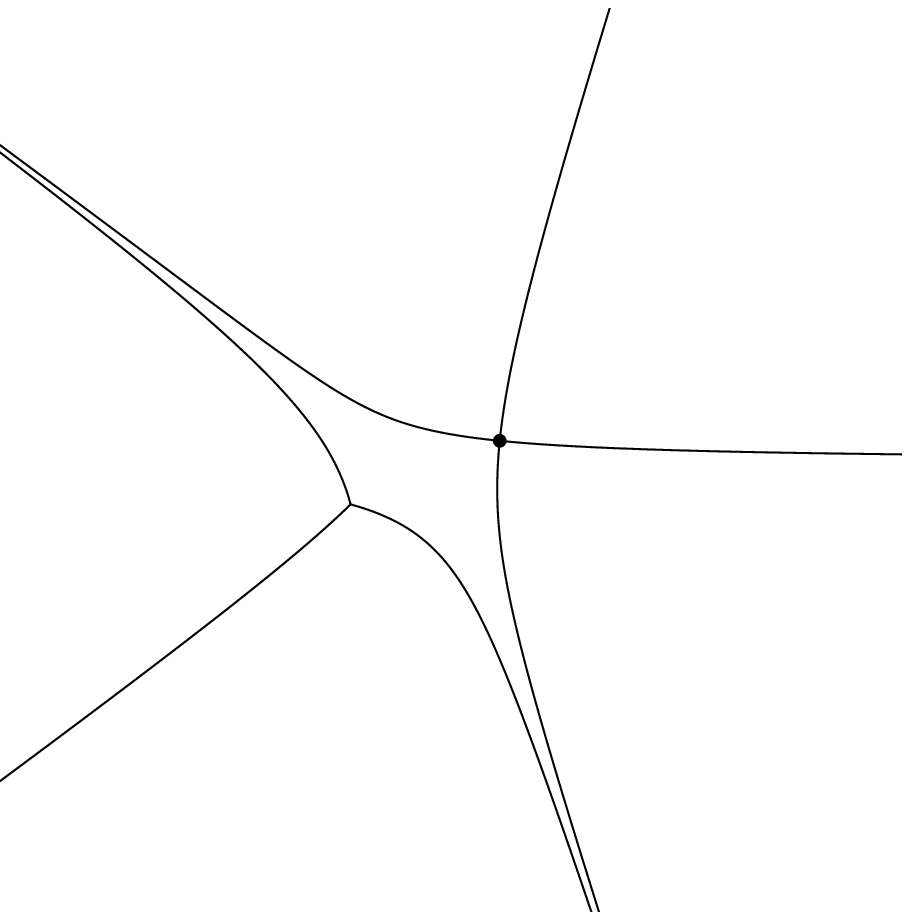} \\
  {(c): At $t=t_3$. \\~}
  \end{center} 
  \end{minipage} \vspace{-.2em}
  \caption{Stokes curves of $(SL_{\rm I})$ (for several $t$).}
  \label{fig:Stokes-curves-of-SLJ}
  \end{figure}

Figure \ref{fig:Stokes-curves-of-SLJ} depicts examples of 
Stokes curves of $(SL_{\rm I})$ for several $t$. 
Here $t_1$ and $t_3$ are some points which do not lie on 
a $P$-Stokes curve of $(P_{\rm I})$, while $t_2$ lies on 
a $P$-Stokes curve $(P_{\rm I})$. $(SL_{\rm I})$ has a 
double turning point at $x=\lambda_0(t)$ and a simple turning 
point at $x=-2\lambda_0(t)$ when $t \ne 0$. 
Note that, since $\lambda_0(t) = \sqrt{-t/6}$ 
for $(P_{\rm I})$, these two turning points merge 
as $t$ tends to the $P$-turning point $t=0$.
We can observe that a Stokes curve of $(SL_{\rm I})$ 
connects these two turning point 
$x=\lambda_0(t)$ and $-2\lambda_0(t)$ when $t = t_2$ 
which lies on a $P$-Stokes curve. 
We call such a Stokes curve connecting turning points 
of $\SLJ$ a {\em degenerate Stokes segment}, or 
a {\em Stokes segment} for short. 
(In the context of quadratic differentials Stokes segments 
are called {\em saddle connections}.) 

Actually, other $\PJ$ and $\SLJ$ also enjoy the same 
geometric properties as $(P_{\rm I})$ and $(SL_{\rm I})$ 
explained here. That is, $P$-turning points and 
$P$-Stokes curves for $(P_J)$ are related to turning points 
and Stokes curves for $(SL_{J})$ in the following manner.
\begin{prop}[{\cite[Proposition 2.1]{KT96}}] 
\label{prop:PStokes-and-Stokes}
\begin{enumerate}[\upshape (i)]
\item For a simple $P$-turning point $r$ (of $\lambda_{J}$), 
there exists a simple turning point $a(t)$ of $\SLJ$
that merges with the double turning point 
$x=\lambda_0(t)$ at $t=r$, 
and consequently there exists a turning point of 
order three at $t=r$ for $\SLJ$. %
\item For the simple $P$-turning point $r$ and the turning point 
$a(t)$ of $\SLJ$ as above, the following equality holds:
\begin{equation} \label{eq:integral-relation}
\int_{a(t)}^{\lambda_0(t)} \sqrt{Q_{J,0}(x,t)}~dx = 
\frac{1}{2}\int_{r}^t \sqrt{F_J^{(1)}(t)}~dt.
\end{equation}
Here the branch of square roots are chosen so that 
\begin{equation}\label{eq:branch-of-Qo-and-FJ1}
\sqrt{Q_{J,0}(x,t)}=C_J(x,t)(x-\lambda_0)\sqrt{R_J(x,t)},\quad
\sqrt{R_J(\lambda_0(t),t)} = \sqrt{F_J^{(1)}(t)}.
\end{equation}
\end{enumerate}
\end{prop}

%
%

Proposition \ref{prop:PStokes-and-Stokes} implies that, 
when $t$ lies on a $P$-Stokes curve emanating from a simple 
$P$-turning point $r$, a Stokes segment appears between 
the double turning point $\lambda_0(t)$ and the simple turning 
point $a(t)$. This relationship between $P$-Stokes curves and 
Stokes curves are essential in the 
construction of WKB theoretic transformation to 
$(P_{\rm I})$ near a simple $P$-turning point 
(see \cite{KT96} and \cite{KT98}).

Similar geometric properties are observed also
when $t$ lies on a $P$-Stokes curve emanating from a 
$P$-turning point of simple-pole type.

\begin{prop} [{\cite[Proposition 3.2 (ii)]{Takei10}}]
Suppose that $t$ lies on a $P$-Stokes curve emanating 
from a $P$-turning point of simple-pole type of $\PJ$.
Then, there exists a Stokes curve of $\SLJ$ which starts 
from $\lambda_0(t)$ and returns to $\lambda_0(t)$ after 
encircling several singular points and/or turning points
of $\SLJ$. 
\end{prop}

\subsection{Degeneration of the $P$-Stokes geometry}
\label{section:degeneration-of-PStokes-geometry}

As is explained in Introduction, we are interested 
in the {\em degenerate} situations of the $P$-Stokes geometry; 
that is, situations where 
there exist a $P$-Stokes curve which 
connects $P$-turning points or $P$-turning points of 
simple-pole type of a 2-parameter solution 
$\lambda_J$ of $\PJ$. 
We will call such special $P$-Stokes curves 
{\em degenerate $P$-Stokes segments}, or 
{\em $P$-Stokes segments} for short. 
In this section we discuss a relationship between 
such a degeneration of the $P$-Stokes geometry of $\PJ$ 
and the Stokes geometry of $\SLJ$.

Typically, there are two types of $P$-Stokes segments 
which appear for the $P$-Stokes geometry in a generic 
situation: A $P$-Stokes segment of the first type 
connects two {\em different} simple $P$-turning points, 
while a $P$-Stokes segment of the second type 
(sometimes called a {\em loop-type}) 
emanates from and returns to the {\em same} $P$-turning 
point and hence forms a closed loop. 

  \begin{figure}[h]
   \begin{center}
  \includegraphics[width=50mm]{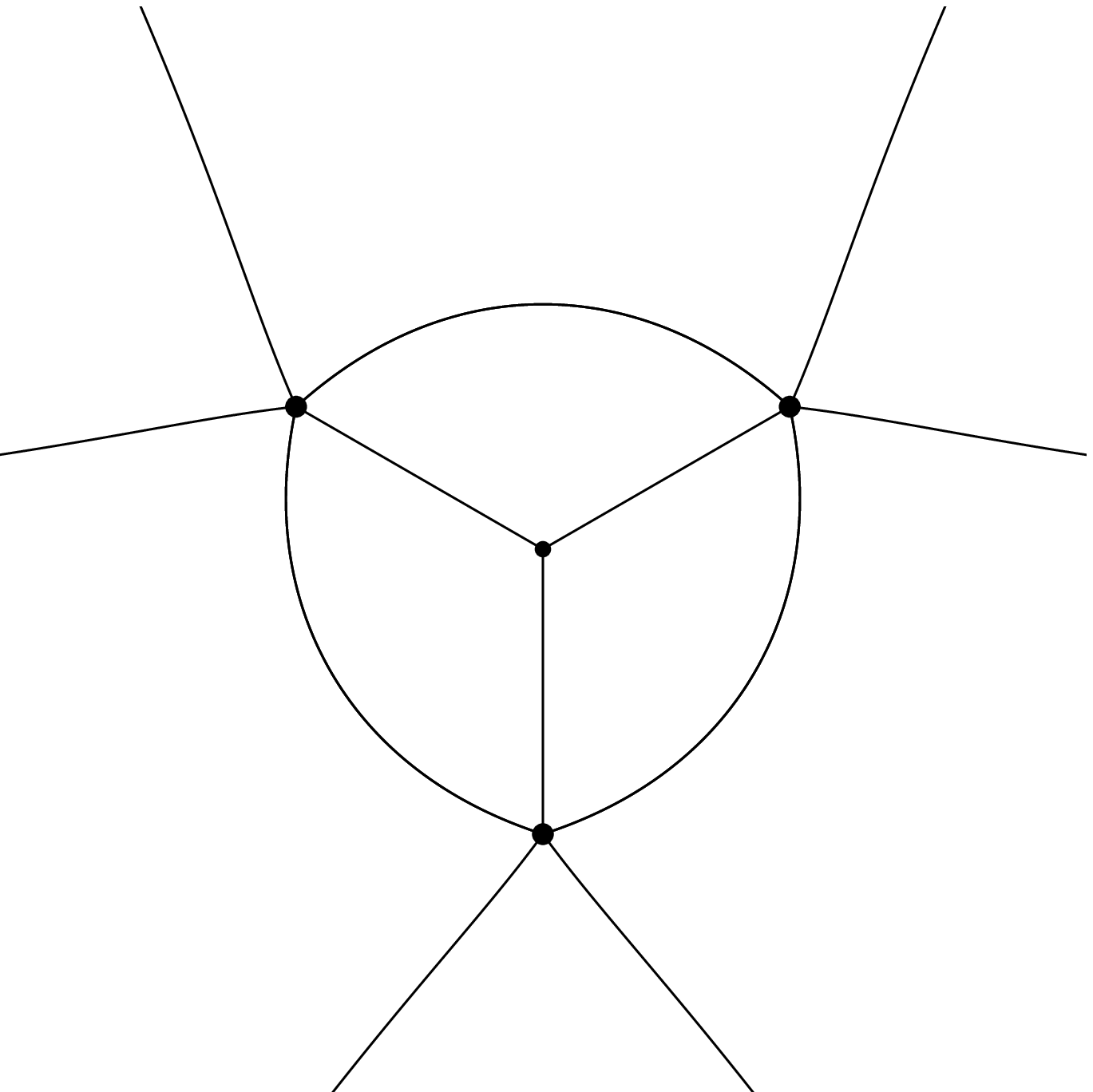} \\[-.2em]
  \end{center} 
  \begin{center}
\begin{pspicture}(0,0)(0,0)
\psset{fillstyle=none}
\rput[c]{0}(-0.5,+4.23){$\times$} 
\rput[c]{0}(-1.15,+3.0){$\times$} 
\rput[c]{0}(-0.5,+4.6){\large $u_{\ast,A}$}
\rput[c]{0}(-1.65,+3.0){\large $u_{\ast,B}$}
\rput[c]{0}(0.6,+4.5){$\Gamma_A$}
\rput[c]{0}(-1.05,+2.10){$\Gamma_B$}
\rput[c]{0}(-1.7,+4.2){\large $r$}
\rput[c]{0}(+1.7,+4.2){\large $r_{A}$}
\rput[c]{0}(0.0,+1.3){\large $r_{B}$}
\rput[c]{0}(0.3,+2.9){\large $0$}
\end{pspicture}
\end{center}
  \vspace{-2.5em}
  \caption{The $P$-Stokes geometry of $\PII$ with 
  $P$-Stokes segments (described on the $u$-plane).}
  \label{fig:P2-Stokes}
  \end{figure}

Figure \ref{fig:P2-Stokes} depicts the $P$-Stokes 
geometry of $\PII$ when $c=i$, 
and we can observe that three 
$P$-Stokes segments appear in the figure. 
Here we have introduced a new variable
\begin{equation} \label{eq:u-P2}
u=\lambda_0(t) 
\end{equation}
of the Riemann surface of $\lambda_0(t)$
and Figure \ref{fig:P2-Stokes} describes the $P$-Stokes 
curves of $\PII$ on the $u$-plane. Using the relation 
$t=-(2u^3+c)/u$, the quadratic differential which 
defines the $P$-Stokes geometry of $\PII$ is written as  
\begin{equation} \label{eq:quad-diff-P2}
F^{(1)}_{\rm II}(t)~dt^2 = {\rm quad}_{\rm II}(u,c)~du^2,~~~
{\rm quad}_{\rm II}(u,c) = \frac{(4u^3-c)^3}{u^5}
\end{equation}
in the $u$-variable. Although Figure \ref{fig:P2-Stokes}
depicts the case $c=i$, the configuration of $P$-Stokes 
geometry of $\PII$ described in the variable $u$ given 
in \eqref{eq:u-P2} for any 
$c \in \hspace{+.1em} i{\mathbb R}_{>0}$ 
(where ${\mathbb R}_{>0}$ denotes the set of positive 
real numbers) is the same as Figure \ref{fig:P2-Stokes}
since the quadratic differential \eqref{eq:quad-diff-P2} 
has the following scale invariance: 
\[
r^{-1} \sqrt{{\rm quad}_{\rm II}(r^{1/3}u,rc)}~d(r^{1/3}u) = 
\sqrt{{\rm quad}_{\rm II}(u,c)}~du~~ (r \ne 0). 
\]
Therefore, when $c \in i \hspace{+.1em} {\mathbb R}_{>0}$, 
$P$-Stokes geometry of $\PII$ has three simple 
$P$-turning points and three $P$-Stokes segments. 
The symbols $r, r_A$ and $r_B$ 
(resp., $\Gamma_A$, $\Gamma_B$) 
in Figure \ref{fig:P2-Stokes} represent the $P$-turning 
points (resp., $P$-Stokes segments) of $\PII$ 
when $c \in i \hspace{+.1em} {\mathbb R}_{>0}$. 
Furthermore, since ${\rm quad}_{\rm II}(u,c)$ is also 
invariant under $(u,c) \mapsto (-u,-c)$, the $P$-Stokes 
geometry when $c \in i \hspace{+.1em} {\mathbb R}_{<0}$ 
(where ${\mathbb R}_{<0}$ denotes the set of negative 
real numbers) is the reflection $u \mapsto -u$ of 
Figure \ref{fig:P2-Stokes}. 

  \begin{figure}[h]
  \begin{center}
\begin{pspicture}(0,0)(0,0)
\psset{fillstyle=none}
\rput[c]{0}(-3.4,-0.9){\large $\lambda_0(t)$}
\rput[c]{0}(-4.8,-2.8){\large $a(t)$}
\rput[c]{0}(-1.3,-3.3){\large $a_{A}(t)$}
\rput[c]{0}(2.0,-2.7){\large $\lambda_0(t)$}
\rput[c]{0}(3.0,-1.0){\large $a(t)$}
\rput[c]{0}(4.8,-3.3){\large $a_{B}(t)$}
\end{pspicture}
\end{center}
  \begin{minipage}{0.45\hsize}
  \begin{center}
  \includegraphics[width=45mm]
  {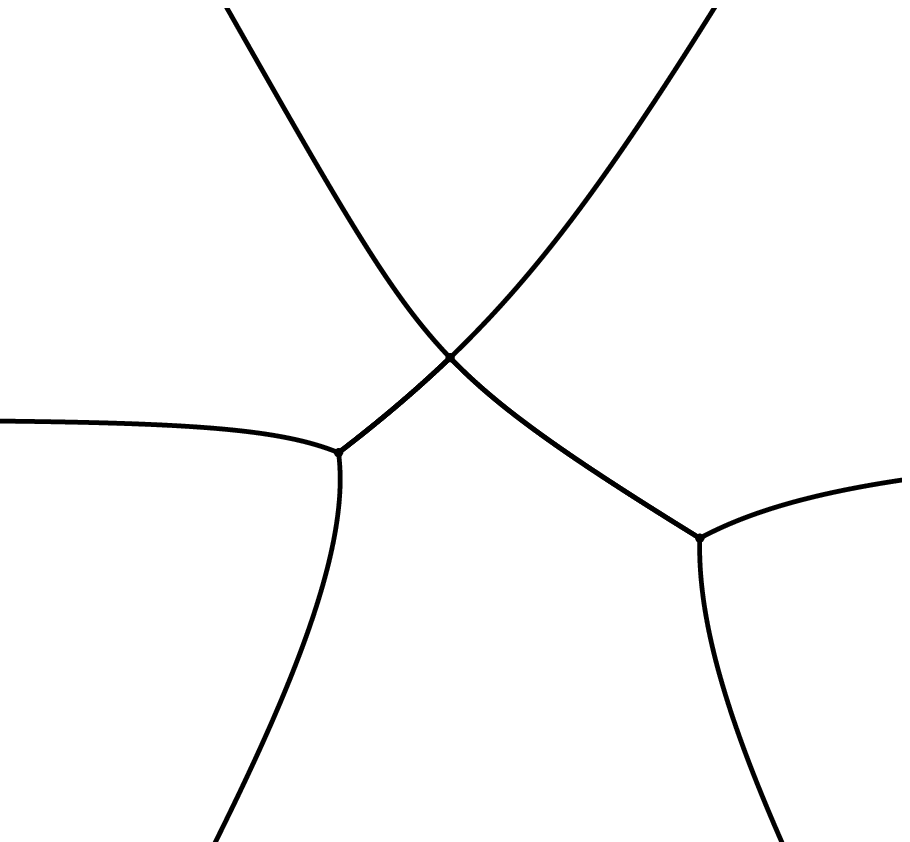} \\ 
  {($SL_{\rm II}$-$A$): The Stokes geometry of $\SLII$ 
  corresponding to $u_{\ast,A}$.}
  \end{center}
  \end{minipage} \hspace{+1.5em}
  \begin{minipage}{0.45\hsize}
  \begin{center}
  \includegraphics[width=45mm]
  {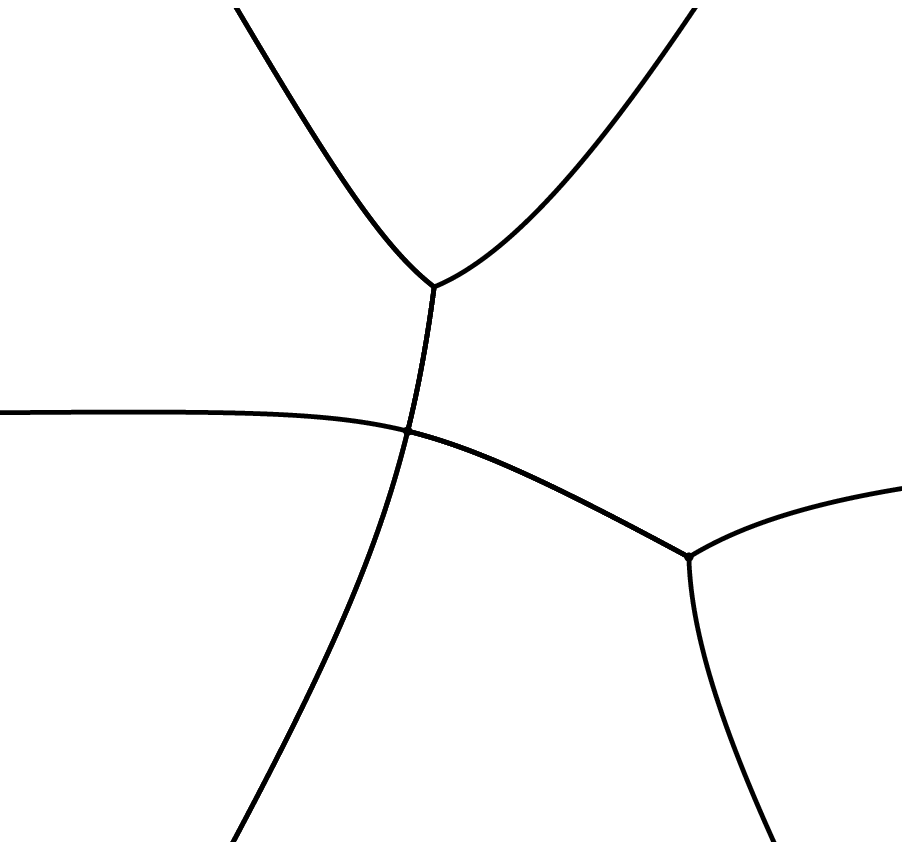} \\ 
  {($SL_{\rm II}$-$B$): The Stokes geometry of $\SLII$ 
  corresponding to $u_{\ast,B}$.}
  \end{center} 
  \end{minipage} 
  \caption{The Stokes geometries of $\SLII$ 
  on $P$-Stokes segments.}
  \label{fig:SL2-Stokes}
  \end{figure}

Figure \ref{fig:SL2-Stokes} ($SL_{\rm II}$-$A$)
(resp., ($SL_{\rm II}$-$B$))
depicts the Stokes geometry of $\SLII$ when we fix $t$ 
at a point $t_{\ast,A}$ (resp., $t_{\ast,B}$) 
corresponding to a point $u_{\ast,A}$ (resp., $u_{\ast,B}$)
which lies on the $P$-Stokes segment $\Gamma_A$ 
(resp., $\Gamma_B$) in Figure \ref{fig:P2-Stokes}.
Note that $u$ assigns a point $t$ on the $t$-plane together 
with a branch of $\lambda_0$ at $t$, and the Stokes 
geometries shown in Figure \ref{fig:SL2-Stokes}
are drawn for the the branch of $\lambda_0$ assigned 
by $u_{\ast,A}$ and $u_{\ast,B}$, 
respectively (see Remark \ref{rem:Stokes-geometry-at-t}).
In both cases of Figure \ref{fig:SL2-Stokes} 
($SL_{\rm II}$-$A$) and ($SL_{\rm II}$-$B$), 
there are two Stokes segments 
in the Stokes geometry of $\SLII$ each of which connects 
the double turning point $x=\lambda_0(t)$ and a simple turning 
point. Here, $a(t)$, $a_A(t)$ and $a_B(t)$ are 
the simple turning points of $\SLII$ which 
merge with $\lambda_0(t)$ at the $P$-turning point 
$r$, $r_A$ and $r_B$, respectively
(cf.~Proposition \ref{prop:PStokes-and-Stokes} (i)).
Here $a_A(t)$ and $a_B(t)$ merge $\lambda_0$ when $t$ 
tends to $r_A$ and $r_B$ along the $P$-Stokes segment
$\Gamma_A$ or $\Gamma_B$, respectively. 

  \begin{figure}[h]
  \begin{center}
\begin{pspicture}(0,0)(0,0)
\psset{fillstyle=none}
\rput[c]{0}(0.32,-1.5){$\times$} 
\rput[c]{0}(-0.32,-1.5){$\times$} 
\rput[c]{0}(1.0,-1.5){\large $u_{\ast, A}$}
\rput[c]{0}(-1.0,-1.5){\large $u_{\ast, B}$}
\rput[c]{0}(0.0,-0.45){\large $\Gamma$}
\rput[c]{0}(-0.3,-2.6){\large $r$}
\rput[c]{0}(-0.4,-4.0){\large $r_{sp}$}
\end{pspicture}
\end{center}
  \begin{center}
  \includegraphics[width=60mm]
  {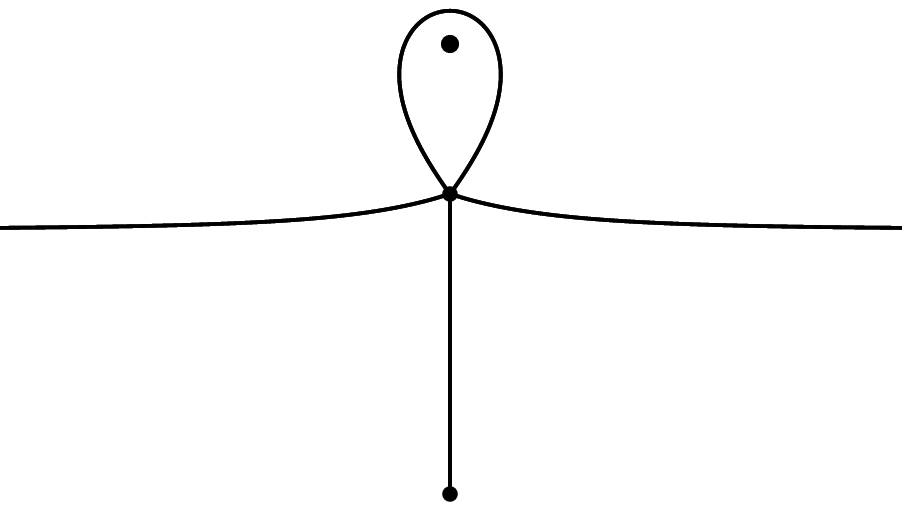} 
  \end{center} \vspace{-1.7em}
  \caption{The P-Stokes geometry of $\PIII$ with a 
  loop-type $P$-Stokes segment (described on the u-plane).}
  \label{fig:P3-Stokes}
  \end{figure}

On the other hand, Figure \ref{fig:P3-Stokes} 
depicts the $P$-Stokes geometry of $\PIII$ when 
$c=i$ by using a new variable 
\begin{equation} \label{eq:u-P3}
u=\frac{2\lambda_0(t)^2}{c\lambda_0(t)-t}
\end{equation}
of the Riemann surface 
of $\lambda_0(t)$. Since $t=-u^2(u-c)/2$, 
the quadratic differential becomes  
\begin{equation} \label{eq:quad-diff-D7}
F^{(1)}_{{\rm III'}(D_7)}(t)~dt^2 = 
{\rm quad}_{{\rm III'}(D_7)}(u,c)~du^2, \quad
{\rm quad}_{{\rm III'}(D_7)}(u,c) =
\frac{(3u-2c)^3}{u(u-c)^2}~du^2.
\end{equation}
Hence there is a one simple $P$-turning point and 
one $P$-turning point of simple-pole type 
in the $P$-Stokes geometry of $\PIII$. 
In Figure \ref{fig:P3-Stokes} we can observe 
that a $P$-Stokes segment of loop-type, 
which is denoted by $\Gamma$, appears around 
the double pole $u=c$ of \eqref{eq:quad-diff-D7}.  
It is known that such a loop appears when the residue 
of $\sqrt{{\rm quad}_{{\rm III'}(D_7)}(u,c)}~du$ 
at $u=c$ takes a pure imaginary value 
(see \cite[Section 7]{Strebel}).
Since the quadratic differential 
\eqref{eq:quad-diff-D7} satisfies 
\[
r^{-1} \sqrt{{\rm quad}_{{\rm III'}(D_7)}
(ru,rc)}~d(ru) = 
\sqrt{{\rm quad}_{{\rm III'}(D_7)}(u,c)}~du 
\]
for any $r \ne 0$, we can conclude that 
the configuration of $P$-Stokes geometry of $\PIII$ 
(described in the variable $u$ given by \eqref{eq:u-P3})
when $c \in i{\mathbb R}_{>0}$ 
is the same as in Figure \ref{fig:P3-Stokes}.
Furthermore, since ${\rm quad}_{\III}(u,c)$ is also 
invariant under $(u,c) \mapsto (-u,-c)$, the $P$-Stokes 
geometry when $c \in i \hspace{+.1em} {\mathbb R}_{<0}$ 
is the reflection $u \mapsto -u$ of Figure \ref{fig:P3-Stokes}.

  \begin{figure}[h]
  \begin{center}
\begin{pspicture}(0,0)(0,0)
\psset{fillstyle=none}
\rput[c]{0}(-4.7,-1.9){\large $a(t)$} 
\rput[c]{0}(-2.4,-2.3){\large $0$}
\rput[c]{0}(-2.9,-3.8){\large $\lambda_0(t)$}
\rput[c]{0}(-3.5,-2.6){\large $\gamma_A$}
\rput[c]{0}(+1.8,-3.6){\large $a(t)$} 
\rput[c]{0}(4.0,-3.1){\large $0$}
\rput[c]{0}(3.2,-1.7){\large $\lambda_0(t)$}
\rput[c]{0}(3.0,-2.85){\large $\gamma_B$}
\end{pspicture}
\end{center}
  \begin{minipage}{0.48\hsize}
  \begin{center}
  \includegraphics[width=50mm]{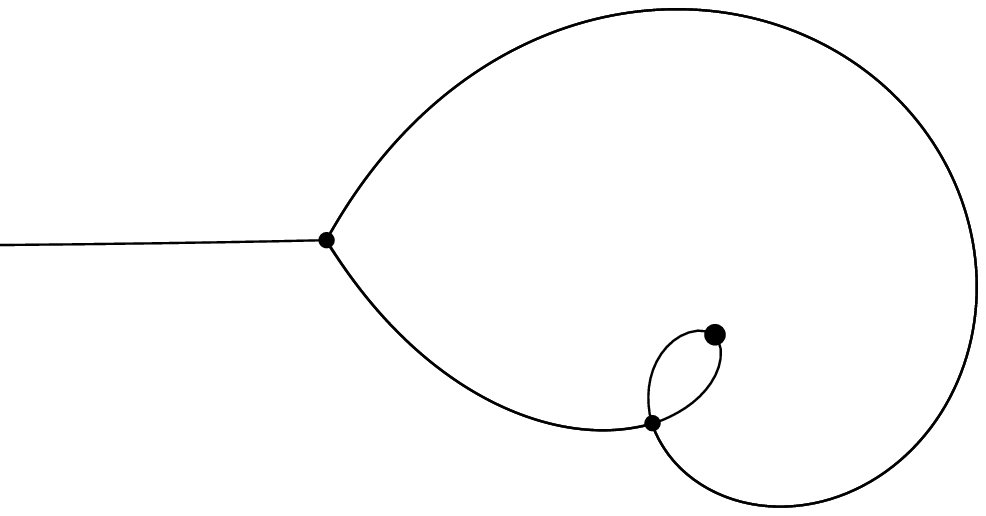} \\[-.5em]
  {($SL_{\III}$-$A$): The Stokes geometry of 
  $\SLIII$ corresponding to $u_{\ast,A}$.}
  \end{center}
  \end{minipage} \hspace{+.6em} 
  \begin{minipage}{0.48\hsize}
  \begin{center}
  \includegraphics[width=50mm]{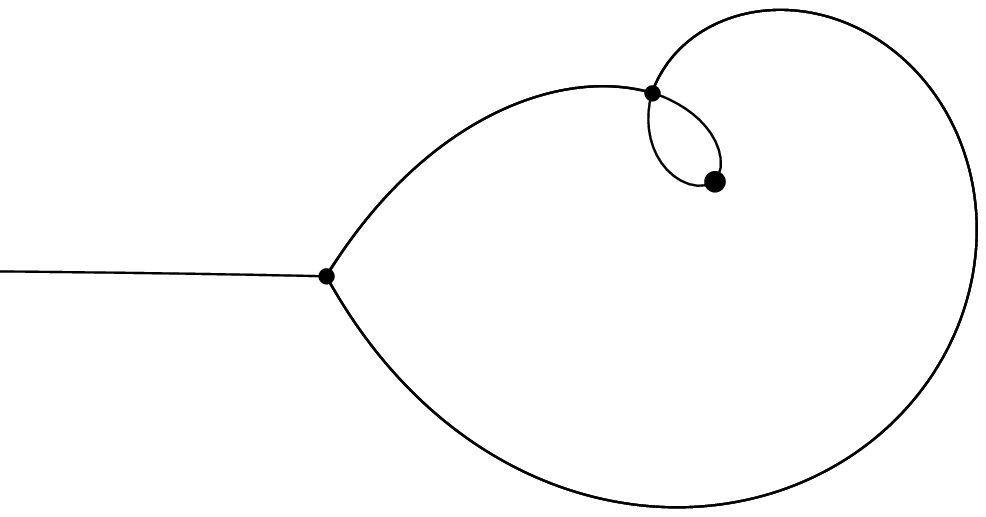} \\[-.5em]
  {($SL_{\III}$-$B$): The Stokes geometry of 
  $\SLIII$ corresponding to $u_{\ast,B}$.}
  \end{center} 
  \end{minipage}
  \caption{The Stokes geometries of $\SLIII$ on 
  the loop-type $P$-Stokes segment.}
  \label{fig:SLD7-Stokes}
  \end{figure}

Figure \ref{fig:SLD7-Stokes} ($SL_{\III}$-$A$)
(resp., ($SL_{\III}$-$B$)) 
depicts the Stokes geometry of $\SLIII$ 
when $t$ is fixed at a point $t_{\ast,A}$ 
(resp., $t_{\ast,B}$) corresponding to 
$u_{\ast,A}$ (resp., $u_{\ast,B}$) 
which lies on the loop-type $P$-Stokes segment 
$\Gamma$ in Figure \ref{fig:P3-Stokes}.
There are two Stokes segments in the Stokes 
geometry of $\SLIII$ both of which connect 
the double turning point $\lambda_0(t)$ and 
the same simple turning point $a(t)$. 
When $t$ tends to the simple $P$-turning point $r$ 
along $\Gamma$ in Figure \ref{fig:P3-Stokes}, 
one of the two Stokes segments shrinks to a point 
(cf.~Proposition \ref{prop:PStokes-and-Stokes} (i)).
In Figure \ref{fig:SLD7-Stokes} ($SL_{\III}$-$A$)
(resp., ($SL_{\III}$-$B$)) the Stokes segment 
$\gamma_A$ (resp., $\gamma_B$) shrinks to a point 
when $t$ tends to $r$ along $\Gamma$ in 
clockwise (resp., counter-clockwise) direction.
 

In Figure \ref{fig:SL2-Stokes} and 
Figure \ref{fig:SLD7-Stokes} 
we can observe common properties 
of the Stokes geometries of $\SLJ$'s 
when $t$ lies on a $P$-Stokes segment. 
Firstly, there appear two Stokes segments 
each of which connects the double turning point 
$\lambda_0(t)$ and a simple turning point. 
Secondly, these two Stokes 
segments are adjacent in the Stokes curves emanating from 
$\lambda_0(t)$. We can show that these properties are 
commonly observed for the Stokes geometry of $\SLJ$ 
when $t$ lies on a $P$-Stokes segment of $\PJ$. 

\begin{prop}\label{prop:P-saddles}
Let $r_1$ and $r_2$ be (possibly same) simple $P$-turning 
points of $\lambda_J$ which are not of simple-pole type, and 
$a_1(t)$ and $a_2(t)$ be the simple turning points of 
$(SL_{J})$ corresponding to $r_1$ and $r_2$ by Proposition 
\ref{prop:PStokes-and-Stokes} (i). 
Suppose that $r_1$ and $r_2$ are connected by 
a $P$-Stokes segment $\Gamma$, and take a point $t_{\ast}$ 
which lies on $\Gamma$ as in Figure \ref{fig:two-saddles}. 
Then, there appear two Stokes segments $\gamma_1$ and 
$\gamma_2$ in the Stokes geometry of $\SLJ$ when $t=t_{\ast}$, 
where $\gamma_1$ (resp., $\gamma_2$) connects 
$\lambda_0(t_{\ast})$ and $a_1(t_{\ast})$ 
(resp., $a_2(t_{\ast})$). 
Moreover, $\gamma_1$ and $\gamma_2$
are adjacent Stokes curves in the four Stokes curves 
emanating from $x=\lambda_0(t_{\ast})$. 
\end{prop}

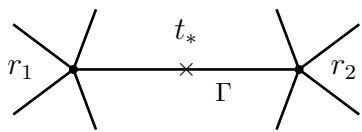
\begin{figure}[h]
\begin{center}
\begin{pspicture}(0.5,7.5)(14,9.0)
%
\psset{fillstyle=solid, fillcolor=black}
\pscircle(5,8){0.06} 
\pscircle(8,8){0.06} 
\psset{fillstyle=none}
\rput[c]{0}(7,7.7){$\Gamma$}
\rput[c]{0}(6.5,8){$\times$}
\rput[c]{0}(4.3,8){\large $r_1$}
\rput[c]{0}(8.6,8){\large $r_2$}
\rput[c]{0}(6.5,8.5){\large $t_{\ast}$}
\psset{linewidth=1pt}
\psline(5,8)(8,8)
\psline(5,8)(4.2,8.6)
\psline(5,8)(4.2,7.4)
\psline(5,8)(5.3,8.8)
\psline(5,8)(5.3,7.2)
\psline(8,8)(8.8,8.6)
\psline(8,8)(8.8,7.4)
\psline(8,8)(7.7,8.8)
\psline(8,8)(7.7,7.2)
\end{pspicture}
\end{center}
\caption{A $P$-Stokes segment $\Gamma$ and 
two simple $P$-turning points $r_1$ and $r_2$.} 
\label{fig:two-saddles}
\end{figure}

\begin{proof}
Since $t_{\ast}$ lies on $P$-Stokes curves emanating 
from $r_1$ and $r_2$ simultaneously, it follows from 
Proposition \ref{prop:PStokes-and-Stokes}
that the double turning point $x=\lambda_0(t)$ lies on 
both Stokes curves emanating from $a_1(t)$ and $a_2(t)$ 
when $t=t_{\ast}$. 
Hence, in the Stokes geometry of $\SLJ$ there are 
two Stokes segments $\gamma_1$ and $\gamma_2$ 
which connects $\lambda_0(t_{\ast})$ and 
$a_{1}(t_{\ast})$ and $a_2(t_{\ast})$, respectively.
Thus the following two cases (i) and (ii) in Figure 
\ref{fig:candidates} may possibly occur: 
In the case (i) (resp., (ii)) 
$\gamma_1$ and $\gamma_2$ are adjacent 
(resp., opposite) Stokes curves which emanate from 
$\lambda_0$. However, the case (ii) does not happen in 
our assumption, due to the following reason. 

For $k=1,2$, set 
\begin{eqnarray}
\label{eq:phi-Jk}
\phi_{J,k}(t) & = & \int_{r_k}^t \sqrt{F^{(1)}_J(t)}~dt, \\ 
{v}_{J,k}(t) & = & \int_{a_k(t)}^{\lambda_0(t)}
\sqrt{Q_{J,0}(x,t)}~dx. \label{eq:v-Jk} 
\end{eqnarray}
Then, Proposition \ref{prop:PStokes-and-Stokes} (ii) 
implies that $v_{J,k}(t) = \phi_{J,k}(t)/2$ ($k=1,2$). 
The real parts of $\phi_{J,1}(t_{\ast})$ and 
$\phi_{J,2}(t_{\ast})$ 
have different sign from each other since the real parts 
are monotonously increasing or decreasing along 
$P$-Stokes curves. 
Thus the real parts of $v_{J,1}(t_{\ast})$ and 
$v_{J,2}(t_{\ast})$ also have different signs.
Therefore, the case (ii) in Figure \ref{fig:candidates} 
never happens and only the case (i) appears. 
\end{proof} 

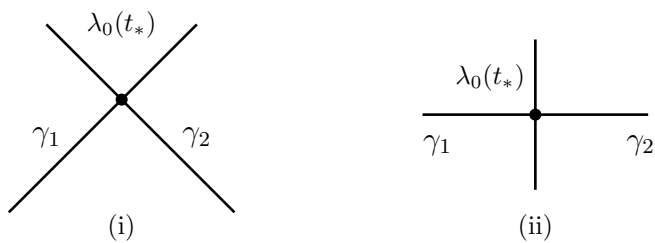
\begin{figure}[h]
\begin{center}
\begin{pspicture}(0,0.7)(11,3.3)
\psset{fillstyle=solid, fillcolor=black}
\pscircle(2.5,2.2){0.08} 
\pscircle(8,2){0.08} 
\psset{fillstyle=none}
\rput[c]{0}(2.5,3.2){$\lambda_0(t_{\ast})$}
\rput[c]{0}(1.5,1.7){\large $\gamma_1$}
\rput[c]{0}(3.5,1.7){\large $\gamma_2$} 
\rput[c]{0}(2.5,0.5){(i)} %
\psset{linewidth=1pt}
\psline(1,0.7)(3.5,3.2)
\psline(4,0.7)(1.5,3.2)
\rput[c]{0}(7.4,2.5){$\lambda_0(t_{\ast})$}
\rput[c]{0}(6.7,1.6){\large $\gamma_1$}
\rput[c]{0}(9.4,1.6){\large $\gamma_2$} 
\rput[c]{0}(8.0,0.5){(ii)} %
\psset{linewidth=1pt}
\psline(6.5,2)(9.5,2)
\psline(8,3)(8,1)
\end{pspicture}
\end{center}
\caption{Two candidates of Stokes segments.} 
\label{fig:candidates}
\end{figure}

Proposition \ref{prop:P-saddles} implies that 
there are the following two possibilities for the geometric 
type of the Stokes geometry of $\SLJ$ when $t_{\ast}$ lies 
on a $P$-Stokes segments
(cf.~Figure \ref{fig:SLJ-Stokes-with-two-saddles}).
\begin{itemize}
\item[(a)] The double turning point $\lambda_0(t_{\ast})$ 
is connected with {\em different} simple turning points 
by two Stokes segments. This case is observed in Figure 
\ref{fig:SL2-Stokes}. 

\item[(b)] The double turning point 
$\lambda_0(t_{\ast})$ is connected with the {\em same} 
simple turning point by two Stokes segments. 
This case is observed in Figure 
\ref{fig:SLD7-Stokes}.
\end{itemize}

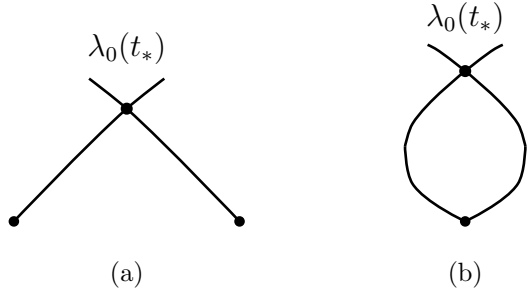
\begin{figure}[h]
\begin{center}
\begin{pspicture}(0.5,2.0)(10,5.5)
%
\psset{fillstyle=solid, fillcolor=black}
\pscircle(3.0,4.0){0.08} 
\pscircle(1.5,2.5){0.07}
\pscircle(4.5,2.5){0.07}
\pscircle(7.5,4.5){0.08} 
\pscircle(7.5,2.50){0.07}
\psset{fillstyle=none}
\rput[c]{0}(3.0,4.8){\large $\lambda_0(t_{\ast})$}
\rput[c]{0}(3.0,1.8){(a)} 
\psset{linewidth=1pt}
\pscurve(1.5,2.5)(3,4)(3.5,4.4)
\pscurve(4.5,2.5)(3,4)(2.5,4.4)
\rput[c]{0}(7.5,5.2){\large $\lambda_0(t_{\ast})$}
\rput[c]{0}(7.5,1.8){(b)} 
\psset{linewidth=1pt}
\pscurve(7.5,2.5)(6.8,3)(6.7,3.5)
\pscurve(7.5,2.5)(8.2,3)(8.3,3.5)
\pscurve(6.7,3.5)(6.8,3.8)(7.5,4.5)
\pscurve(8.3,3.5)(8.2,3.8)(7.5,4.5)
\pscurve(7.5,4.5)(7.9,4.8)(8.,4.85)
\pscurve(7.5,4.5)(7.1,4.8)(7.0,4.85)
%
\end{pspicture}
\end{center}
\caption{Two Stokes segments in the Stokes geometry of $\SLJ$.} 
\label{fig:SLJ-Stokes-with-two-saddles}
\end{figure}

The following fact will be used in the proof 
of our main results.
\begin{lem}\label{lemma:integral-equality}
In the same situation of Proposition \ref{prop:P-saddles}, 
we have 
\begin{equation} \label{eq:integrals-for-saddle-connections}
\int_{a_1(t)}^{a_2(t)}\sqrt{Q_{J,0}(x,t)}~dx = 
\frac{1}{2} \int_{r_1}^{r_2} \sqrt{F^{(1)}_J(t)}~dt.
\end{equation}
Here the path of integral in the left-hand side is 
taken along a composition of two Stokes segments 
$\gamma_1$ and $\gamma_2$ in the Stokes geometry 
of $\SLJ$, while that in the right-hand side 
is taken along the $P$-Stokes segment $\Gamma$. 
\end{lem}
\begin{proof}
Let $\phi_{J,k}(t)$ and $v_{J,k}(t)$ be functions 
defined in \eqref{eq:phi-Jk} and \eqref{eq:v-Jk}. 
Since $v_{J,k}(t) = \phi_{J,k}(t)/2$ holds for $k=1,2$, 
we have $v_{J,1}(t) - v_{J,2}(t) = 
(\phi_{J,1}(t) - \phi_{J,2}(t))/2$. 
This shows the desired relation.
\end{proof}

Lemma \ref{lemma:integral-equality} entails that 
the integral of $\sqrt{Q_{J,0}(x,t)}~dx$ appearing 
\eqref{eq:integrals-for-saddle-connections} does not 
depend on $t$. Generally, the integral of 
$S_{J,{\rm odd}}(x,t,\eta)$ along a closed cycle 
in the Riemann surface of $\sqrt{Q_{J,0}(x,t)}$ is independent 
of $t$ by \eqref{eq:t-derivative-of-Sodd}.
Especially, from the equality 
\eqref{eq:integrals-for-saddle-connections} 
and Table \ref{table:residues} we can show the following. 
\begin{lem} \label{lem:integral-P2-P3}
\begin{enumerate}[\upshape (i)]
\item For $J = {\rm II}$, we have
\begin{equation} \label{eq:integral-P2}
\int_{r_1}^{r_2}\sqrt{F^{(1)}_{\rm II}(t)}~dt =  
\pm 2\pi i c
\end{equation}
when $c \in i \hspace{+.1em} {\mathbb R}_{\ne 0}$.
Here $r_1$ and $r_2$ are two simple $P$-turning points 
of $\PII$ connected by a $P$-Stokes segment, and the path 
of integral is taken along the $P$-Stokes segment.
The sign $\pm$ depends on the branch of square root. %
\item For $J = {\rm III'}(D_7)$, we have
\begin{equation} \label{eq:integral-D7}
\int_{\Gamma}\sqrt{F^{(1)}_{{\rm III'}(D_7)}(t)}~dt = 
\pm  2 \pi i c
\end{equation}
when $c \in i \hspace{+.1em} {\mathbb R}_{\ne 0}$.
Here the path of integral is taken along the loop-type 
$P$-Stokes segment $\Gamma$ depicted in Figure 
\ref{fig:P3-Stokes}.
The sign $\pm$ depends on the branch of square root. %
\end{enumerate}
\end{lem}
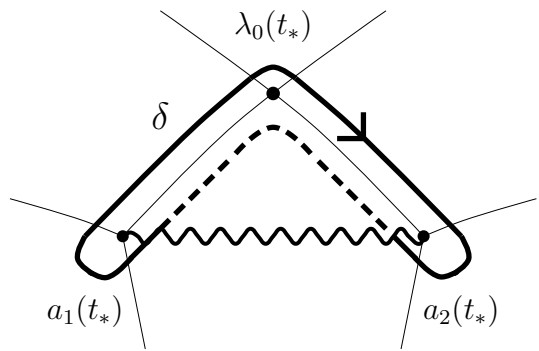
\begin{figure}
\begin{center}
\begin{pspicture}(0,1)(7,6)
%
\psset{fillstyle=solid, fillcolor=black}
\pscircle(3.5,4.4){0.09} 
\pscircle(1.5,2.5){0.08}
\pscircle(5.5,2.5){0.08}
\psset{fillstyle=none}
\psset{linewidth=1.5pt}
\pscurve(5.5,2.5)(5.45,2.4)(5.3,2.5)(5.2,2.6)(5.1,2.5)
(5.0,2.4)(4.9,2.5)(4.8,2.6)(4.7,2.5)(4.6,2.4)(4.5,2.5)
(4.4,2.6)(4.3,2.5)(4.2,2.4)(4.1,2.5)(4,2.6)(3.9,2.5)
(3.8,2.4)(3.7,2.5)(3.6,2.6)(3.5,2.5)(3.4,2.4)(3.3,2.5)
(3.2,2.6)(3.1,2.5)(3,2.4)(2.9,2.5)(2.8,2.6)(2.7,2.5)
(2.6,2.4)(2.5,2.5)(2.4,2.6)(2.3,2.5)(2.2,2.4)(2.1,2.5)
(2,2.6)(1.9,2.5)(1.8,2.4)(1.7,2.5)(1.6,2.55)(1.5,2.5)
\rput[c]{0}(3.5,5.3){\large $\lambda_0(t_{\ast})$}
\rput[c]{0}(1.0,1.5){\large $a_{1}(t_{\ast})$}
\rput[c]{0}(6.0,1.5){\large $a_{2}(t_{\ast})$} 
\rput[c](2.0,4.1){\Large $\delta$} 
\psset{linewidth=0.2pt}
\pscurve(1.5,2.5)(3,4)(5,5.5)
\pscurve(5.5,2.5)(4,4)(2,5.5)
\pscurve(1.5,2.5)(1,2.7)(0,3) 
\pscurve(5.5,2.5)(6,2.7)(7,3)
\pscurve(1.5,2.5)(1.6,2)(1.8,1)
\pscurve(5.5,2.5)(5.4,2)(5.2,1)
\psset{linewidth=2.0pt}
\pscurve(3.2,4.6)(3.5,4.75)(3.8,4.6)
\pscurve(1,2.5)(2.5,4)(3.2,4.6)
\pscurve(6,2.5)(4.5,4)(3.8,4.6)
\pscurve(1,2.5)(0.9,2.2)(1.1,2)
\pscurve(6,2.5)(6.1,2.2)(5.9,2)
\pscurve(1.05,2.03)(1.28,1.95)(1.62,2.22)
\pscurve(5.95,2.03)(5.72,1.96)(5.38,2.22)
\pscurve(1.6,2.2)(1.7,2.3)(2.0,2.6)
\pscurve(5.4,2.2)(5.3,2.3)(5.1,2.5)
\psline(4.7,3.8)(4.35,3.8)
\psline(4.7,3.8)(4.7,4.15)
\psset{linewidth=2.0pt,linestyle=dashed}
\pscurve(3.0,3.6)(3.5,3.95)(4.0,3.6) 
\pscurve(2.0,2.6)(2.5,3.1)(2.9,3.5) 
\pscurve(5.0,2.6)(4.5,3.1)(4.1,3.5) 
%
\end{pspicture}
\end{center}
\caption{The cycle $\delta$.} 
\label{fig:delta}
\end{figure}
\begin{proof}
We prove \eqref{eq:integral-P2}. 
Let $t_{\ast}$ be a point on the $P$-Stokes segment 
connecting $r_1$ and $r_2$, and $a_1(t)$ and 
$a_2(t)$ be the simple turning points 
of $\SLII$ which correspond to $r_1$ and $r_2$ by 
Proposition \ref{prop:PStokes-and-Stokes} (i).
Then, we have 
\[
\int_{r_1}^{r_2}\sqrt{F^{(1)}_{\rm II}(t)}~dt =  
2 \int_{a_1(t_{\ast})}^{a_2(t_{\ast})} 
\sqrt{Q_{\rm II,0}(x,t_{\ast})}~dx
\]
by \eqref{eq:integrals-for-saddle-connections}. 
The integral of $\sqrt{Q_{\rm II,0}(x,t_{\ast})}~dx$ can 
be written as 
\[
2 \int_{a_1(t_{\ast})}^{a_2(t_{\ast})} 
\sqrt{Q_{\rm II,0}(x,t_{\ast})}~dx = 
\oint_{\delta} \sqrt{Q_{\rm II,0}(x,t_{\ast})}~dx, 
\]
where $\delta$ is a closed cycle in the Riemann surface
of $\sqrt{Q_{\rm II,0}(x,t_{\ast})}$ described 
in Figure \ref{fig:delta}. 
The wiggly line in Figure \ref{fig:delta} represents 
the branch cut to determine the branch of 
$\sqrt{Q_{\rm II,0}(x,t_{\ast})}$, 
and solid and dashed line represents the path on 
the first and the second sheet of the Riemann surface 
of $\sqrt{Q_{\rm II,0}(x,t_{\ast})}$, respectively.
Since the 1-from $\sqrt{Q_{\rm II,0}(x,t_{\ast})}~dx$ 
has no other singular point other than 
$x=a_1(t_{\ast})$, $a_2(t_{\ast})$ 
and $\infty$, we have
\begin{equation} \label{eq:P2-SL2-residue}
\oint_{\delta} \sqrt{Q_{\rm II,0}(x,t_{\ast})}~dx = 
2 \pi i \Res_{x=\infty} 
\sqrt{Q_{\rm II,0}(x,t_{\ast})}~dx = 
\pm  2 \pi i c.
\end{equation}
Here we have used \eqref{eq:resSodd=res-1} 
and Table \ref{table:residues} of residues.
Thus we have proved \eqref{eq:integral-P2}. 
The equality \eqref{eq:integral-D7} 
can be proved in the same manner by using 
the following fact:  
\begin{equation} \label{eq:PD7-SLD7-residue}
\Res_{x=0} \sqrt{Q_{{\rm III}(D_7),0}(x,t)}~dx = 
\pm \frac{c}{2}.
\end{equation}
\end{proof}

\section{WKB theoretic transformation to 
$\PII$ on $P$-Stokes segments} 
\label{section:main-results}

Here we show our main claims concerning with WKB 
theoretic transformations between 
Painlev\'e transcendents on $P$-Stokes segments. 
Since we simultaneously deal with two different 
Painlev\'e equations $\PJ$ and $\PII$, 
in this section we put symbol $\sim$ over variables 
or functions relevant to $\PJ$ and $\SLJ$ 
in order to avoid confusions. 

\subsection{Assumptions and statements}

Let $(\tilde{\lambda}_J, \tilde{\nu}_J) = $ 
$(\tilde{\lambda}_J(\tt,\eta;\talpha,\tbeta), 
\tilde{\nu}_J(\tt,\eta;\talpha,\tbeta))$ 
be a 2-parameter solution of $(H_J)$ 
defined in a neighborhood of a point 
$\tt_{\ast} \in \Omega_J$, and consider 
$\SLJ$ and $\DJ$ with 
$(\tilde{\lambda}_J, \tilde{\nu}_J)$ 
substituted into their coefficients. 
Here we assume the following conditions.
\begin{ass} \label{ass:transform-to-P2}
\begin{enumerate}[\upshape (1)]
\item %
$J \in \{{\rm II}, {{\rm III'}(D_6)}, 
{\rm IV}, {\rm V}, {\rm VI} \}$. %
\item %
There is a $P$-Stokes segment $\tilde{\Gamma}$ 
in the $P$-Stokes geometry of $\PJ$ which connects 
two {\em different} simple $P$-turning points 
$\tilde{r}_1$ and $\tilde{r}_2$ of $\tlambda_J$ 
(which are not simple-pole type), and 
the point $\tt_{\ast}$ in question lies on $\tilde{\Gamma}$. %
\item %
The function \eqref{eq:phaseJ} appearing 
in the instanton $\tPhi_J(\tt,\eta)$ of the 2-parameter solution 
$(\tilde{\lambda}_J, \tilde{\nu}_J)$ is normalized 
at the simple $P$-turning point $\tr_1$ as 
\begin{equation} \label{eq:tlidephiJ}
\tphi_J(\tt) = \int_{\tr_1}^{\tt}\sqrt{\tF_J^{(1)}(\tt)}~d\tt.
\end{equation}
\item %
The Stokes geometry of $\SLJ$ at $\tt=\tt_{\ast}$ 
contains the same configuration as in Figure 
\ref{fig:SLJ-Stokes-with-two-saddles} (a). 
That is, the double turning point $\tlambda_0(\tt_{\ast})$ 
is connected with two {\em different} simple turning points 
$\tilde{a}_1(\tt_{\ast})$ and $\tilde{a}_2(\tt_{\ast})$
by two Stokes segments $\tilde{\gamma}_{1}$ 
and $\tilde{\gamma}_{2}$, respectively. 
Here the labels of the simple turning points 
and the Stokes segments are assigned by 
the following rule: When $\tt$ tends to $\tr_1$ 
(resp., $\tr_2$) along $\tGamma$, $\ta_1(\tt)$ 
(resp., $\ta_2(\tt)$) merges with $\tlambda_0(\tt)$
(cf.~Proposition \ref{prop:PStokes-and-Stokes}). %
\item %
All singular points of 
$\tQ_{J,0}(\tx,\tt_{\ast})$ (as a function of $\tx$)
are poles of {\em even} order. %
\end{enumerate}
\end{ass} %
Since the $P$-Stokes geometry for $J = {\rm I}$, 
${\rm III'}(D_7)$ and ${\rm III'}(D_8)$ never 
contains a $P$-Stokes segment connecting 
two different simple $P$-turning points,  
we have excluded these cases. 
One of our main results below claims that, 
under Assumption \ref{ass:transform-to-P2}
we can construct a formal transformation 
series defined on a neighborhood of the union 
$\tilde{\gamma}_1 \cup \tilde{\gamma}_2$ 
of two Stokes segments that brings  
$\SLJ$ to $\SLII$ with an appropriate 
2-parameter solution $(\lambda_{\rm II},\nu_{\rm II})$ of 
$(H_{\rm II})$ being substituted into $(\lambda,\nu)$
in ($SL_{\rm II}$), in the following sense. 

First, we fix the constant $c$ contained in $\PII$ 
and $\SLII$ by 
\begin{equation} \label{eq:c-P2}
c = \frac{1}{2\pi i} \int_{\tr_1}^{\tr_2}
\sqrt{\tF_{J}^{(1)}(\tt)}~d\tt, 
\end{equation}
where the path of integral is taken along the 
$P$-Stokes segment $\tGamma$. 
Since the function \eqref{eq:tlidephiJ}
is monotone and takes real values along $\tGamma$, 
the constant $c$ determined by \eqref{eq:c-P2} is 
non-zero and pure-imaginary. 
Here we assume that the imaginary part of $c$ is positive;
$c \in i \hspace{+.1em} {\mathbb R}_{>0}$. 
Then, the geometric configuration of $P$-Stokes geometry 
of $\PII$ (described in the variable $u$ given 
by \eqref{eq:u-P2}) when $c$ is given by 
\eqref{eq:c-P2} is the same as 
Figure \ref{fig:P2-and-SL2-Stokes-geometry} (P).
Thus, the $P$-Stokes geometry of $\PII$ has three simple 
$P$-turning points, and three $P$-Stokes segments 
appear simultaneously. (As is remarked in 
Section \ref{section:degeneration-of-PStokes-geometry}, 
when $c \in i \hspace{+.1em} {\mathbb R}_{<0}$, 
the $P$-Stokes geometry of $\PII$ is the reflection 
$u \mapsto -u$ of 
Figure \ref{fig:P2-and-SL2-Stokes-geometry} (P). 
Our discussion below is also applicable to the case 
of $c \in i \hspace{+.1em} {\mathbb R}_{<0}$.) 
Furthermore, we can verify that the corresponding 
Stokes geometry of $\SLII$ on a $P$-Stokes segment 
is of the same type as in 
Figure \ref{fig:P2-and-SL2-Stokes-geometry} (SL).
That is, when we take any point $t_{\ast}$ 
on a $P$-Stokes segment of $\PII$, 
say $\Gamma$ depicted in 
Figure \ref{fig:P2-and-SL2-Stokes-geometry} (P), 
then the corresponding Stokes geometry 
of $\SLII$ has one double turning point at 
$x=\lambda_0(t_{\ast})$ and two simple turning points 
$x=a(t_{\ast})$ and $a'(t_{\ast})$, 
and there are two Stokes segments 
$\gamma$ and $\gamma'$ by which 
$\lambda_0(t_{\ast})$ is connected with 
these simple turning points. 
Note that $a(t)$ (resp., $a'(t)$) merges with 
$\lambda_0(t)$ as $t$ tends to $r$ (resp., $r'$)
along the $P$-Stokes segment $\Gamma$.


  \begin{figure}[h]
  \begin{center}
\begin{pspicture}(0,0)(0,0)
\psset{fillstyle=none}
\rput[c]{0}(-3.7,-1.65){$\times$} 
\rput[c]{0}(-3.7,-3.0){\large $0$}
\rput[c]{0}(-3.7,-1.2){\large $t_{\ast}$}
\rput[c]{0}(-5.1,-1.8){\large $r$}
\rput[c]{0}(-1.7,-1.7){\large $r'$}
\rput[c]{0}(-2.8,-1.4){$\Gamma$}
\rput[c]{0}(1.0,-3.1){\large $a(t_{\ast})$}
\rput[c]{0}(5.0,-3.5){\large $a'(t_{\ast})$}
\rput[c]{0}(2.7,-0.8){\large $\lambda_0(t_{\ast})$}
\rput[c]{0}(1.9,-2.0){\large $\gamma$}
\rput[c]{0}(3.6,-2.1){\large $\gamma'$}
\end{pspicture}
\end{center}
  \begin{minipage}{0.45\hsize}
  \begin{center}
  \includegraphics[width=50mm]{P2-saddle-u.eps} \\[-.2em]
  {(P): $P$-Stokes geometry of $\PII$ 
  (described on the $u$-plane).}
  \end{center}
  \end{minipage} \hspace{+1.5em}
  \begin{minipage}{0.45\hsize}
  \begin{center}
  \includegraphics[width=47mm]
  {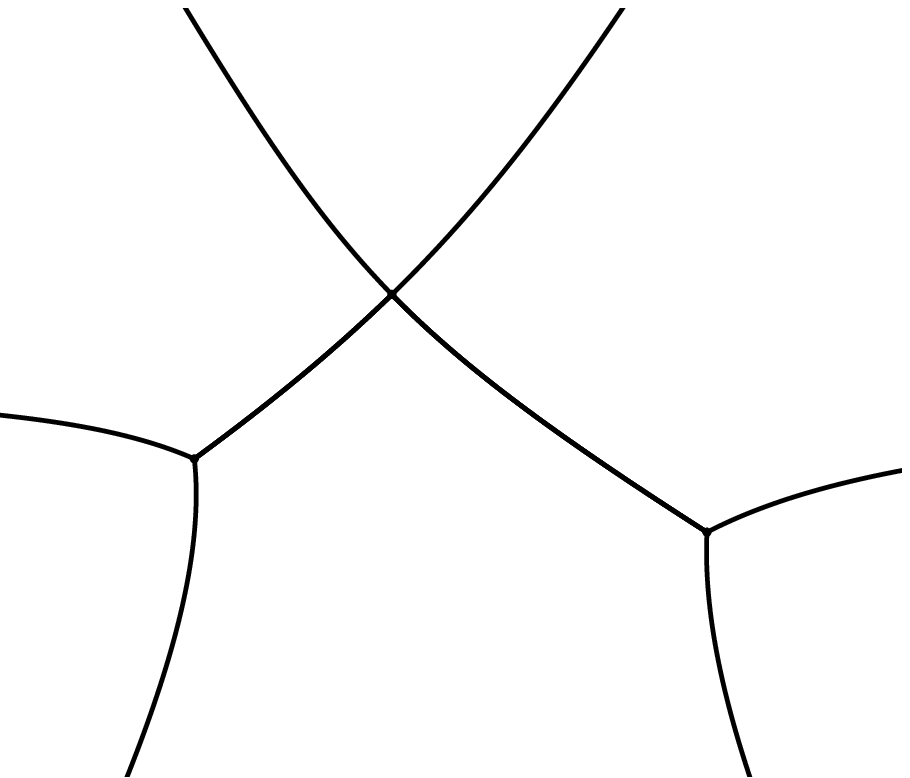} \\[+2.5em]
  {(SL): Stokes geometry of $\SLII$ at $t = t_{\ast}$.}
  \end{center} 
  \end{minipage} 
  \caption{The $P$-Stokes geometry of $\PII$
  and the Stokes geometry of $\SLII$.}
  \label{fig:P2-and-SL2-Stokes-geometry}
  \end{figure}

Having these geometric properties in mind, 
we formulate the precise statement of our 
first main result as follows. 

\begin{thm} \label{thm:main-theorem1}
Under Assumption \ref{ass:transform-to-P2}, 
for any 2-parameter solution 
$(\tilde{\lambda}_J, \tilde{\nu}_J) = $
$(\tilde{\lambda}_J(\tilde{t},\eta; \tilde{\alpha},
\tilde{\beta}), \tilde{\nu}_J(\tilde{t},\eta; 
\tilde{\alpha}, \tilde{\beta}))$ of $(H_{J})$, 
there exist
\begin{itemize} %
\item a domain $\tilde{U}$ which contains the union  
$\tilde{\gamma}_1 \cup \tilde{\gamma}_2$ 
of two Stokes segments, %
\item a neighborhood $\tilde{V}$ of $\tilde{t}_{\ast}$, %
\item formal series 
\[
x(\tx,\tt,\eta)=\sum_{j\ge0}\eta^{-j/2}x_{j/2}(\tx,\tt,\eta), 
\quad 
t(\tt,\eta)=\sum_{j\ge0}\eta^{-j/2}t_{j/2}(\tt,\eta)
\]
whose coefficients 
$\{x_{j/2}(\tilde{x},\tilde{t},\eta)\}_{j=0}^{\infty}$ 
and $\{t_{j/2}(\tilde{t},\eta)\}_{j=0}^{\infty}$ 
are functions defined on $\tilde{U}\times\tilde{V}$ 
and $\tV$, respectively, and may depend on $\eta$, %
\item a 2-parameter solution 
\begin{eqnarray*}
(\lambda_{\rm II},\nu_{\rm II}) & = &
(\lambda_{\rm II}(t,\eta;\alpha,\beta),
\nu_{\rm II}(t,\eta;\alpha,\beta)), \\
(\alpha,\beta) & = & (\sum_{n=0}^{\infty}\eta^{-n}\alpha_n, 
\sum_{n=0}^{\infty}\eta^{-n}\beta_n)
\end{eqnarray*}
of $\HII$ with the constant $c$ being determined by 
\eqref{eq:c-P2}, and the function \eqref{eq:phaseJ} 
appearing in the instanton $\Phi_{\rm II}(t,\eta)$ 
that is normalized at a simple $P$-turning point $r_1$ 
of $\PII$ as 
\begin{equation} \label{eq:phiII}
\phi_{\rm II}(t) = \int_{r_1}^{t}\sqrt{F_{\rm II}^{(1)}(t)}~dt,
\end{equation} %
\end{itemize}
which satisfy the relations below: %
\begin{enumerate}[\upshape (i)]
\item %
The function $t_0(\tt)$ is independent of 
$\eta$ and satisfies
\begin{equation} \label{eq:correspondence-of-instantonP2}
\tphi_{J}(\tt) = \phi_{{\rm II}}(t_0(\tt)).
\end{equation}
\item $dt_0/d\tt$ never vanishes on $\tV$. %
\item The function $x_0(\tx,\tt)$ is also 
independent of $\eta$ and satisfies 
\begin{eqnarray}
x_0(\tlambda_0(\tt),\tt) & = & \lambda_0(t_{0}(\tt)), \\
x_0(\ta_k(\tt),\tt) & = & a_k(t_{0}(\tt))~~(k=1,2).
\end{eqnarray} %
Here $\lambda_0(t)$ and $a_{k}(t)$ ($k=1,2$)
are double and two simple turning points of $\SLII$.
\item $\p x_0/\p\tx$ never vanishes on $\tU\times\tV$. %
\item $x_{1/2}$ and $t_{1/2}$ vanish identically. %
\item The $\eta$-dependence of $x_{j/2}$ and $t_{j/2}$ ($j \ge 2$)
is only through instanton terms 
$\exp(\ell \hspace{+.1em} \tPhi_{J}(\tt,\eta))$
($\ell = j-2-2m$ with $0 \le m \le j-2$) that appears 
in the 2-parameter solution $(\tlambda_J,\tnu_{J})$ of $(H_J)$. %
\item The following relations hold:
\begin{eqnarray} \label{eq:equivalece-of-PJ}
x(\tlambda_J(\tt,\eta;\talpha,\tbeta),\tt,\eta) & = &
\lambda_{\rm II}(t(\tt,\eta),\eta;\alpha,\beta), \\%
\label{eq:equivalece-of-SLJ}
\tQ_{J}(\tx,\tt,\eta) & = & 
\left(\frac{\p x(\tx,\tt,\eta)}{\p\tx}\right)^2
Q_{\rm II}(x(\tx,\tt,\eta),t(\tt,\eta),\eta) \\
& & 
-\frac{1}{2}\eta^{-2}\{ x(\tx,\tt,\eta);\tx\}, \nonumber
\end{eqnarray}%
where the 2-parameter solutions of $(H_J)$ and $(H_{\rm II})$ 
are substituted into $(\lambda,\nu)$ in the coefficients 
of $\tQ_J$ and $Q_{\rm II}$, respectively, and
$\{x(\tx,\tt,\eta);\tx\}$ denotes the 
Schwarzian derivative \eqref{eq:Schwarzian-derivative}.
\end{enumerate} %
\end{thm}%

The rest of this section is devoted to 
the proof of Theorem \ref{thm:main-theorem1}.

\subsection{Construction of the top term of the transformation}
\label{section:top-terms}

Here we construct the top terms 
$x_{0}(\tx,\tt)$ and $t_{0}(\tt)$ of the formal series. 

First, we explain the construction of $t_0(\tt)$. 
Since $\tt$ lies on a $P$-Stokes curve emanating 
from $\tr_k$ ($k=1,2$), it is shown in 
\cite[Theorem 2.2]{KT96} that there exists a function 
$t^{(k)}_0(\tt)$ such that 
\begin{equation} \label{eq:top-t0k}
\tphi_{J,k}(\tt) = \phi_{{\rm II},k}(t^{(k)}_0(\tt))
\end{equation}
holds for each $k=1$ and $2$, where 
\begin{equation} \label{eq:tphi-and-phi2}
\tphi_{J,k}(\tt) = \int_{\tr_{k}}^{\tt}
\sqrt{\tF_{J}^{(1)}(\tt)}~d\tt,\quad
\phi_{{\rm II},k}(t) = \int_{r_{k}}^{t}
\sqrt{{F^{(1)}_{\rm II}}(t)}~dt.
\end{equation} %
Here $r_{1}$ and $r_{2}$ are two simple 
$P$-turning points of $\PII$ chosen by 
the following rule. 
Note that we have the following 
two possibilities for the configuration of 
the Stokes geometry of $\SLJ$ at $\tt_{\ast}$ 
(see Figure \ref{fig:P2-transformable-cases}):
\begin{itemize}
\item[(A)] The Stokes segment $\tgamma_2$ 
comes next to the Stokes segment $\tgamma_{1}$ 
in the {\em counter-clockwise} order 
near $\tlambda_0(\tt_{\ast})$. %
\item[(B)] The Stokes segment $\tgamma_2$ 
comes next to the Stokes segment $\tgamma_1$ 
in the {\em clockwise} order 
near $\tlambda_0(\tt_{\ast})$. %
\end{itemize}
Then, we set 
\begin{eqnarray} 
\label{eq:choice-of-r1-and-r2}
(r_1,r_2) = 
\begin{cases}
(r, r') & \text{when the case (A) in Figure 
\ref{fig:P2-transformable-cases} happens}, \\
(r', r) & \text{when the case (B) in Figure 
\ref{fig:P2-transformable-cases} happens},
\end{cases}
\end{eqnarray}
where $r$ and $r'$ are the $P$-turning points 
of $\PII$ depicted in Figure 
\ref{fig:P2-and-SL2-Stokes-geometry} (P). 
Moreover, the branch of $\sqrt{F_{\rm II}^{(1)}(t)}$ 
is taken so that the sign appearing in the right-hand side 
of \eqref{eq:integral-P2} is $+$:
\begin{equation} \label{eq:integral-P2-plus}
\int_{r_1}^{r_2} \sqrt{F_{\rm II}^{(1)}(t)}~dt = 
+ 2 \pi i  c.
\end{equation}
This choice \eqref{eq:choice-of-r1-and-r2}
of $r_1$ and $r_2$ is essential in 
the construction of $x_0(\tx,\tt)$ later. 

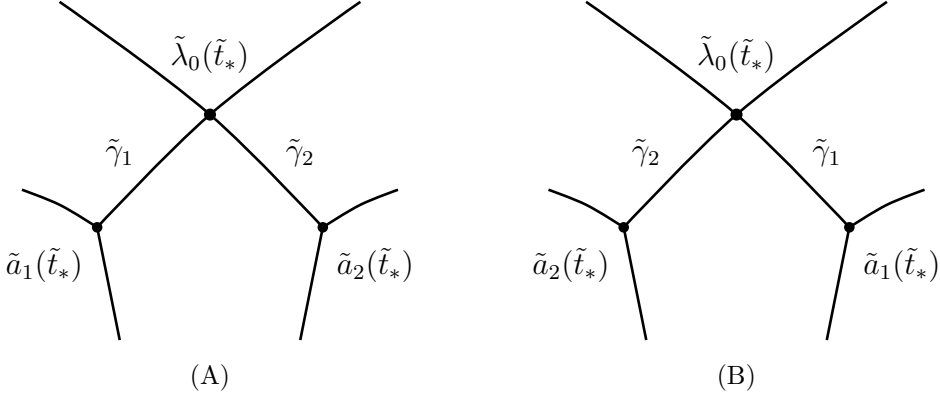
\begin{figure}[h]
\begin{center}
\begin{pspicture}(1.0,0.5)(12.5,5.8)
%
\psset{fillstyle=solid, fillcolor=black}
\pscircle(3.0,4.0){0.08} 
\pscircle(1.5,2.5){0.07}
\pscircle(4.5,2.5){0.07}
\pscircle(10.0,4.0){0.08} 
\pscircle(8.5,2.5){0.07}
\pscircle(11.5,2.5){0.07}
\psset{fillstyle=none}
\rput[c]{0}(3.0,4.8){\large $\tlambda_0(\tt_{\ast})$}
\rput[c]{0}(0.8,2.0){\large $\ta_{1}(\tt_{\ast})$}
\rput[c]{0}(5.2,2.0){\large $\ta_{2}(\tt_{\ast})$} 
\rput[c]{0}(1.8,3.5){\large $\tgamma_1$}
\rput[c]{0}(4.2,3.5){\large $\tgamma_2$} 
\rput[c]{0}(3.0,0.5){(A)} 
\psset{linewidth=1pt}
\pscurve(1.5,2.5)(3,4)(5,5.5)
\pscurve(4.5,2.5)(3,4)(1,5.5)
\pscurve(1.5,2.5)(1,2.8)(0.5,3) 
\pscurve(4.5,2.5)(5,2.8)(5.5,3)
\pscurve(1.5,2.5)(1.6,2)(1.8,1)
\pscurve(4.5,2.5)(4.4,2)(4.2,1)
\rput[c]{0}(10.0,4.8){\large $\tlambda_0(\tt_{\ast})$}
\rput[c]{0}(7.8,2.0){\large $\ta_{2}(\tt_{\ast})$}
\rput[c]{0}(12.2,2.0){\large $\ta_{1}(\tt_{\ast})$} 
\rput[c]{0}(8.8,3.5){\large $\tgamma_2$}
\rput[c]{0}(11.2,3.5){\large $\tgamma_1$} 
\rput[c]{0}(10.0,0.5){(B)}

\psset{linewidth=1pt}
\pscurve(8.5,2.5)(10,4)(12,5.5)
\pscurve(11.5,2.5)(10,4)(8,5.5)
\pscurve(8.5,2.5)(8,2.8)(7.5,3) 
\pscurve(11.5,2.5)(12,2.8)(12.5,3)
\pscurve(8.5,2.5)(8.6,2)(8.8,1)
\pscurve(11.5,2.5)(11.4,2)(11.2,1)
\end{pspicture}
\end{center}
\caption{Two possibilities for adjacent 
Stokes segments of $\SLJ$.} 
\label{fig:P2-transformable-cases}
\end{figure}

For each $k=1,2$, 
the function $t^{(k)}_0(\tt)$ satisfying 
\eqref{eq:top-t0k} is unique if we require that 
$t_0^{(k)}(\tt_{\ast})$ lies on the $P$-Stokes segment 
$\Gamma$ depicted in 
Figure \ref{fig:P2-and-SL2-Stokes-geometry} (P)
(cf.~\cite[Section 2, (2.21)]{KT96}).
In what follows we assume that 
$t_0^{(k)}(\tt_{\ast})$ 
lies on $\Gamma$. 
Then, our choice \eqref{eq:c-P2} 
of the constant $c$ in $\PII$ and 
\eqref{eq:integral-P2-plus} imply that 
\begin{equation}\label{eq:difference-coincide}
\phi_{{\rm II},1}(t^{(k)}_0(\tt)) - 
\phi_{{\rm II},2}(t^{(k)}_0(\tt)) = 
\int_{\tr_1}^{\tr_2}\sqrt{\tF_J^{(1)}(\tt)}~d\tt
= \tphi_{J, 1}(\tt) - \tphi_{J, 2}(\tt)
\end{equation}
holds for both $k=1$ and $2$. 
Especially, we have the equality
$\phi_{{\rm II},1}(t^{(2)}_0(\tt)) = \tphi_{J, 1}(\tt)$
as the case of $k=2$ of \eqref{eq:difference-coincide}.
Since $t^{(2)}_0(\tt_{\ast})$ lies on $\Gamma$, 
we have $t^{(1)}_0(\tt) = t^{(2)}_0(\tt)$ 
due to the uniqueness explained above.
We set $t_0(\tt) = t^{(1)}_0(\tt) = t^{(2)}_0(\tt)$ 
and \eqref{eq:correspondence-of-instantonP2} 
follows from \eqref{eq:top-t0k} for $k=1$. 
Taking a small neighborhood $\tV$ of $t_{\ast}$, 
we may assume that the derivative $dt_0/d\tt$ also 
never vanishes on $\tV$. Thus we obtain 
$t_0(\tt)$ satisfying (i) and (ii) of our main claim. 
Especially, we have 
\begin{equation} \label{eq:dt0-dt}
\sqrt{\tF^{(1)}_{J}(\tt)} = 
\frac{dt_0(\tt)}{d\tt} \sqrt{F_{\rm II}^{(1)}(t_0(\tt))} .
\end{equation}

Next, we construct $x_0(\tx,\tt)$. Set 
\begin{eqnarray} 
\label{eq:tw0-Stokes-segments-gamma1-and-gamma2}
(\gamma_1,\gamma_2) & = &
\begin{cases}
(\gamma, \gamma') & \text{when the case (A) in Figure 
\ref{fig:P2-transformable-cases} happens}, \\
(\gamma', \gamma) & \text{when the case (B) in Figure 
\ref{fig:P2-transformable-cases} happens},
\end{cases} 
\end{eqnarray}
where $\gamma$ and $\gamma'$ are the Stokes segments 
of $\SLII$ (at $t = t_0(\tt_{\ast})$) 
depicted in 
Figure \ref{fig:P2-and-SL2-Stokes-geometry} (SL), 
and denote by $a_1(t)$ (resp., $a_2(t)$) the simple 
turning point of $\SLII$ which is the end-point of 
the Stokes segment $\gamma_1$ (resp., $\gamma_2$)
at $t = t_0(\tt_{\ast})$.
Since $\tt_{\ast}$ lies on a $P$-Stokes curve 
emanating from $\tr_1$, and $t_0(\tt)$ satisfies 
$\phi_{J,1}(\tt) = \phi_{\rm II,1}(t_0(\tt))$, 
the same discussion as in 
\cite[Section 2]{KT96} enables us to construct 
$x_0(\tx,\tt)$ satisfying the following conditions. 
\begin{itemize}
\item %
$x_0(\tx,\tt)$ is holomorphic on a domain 
$\tU_1\times \tV$, where $\tU_1$ is an open neighborhood 
of the Stokes segment $\tgamma_1$ of $\SLJ$, 
and $\p x_0/\p \tx$ never vanishes on $\tU_1\times\tV$. %
\item %
For any $\tt \in \tV$, $x_0(\tx,\tt)$ maps $\tU_1$ 
biholomorphically to an open neighborhood $U_1$ of 
the Stokes segment $\gamma_1$ of $\SLII$. 
\item %
Set 
\begin{equation} \label{eq:ZJ-and-ZII}
\tilde{Z}_{J}(\tx,\tt) = \int_{\tlambda_0(\tt)}^{\tx}
\sqrt{\tQ_{J, 0}(\tx,\tt)}~d\tx,~~~
Z_{{\rm II}}(x,\tt) = 
\int_{\lambda_0(t_0(\tt))}^{x}
\sqrt{Q_{{\rm II}, 0}(x,t_0(\tt))}~dx, 
\end{equation}
where 
the branch of $\sqrt{\tQ_{J, 0}(\tx,\tt)}$ and 
$\sqrt{Q_{{\rm II}, 0}(x,t)}$ are chosen so that  
\begin{equation} \label{eq:branches-P2SL2} 
\int_{\tgamma_k} \sqrt{\tQ_{J}(\tx,\tt)}~d\tx =  
\frac{1}{2} \int_{\tr_{k}}^{\tt} 
\sqrt{\tF^{(1)}_J(\tt)}~d\tt, \quad %
\int_{\gamma_k} \sqrt{Q_{\rm II}(x,t)}~dx = %
\frac{1}{2} \int_{r_k}^{t} 
\sqrt{F^{(1)}_{\rm II}(t)}~dt
\end{equation}
hold for $k=1,2$ (cf.~\eqref{eq:integral-relation}). 
In \eqref{eq:branches-P2SL2} Stokes segments 
are directed from the simple turning point 
to the double turning point. 
Then, the following equalities hold:
\begin{eqnarray} \label{eq:implicit-relation-for-x0}
\tilde{Z}_{J}(\tx,\tt)
& = &
Z_{{\rm II}}(x_0(\tx,\tt),\tt), \\
\label{eq:boundary-condition}
x_0(\tlambda_0(\tt),\tt) & = & 
\lambda_0(t_0(\tt)), \quad 
x_0(\ta_1(\tt),\tt) = a_1(t_0(\tt)).
\end{eqnarray} %
\end{itemize}
It is also shown in \cite[Section 2]{KT96} 
that $x_0(\tx,\tt)$ is the unique holomorphic solution 
(satisfying $x_0(\tlambda_0(\tt),\tt),
({\p x_0}/{\p \tx})(\tlambda_0(\tt),\tt) \ne 0$)
of the following implicit functional equation: 
\begin{eqnarray*} 
Z_{J}(\tx,\tt)^{1/2} = 
Z_{\rm II}(x_0(\tx,\tt),\tt)^{1/2}.
\end{eqnarray*}
Here the branch of $Z_{J}(\tx,\tt)^{1/2}$ and 
$Z_{\rm II}(x,\tt)^{1/2}$ are chosen so that, 
they are positive on $\tgamma_1$ and $\gamma_1$, 
respectively, when $\tt = \tt_{\ast}$.
Note that, since we have assumed that the imaginary part 
of $c$ in \eqref{eq:c-P2} is positive, the real parts of  
$\tphi_{J,1}(\tt)$ and $\phi_{{\rm II},1}(t)$ 
are monotonously 
decreasing along the $P$-Stokes segments $\tGamma$ and 
$\Gamma$, respectively. Then, the equality 
\eqref{eq:branches-P2SL2} shows that the real parts of 
$\tilde{Z}_J(\tx,\tt_{\ast})$ and $Z_{\rm II}(x,\tt_{\ast})$ 
are positive along $\tgamma_1$ and $\gamma_1$,
respectively. 

In view of \eqref{eq:implicit-relation-for-x0}, four Stokes 
curves of $\SLJ$ emanating from $\tlambda_0(\tt)$ are mapped  
to those of $\SLII$ emanating from $\lambda_0(t_0(\tt))$ 
by $x_0(\tx,\tt)$ locally. 
Especially, the Stokes segment $\tgamma_1$ of $\SLJ$ 
is mapped to the Stokes segment $\gamma_1$ of $\SLII$
when $\tt = \tt_{\ast}$. Furthermore, since 
$\p x_0 / \p\tx \ne 0$ at $\tx=\tlambda_0(\tt)$, 
the other Stokes segment $\tgamma_2$ 
is mapped to the Stokes curve emanating from 
$\lambda_0(t_0(\tt_{\ast}))$ 
which comes next to $\gamma_1$ 
in the counter-clockwise (resp., clockwise) order
in the case (A) (resp., (B)), when $\tt = \tt_{\ast}$. 
Thus, our choice \eqref{eq:choice-of-r1-and-r2} 
of the $P$-turning points $r_1$ and $r_2$ of $\PII$ 
entails that $x_0(\tx,\tt_{\ast})$ 
maps $\tgamma_2$ to the Stokes segment $\gamma_2$ of 
$\SLII$ given by 
\eqref{eq:tw0-Stokes-segments-gamma1-and-gamma2} 
near $\tx=\tlambda_0(\tt_{\ast})$.

Since our choice \eqref{eq:c-P2} of the constant
$c$ in $\PII$ also ensures the equality 
$\phi_{J,2}(\tt) = \phi_{\rm II,2}(t_0(\tt))$, 
the same discussion as in \cite[Section 2]{KT96} 
again enables us to show that $x_0(\tx,\tt)$ 
is also holomorphic at the simple turning point 
$\ta_2(\tt)$ and satisfies
\begin{equation}
x_0(\ta_2(\tt),\tt) = a_2(t_0(\tt)).
\end{equation}
Thus we have constructed $x_0(\tx,\tt)$ satisfying 
the desired properties (iii) and (iv) of our main theorem. 

\subsection{Transformation near the double turning point}
\label{section:transformation-at-double}


In this section we follow the discussion 
given in \cite[Section 4]{KT98}. Namely, 
with the aid of Theorem 
\ref{thm:transformation-at-double-turning-point}, 
we construct a pair of formal series 
$x^{\rm pre}(\tx,\tt,\eta)$ and $t^{\rm pre}(\tt,\eta)$
which transforms $\SLJ$ and the deformation equation 
$\DJ$ to $\SLII$ and $\DII$. 

Let us first fix the correspondence of the parameters: 
For a given pair of parameters
$(\talpha,\tbeta) = (\sum_{n=0}^{\infty}\eta^{-n}\talpha_n, 
\sum_{n=0}^{\infty}\eta^{-n}\tbeta_n)$ of 
$(\tlambda_J,\tnu_J)$ satisfying \eqref{eq:genericity}, 
we choose $(A(\eta),B(\eta)) = (\sum_{n=0}^{\infty}\eta^{-n}A_n, 
\sum_{n=0}^{\infty}\eta^{-n}B_n)$ in 
\eqref{eq:sol-of-Hcan} and 
$(\alpha,\beta) = (\sum_{n=0}^{\infty}\eta^{-n}\alpha_n, 
\sum_{n=0}^{\infty}\eta^{-n}\beta_n)$ in 
$(\lambda_{\rm II},\nu_{\rm II})$ so that 
\begin{equation} \label{eq:invariance-of-residues}
E_{\rm II}(\alpha,\beta) = -16 A(\eta) B(\eta) = 
\tE_J(\talpha,\tbeta)
\end{equation}
holds. Here $\tE_{J}(\talpha,\tbeta)$ and 
$E_{\rm II}(\alpha,\beta)$ be the formal power series 
defined in \eqref{eq:EJ}.
Lemma \ref{lemma:E} guarantees that 
such a choice of parameters is possible. 
Then the discussion in Section \ref{section:zJ-and-sJ}
enables us to construct formal series 
$\tz_J(\tx,\tt,\eta)$ and $\ts_J(\tt,\eta)$ 
(resp., $z_{\rm II}(\tx,\tt,\eta)$ and 
$s_{\rm II}(\tt,\eta)$) satisfying the properties 
in Theorem \ref{thm:transformation-at-double-turning-point}
for such a given $(A(\eta),B(\eta))$; that is, 
\begin{eqnarray}
\sigma(\ts_J(\tt,\eta);A(\eta),B(\eta)) & = & \eta^{1/2}
\tz_J(\tlambda_J(\tt,\eta;\talpha,\tbeta),\tt,\eta), \\
\sigma(s_{\rm II}(t,\eta);A(\eta),B(\eta)) & = & \eta^{1/2}
z_{\rm II}(\lambda_{\rm II}(t,\eta;\alpha,\beta),t,\eta).
\end{eqnarray}
Similarly to \cite[Section 4]{KT98}, define 
\begin{eqnarray} \label{eq:x---pre}
x^{\rm pre}(\tx,\tt,\eta) & = & z_{\rm II}^{-1}
(\tz_{J}(\tx,\tt,\eta), s_{J}(\tt,\eta),\eta),  \\
t^{\rm pre}(\tt,\eta) & = & s_{\rm II}^{-1}
(\ts_{J}(\tt,\eta),\eta). \label{eq:t---pre}
\end{eqnarray}
Then, each coefficient of the formal power series 
$x^{\rm pre}(\tx,\tt,\eta)$ is holomorphic in 
$\tx$ near $\tx = \tlambda_0(\tt)$ and also in $\tt$ on $\tV$, 
and each coefficient of $t^{\rm pre}(\tt,\eta)$  
is holomorphic in $\tt$ on $\tV$.
Furthermore, $x^{\rm pre}(\tx,\tt,\eta)$ 
and $t^{\rm pre}(\tt,\eta)$ have the property 
of alternating parity; that is, 
if we denote by 
$\{x^{\rm pre}_{j/2}(\tx,\tt,\eta)\}_{j=0}^{\infty}$
(resp., $\{t^{\rm pre}_{j/2}(\tt,\eta)\}_{j=0}^{\infty}$)
the coefficient of $\eta^{-j/2}$ in the 
formal series \eqref{eq:x---pre} 
(resp., \eqref{eq:t---pre}), then 
the following conditions hold.
\begin{itemize}
\item %
$x^{\rm pre}_0(\tx,\tt)$ and $t^{\rm pre}_0(\tt)$ 
are independent of $\eta$, 
\item %
$x^{\rm pre}_{1/2}$ and 
$t^{\rm pre}_{1/2}$ vanish identically,  
\item %
For $j \ge 2$, the $\eta$-dependence of 
$x^{\rm pre}_{j/2}(\tx,\tt,\eta)$ 
and $t^{\rm pre}_{j/2}(\tt,\eta)$
are only through instanton terms 
$\exp(\ell \tPhi_{J}(\tt,\eta))$
($\ell = j-2-2m$ with $0 \le m \le j-2$).
\end{itemize}
\begin{lem} \label{lemma:x0-and-t0}
The top terms $x^{\rm pre}_0(\tx,\tt)$ and 
$t^{\rm pre}_0(\tt)$ coincide with $x_0(\tx,\tt)$ and 
$t_0(\tt)$ constructed in Section \ref{section:top-terms}, 
respectively:
\begin{equation}
x^{\rm pre}_0(\tx,\tt) = x_0(\tx,\tt), \quad
t^{\rm pre}_0(\tt) = t_0(\tt).
\end{equation} 
\end{lem}
\begin{proof}
It follows from \eqref{eq:sJ0} and the normalizations 
\eqref{eq:tlidephiJ} and \eqref{eq:phiII} that, 
$t^{\rm pre}_0(\tt)$ here satisfies 
\[
\tphi_{J,1}(\tt) = \phi_{\rm II,1}(t^{\rm pre}_0(\tt)).
\] 
Hence it coincides with $t_0(\tt)$ constructed 
in Section \ref{section:top-terms}. Furthermore, 
by choosing a branch of the square root in \eqref{eq:zJ0} 
appropriately, we can show that $x^{\rm pre}_0(\tx,\tt)$ 
satisfies the following conditions in a neighborhood of 
the Stokes segment $\tgamma_1$ of $\SLJ$:
\[
x^{\rm pre}_0(\tlambda_0(\tt),\tt) = \lambda_0(t_0(\tt)),~~ 
(\p x^{\rm pre}_0/\p \tx)(\tlambda_0(\tt),\tt)) \ne 0,
\] 
\[
Z_{J}(\tx,\tt)^{1/2} = 
Z_{\rm II}(x^{\rm pre}_0(\tx,\tt),\tt)^{1/2}.
\]
Here $Z_J$ and $Z_{\rm II}$ are given in 
\eqref{eq:ZJ-and-ZII}, and the branch of 
$Z_{J}(\tx,\tt)^{1/2}$ and $Z_{\rm II}(x,\tt)^{1/2}$ 
are chosen so that they are positive on 
$\tgamma_1$ and $\gamma_1$. Thus, the top term 
$x^{\rm pre}_0(\tx,\tt)$ of $x^{\rm pre}(\tx,\tt,\eta)$ 
(defined by choosing an appropriate branch of \eqref{eq:zJ0}) 
also coincides with $x_0(\tx,\tt)$ constructed in 
Section \ref{section:top-terms}. 
\end{proof}

Therefore, the top terms 
of $x^{\rm pre}(\tx,\tt,\eta)$ and $t^{\rm pre}(\tt,\eta)$
enjoy the desired properties. Moreover, 
they give a local equivalence between 
$\SLJ$ and $\SLII$ together 
with their deformation equations 
$\DJ$ and $\DII$ near $\tx = \tlambda_0(\tt)$ 
in the following sense. 
\begin{prop} [{\cite[Section 4]{KT98}}]
\label{prop:transformation-at-double-TP}
The following equalities hold near 
$\tx = \tlambda_0(\tt)$ and $\tt \in \tV$: 
\begin{eqnarray} 
\tS_{J,{\rm odd}}(\tx,\tt,\eta) & = &
\left( \frac{\p x^{\rm pre}}{\p\tx}(\tx,\tt,\eta) \right)
S_{\rm II, odd}\left( x^{\rm pre}(\tx,\tt,\eta), 
t^{\rm pre}(\tt,\eta), \eta \right), \hspace{-1.9em}
\label{eq:xpre-relation1} \\[+.5em]
\frac{\p x^{\rm pre}}{\p \tt}(\tx,\tt,\eta) & = &
\tA_{J}(\tx,\tt,\eta)\frac{\p x^{\rm pre}}{\p \tx} - 
A_{\rm II}(x^{\rm pre}(\tx,\tt,\eta), 
t^{\rm pre}(\tt,\eta),\eta) \frac{\p t^{\rm pre}}{\p \tt}.
\hspace{-1.7em}
\label{eq:xpre-relation2}
\end{eqnarray} %
\end{prop}
It follows from \eqref{eq:xpre-relation1} and 
\eqref{eq:xpre-relation2} 
that, if a WKB solution $\psi_{\rm II}(x,t,\eta)$ of $\SLII$ 
also solves the deformation equation $\DII$, then 
\[
\tilde{\psi}_J(\tx,\tt,\eta) = 
\left( \frac{\p x^{\rm pre}}{\p \tx}(\tx,\tt,\eta) \right)^{-1/2}
\psi_{\rm II}\bigl( x^{\rm pre}(\tx,\tt,\eta), 
t^{\rm pre}(\tt,\eta),\eta \bigr)
\]
is a WKB solution of $\SLJ$ which also satisfies $\DJ$
simultaneously near $\tx=\tlambda_0(\tt)$ 
(cf.~\cite[Proposition 3.1]{KT98}). 

Therefore, the formal series defined by \eqref{eq:x---pre}
and \eqref{eq:t---pre} are ``almost the required" one. 
However, these formal series $x^{\rm pre}(\tx,\tt,\eta)$ 
and $t^{\rm pre}(\tt,\eta)$ may not be a desired one;
that is, each coefficient of $x^{\rm pre}(\tx,\tt,\eta)$
may not be holomorphic near a pair of simple turning points 
$\ta_1$ and $\ta_2$, due to the following reason.  

The equality \eqref{eq:xpre-relation1} tells us that 
the coefficient $x^{\rm pre}_{j/2}(\tx,\tt,\eta)$ ($j \ge 1$) 
of $x^{\rm pre}(\tx,\tt,\eta)$ satisfies 
the following 
linear inhomogeneous differential equation:
\begin{equation} \label{eq:xj/2}
S_{-1}(x_0,t_0) \frac{\p x^{\rm pre}_{j/2}}{\p \tx} + 
\frac{\p x_0}{\p \tx} \frac{\p S_{-1}}{\p x}(x_0,t_0) 
\hspace{+.1em} x^{\rm pre}_{j/2} 
+ \frac{\p x_0}{\p \tx} \frac{\p S_{-1}}{\p t}(x_0,t_0) 
\hspace{+.1em} t^{\rm pre}_{j/2} 
= R_{j/2}(\tx,\tt).
\end{equation}
Here $S_{-1}(x,t) = \sqrt{Q_{\rm II,0}(x,t)}$ be the 
top term of $S_{\rm II, odd}(x,t,\eta)$ and 
$R_{j/2}$ consists of the terms given by 
$x^{\rm pre}_0, \dots, x^{\rm pre}_{(j-1)/2}$. 
Since the coefficients of 
$\tS_{J, \rm odd}(\tx,\tt,\eta)$ are singular at simple 
turning points, the coefficient $R_{j/2}$ may be singular 
at $\tx=\ta_1$ and $\tx=\ta_2$, that is, 
$x^{\rm pre}_{j/2}$ is not holomorphic there in general. 

Recall that the transformation series 
$\ts_J(\tt,\eta)$ and $s_{\rm II}(t,\eta)$ 
contain infinitely many free parameters  
as explained in Section \ref{section:zJ-and-sJ}. 
Thus the formal series $t^{\rm pre}(\tt,\eta)$ also 
has free parameters, which will be denoted by $C_n$, 
and we write 
\begin{equation} \label{eq:free-parameter-C}
C(\eta) = \sum_{n=1}^{\infty} \eta^{-n} C_n. 
\end{equation} 
Since the free parameters are contained in $\ts_J(\tt,\eta)$ 
and $s_{\rm II}(t,\eta)$ additively 
(cf.~Section \ref{section:zJ-and-sJ}), 
the formal series $t^{\rm pre}(\tt,\eta)$ 
contains the free parameters in the following manner:
\begin{equation} \label{eq:C-and-tpre}
\ts_J(\tt,\eta) = 
s_{\rm II}(t^{\rm pre}(\tt,\eta),\eta) + C(\eta).
\end{equation}
In the subsequent subsections, we will show that, 
by appropriately choosing the free parameters 
$C_n$'s (i.e., correct choices of $t^{\rm pre}_{j/2}$'s 
appearing in \eqref{eq:xj/2}), 
$x^{\rm pre}_{j/2}$'s become holomorphic in 
neighborhoods of both simple turning points 
$\tx=\ta_1$ and $\ta_2$.
The condition for $C_n$'s together with the 
constraint \eqref{eq:invariance-of-residues} 
between the parameters $(\talpha,\tbeta)$ 
and ($\alpha,\beta$) gives a correspondence 
between 2-parameter solutions $(\tlambda_J,\tnu_J)$ 
of $\PJ$ and $(\lambda_{\rm II},\nu_{\rm II})$ of $\PII$.

\subsection{Matching of two transformations}
\label{section:matching-P2}

With the aid of the idea of \cite{KT98}, 
we show that, by appropriately choosing the 
free parameters $C_n$, the coefficients 
$x^{\rm pre}_{j/2}$ of the formal series 
$x^{\rm pre}(\tx,\tt,\eta)$
become holomorphic in a neighborhood of one of 
the two simple turning points $\tx=\ta_1$ 
and $\tx = \ta_2$. 

The following lemma can be shown by using the same 
discussion as in \cite{AKT91} and \cite{KT98}. 
\begin{lem}
[{cf. \cite[Lemma 2.2]{AKT91}, \cite[Sublemma 4.1]{KT98}}]
\label{lemma:transformation-at-simple-TP}
For each $k=1,2$, there exist an open neighborhood $\tU'_k$
of $\tx=\ta_k(\tt)$ and a formal series 
\begin{equation}
y^{(k)}(\tx,\tt,\eta) = \sum_{j = 0}^{\infty}
\eta^{-j/2} y^{(k)}_{j/2}(\tx,\tt,\eta) 
\end{equation} 
satisfying the following conditions.
\begin{enumerate}[\upshape (i)]
\item %
Each coefficient $y^{(k)}_{j/2}(\tx,\tt,\eta)$ is 
holomorphic in $\tU'_k \times \tV$. %
\item %
The top term $y^{(k)}_0(\tx,\tt)$ 
is free from $\eta$ and  
$\p y^{(k)}_0/\p\tx$ never vanishes 
on $\tU'_k \times \tV$.
\item %
$y^{(k)}_0(\tx,\tt)$ satisfies 
$y^{(k)}_0(\ta_k(\tt),\tt) = a_k(t_0(\tt))$
and maps the Stokes segment 
$\tgamma_k$ of $\SLJ$ to the Stokes segment 
$\gamma_k$ of $\SLII$ locally near $\tx=\ta_k(\tt)$. %
\item %
$y^{(k)}_{1/2}$ vanishes identically.  
\item %
For $j \ge 2$, the $\eta$-dependence of 
$y^{(k)}_{j/2}(\tx,\tt,\eta)$ 
is only through instanton terms 
$\exp(\ell \tPhi_{J}(\tt,\eta))$
($\ell = j-2-2m$ with $0 \le m \le j-2$). 
\item %
The equalities 
\begin{eqnarray} \label{eq:yk-relation} 
& & \\ \nonumber
\tS_{J, {\rm odd}}(\tx,\tt,\eta) & = & 
\left( \frac{\p y^{(k)}}{\p\tx}(\tx,\tt,\eta) \right)
S_{\rm II, odd}\left( y^{(k)}(\tx,\tt,\eta), 
t^{\rm pre}(\tt,\eta), C(\eta), \eta \right), \\
\frac{\p y^{(k)}}{\p\tt} & = & \tA_J(\tx,\tt,\eta)
\frac{\p y^{(k)}}{\p \tx} - 
A_{\rm II}(y^{(k)}(\tx,\tt,\eta),t^{\rm pre}(\tt,\eta),\eta)
\frac{\p t^{\rm pre}}{\p\tt}
\label{eq:yk-relation2}
\end{eqnarray}
hold on $\tU'_k \times \tV$. Here 
$t^{\rm pre}(\tt,\eta)$ is given in \eqref{eq:t---pre}.
\end{enumerate}
\end{lem}
The top term $y^{(k)}_0(\tx,\tt)$ is fixed as the 
unique holomorphic function near $\tx=\ta_k(\tt)$
satisfying 
\begin{equation}
\sqrt{\tQ_{J,0}(\tx,\tt)} = 
\left( \frac{\p y^{(k)}_0}{\p \tx} (\tx,\tt) \right)
\sqrt{Q_{\rm II,0}(y^{(k)}_0(\tx,\tt),\tt)}
\end{equation}
at $\tx=\ta_k(\tt)$ and the condition (iii) 
in Lemma \ref{lemma:transformation-at-simple-TP}.
Since $x_0(\tx,\tt)$ constructed in Section 
\ref{section:top-terms} also satisfies the conditions 
for both $k=1,2$, we can conclude that 
\begin{equation} \label{eq:coincidence-of-top-terms}
y^{(1)}_0(\tx,\tt) = y^{(2)}_0(\tx,\tt) = x_0(\tx,\tt).
\end{equation}

Now we try to adjust the free parameters $C_n$ 
that remain in $t^{\rm pre}(\tt,\eta)$ as described 
in \eqref{eq:C-and-tpre} so that the higher order terms 
of transformations $x^{\rm pre}(\tx,\tt,\eta)$ and 
$y^{(1)}(\tx,\tt,\eta)$ constructed above coincide. 
This is a kind of ``matching problem" 
which has been used in constructions of WKB theoretic 
transformations as in 
\cite{AKT91}, \cite{KT96}, \cite{KT98}, etc.

In this subsection we denote by $y^{\rm pre}(\tx,\tt,\eta)$
the formal series $y^{(1)}(\tx,\tt,\eta)$, and write
\begin{equation} \label{eq:ypre}
y^{\rm pre}(\tx,\tt,\eta) = \sum_{j=0}^{\infty}
\eta^{-j/2} y^{\rm pre}_{j/2}(\tx,\tt,\eta) ~~
\left( = y^{(1)}(\tx,\tt,\eta) \right).
\end{equation}
We note that the coefficients of formal series 
$y^{(k)}(\tx,\tt,\eta)$ are holomorphic along 
each Stokes curve emanating from $\ta_k(\tt)$ 
(cf.~\cite[Appendix A.2]{AKT91}). 
Thus, there exists a domain in the 
$\tx$-plane on which both of the coefficients of 
formal series $x^{\rm pre}(\tx,\tt,\eta)$ and 
$y^{\rm pre}(\tx,\tt,\eta)$ are holomorphic 
since the Stokes segment $\tgamma_1$ connects 
the simple turning point $\ta_1(\tt)$ and 
the double turning point $\tlambda_0(\tt)$ 
of $\SLJ$ when $\tt=\tt_{\ast}$. 
In what follows we suppose that $\tx$ 
lies on this domain. To attain the matching, 
we introduce the following functions:
\begin{eqnarray}
{\cal R}(x,t,\eta) & = & \int_{a_1(t)}^x \eta^{-1} 
S_{\rm II, odd}(x,t,\eta)~dx,  
\label{eq:Rxteta} \\
{\cal F}(\tx,\tt,\eta) & = & 
{\cal R}(x^{\rm pre}(\tx,\tt,\eta),t^{\rm pre}(\tt,\eta),\eta), 
\label{eq:calF} \\
{\cal G}(\tx,\tt,\eta) & = & 
{\cal R}(y^{\rm pre}(\tx,\tt,\eta),t^{\rm pre}(\tt,\eta),\eta).
\label{eq:calG}
\end{eqnarray}
Due to the factor $\eta^{-1}$ in \eqref{eq:Rxteta}, 
${\cal F}$ and ${\cal G}$ become formal series 
starting from $\eta^{0}$. It is clear from the definition 
\eqref{eq:xpre-relation1} and \eqref{eq:yk-relation} 
that we find 
\begin{equation} \label{eq:x-indep}
\frac{\p ({\cal F} - {\cal G})}{\p \tx} = 
\eta^{-1} \tS_{J,{\rm odd}}(\tx,\tt,\eta) - 
\eta^{-1} \tS_{J,{\rm odd}}(\tx,\tt,\eta) = 0.
\end{equation}
Furthermore, using \eqref{eq:t-derivative-of-Sodd}, 
\eqref{eq:xpre-relation1} and \eqref{eq:xpre-relation2}, 
we have
\begin{eqnarray*}
\frac{\p {\cal F}}{\p \tt} 
& = & \eta^{-1}S_{\rm II, odd}(x^{\rm pre}(\tx,\tt,\eta),
t^{\rm pre}(\tt,\eta),\eta) \\
& \times & 
\left( \frac{\p x^{\rm pre}}{\p\tt}(\tx,\tt,\eta) + 
A_{\rm II}(x^{\rm pre}(\tx,\tt,\eta),
t^{\rm pre}(\tt,\eta),\eta) 
\frac{\p t^{\rm pre}}{\p\tt}(\tt,\eta) \right) \\
& = & \eta^{-1} \tA_J(\tx,\tt,\eta) 
\frac{\p x^{\rm pre}}{\p \tx}(\tx,\tt,\eta)
S_{\rm II, odd}(x^{\rm pre}(\tx,\tt,\eta),
t^{\rm pre}(\tt,\eta),\eta) \\
& = & \eta^{-1} \tA_J(\tx,\tt,\eta) \hspace{+.1em}
\tS_{J, {\rm odd}}(\tx,\tt,\eta)
\end{eqnarray*}
by a straightforward computation. In the same way we have
\begin{eqnarray*}
\frac{\p {\cal G}}{\p \tt} = \eta^{-1} 
\tA_J(\tx,\tt,\eta) \hspace{+.1em}
\tS_{J, {\rm odd}}(\tx,\tt,\eta).
\end{eqnarray*}
Therefore, 
\begin{equation} \label{eq:y-indep}
\frac{\p ({\cal F}-{\cal G})}{\p \tt} = 0.
\end{equation}
Combining \eqref{eq:x-indep} and \eqref{eq:y-indep}, we conclude
\begin{equation} \label{eq:xy-indep}
{\cal F}-{\cal G} = \sum_{j=0}^{\infty} \eta^{-j/2} 
{\cal I}_{j/2}
\end{equation}
holds with genuine constants ${\cal I}_{j/2}$.

Let us prove the following statement $(\ast)_j$ 
for any $j$ by the induction on $j$:
\begin{eqnarray*}
(\ast)_j~~
\text{A correct choice of $t^{\rm pre}_{j/2}$ 
entails the vanishing} \\
\text{of ${\cal I}_{j/2}$ and coincidence of 
$x^{\rm pre}_{j/2}$ and $y^{\rm pre}_{j/2}$.} \hspace{+1.8em}
\end{eqnarray*}

As we have shown in Section \ref{section:top-terms}
and \eqref{eq:coincidence-of-top-terms}, 
$(\ast)_0$ holds. 
Since $x^{\rm pre}_{1/2} = y^{\rm pre}_{1/2} = 0$ and 
$t^{\rm pre}_{1/2} = 0$, $(\ast)_1$ is also valid. 
Let us suppose $j \ge 2$ and $(\ast)_k$ holds 
for all $k < j$ and show $(\ast)_j$. 
It follows from the definition 
\eqref{eq:calF} and \eqref{eq:calG} and 
the induction hypothesis that 
\begin{equation} \label{eq:Cj/2}
{\cal I}_{j/2} = S_{-1}(x_0,t_0) 
(x^{\rm pre}_{j/2} - y^{\rm pre}_{j/2})
\end{equation}
holds. Here $S_{-1}(x,t)$ is the top term of 
$S_{\rm II, odd}(x,t,\eta)$. 
On the other hand, as we have seen in \eqref{eq:xj/2},
the functions $x^{\rm pre}_{j/2}$ and 
$y^{\rm pre}_{j/2}$ 
satisfy linear inhomogeneous differential equations 
\begin{eqnarray} \label{eq:diff-eq-for-xj/2}
{L}x^{\rm pre}_{j/2} & = & 
R(x^{\rm pre}_0,\dots,x^{\rm pre}_{(j-1)/2},
t^{\rm pre}_0,\dots,t^{\rm pre}_{(j-1)/2}), \\
{L}y^{\rm pre}_{j/2} & = & 
R(y^{\rm pre}_0,\dots,y^{\rm pre}_{(j-1)/2},
t^{\rm pre}_0,\dots,t^{\rm pre}_{(j-1)/2}), 
\label{eq:diff-eq-for-yj/2}
\end{eqnarray}
where $L$ is a differential operator defined by 
\begin{eqnarray}
Lw & = & S_{-1}(x_0,t_0) \frac{\p w}{\p \tx} + 
\frac{\p x_0}{\p\tx} \frac{\p S_{-1}}{\p x}(x_0,t_0) 
\hspace{+.2em} w  \\
& + & \nonumber
\frac{\p x_0}{\p\tx} \frac{\p S_{-1}}{\p t}(x_0,t_0) 
\hspace{+.2em} t^{\rm pre}_{j/2},
\end{eqnarray}
and the right-hand side of \eqref{eq:diff-eq-for-xj/2}
(resp., \eqref{eq:diff-eq-for-yj/2}) is a function 
determined by $x^{\rm pre}_{j'/2}$ 
(resp., $y^{\rm pre}_{j'/2}$) and 
$t^{\rm pre}_{j'/2}$ with $j'\le j-1$. 
The induction hypothesis implies that 
\[
R(x^{\rm pre}_0,\dots,x^{\rm pre}_{(j-1)/2},
t^{\rm pre}_0,\dots,t^{\rm pre}_{(j-1)/2}) = 
R(y^{\rm pre}_0,\dots,y^{\rm pre}_{(j-1)/2},
t^{\rm pre}_0,\dots,t^{\rm pre}_{(j-1)/2}).
\] 
Moreover, since $x^{\rm pre}_{j/2}$ is non-singular 
near $\tx=\tlambda_0(\tt)$, the right-hand sides 
of \eqref{eq:diff-eq-for-xj/2} and 
\eqref{eq:diff-eq-for-yj/2} must be holomorphic 
at $\tx=\tlambda_0(\tt)$. 
The method of variation of constants shows that 
$y^{\rm pre}_{j/2}$ has an at most simple pole near 
$\tx=\tlambda_0(\tt)$,
and has the form 
\begin{equation} \label{eq:singularity-of-yj/2}
y^{\rm pre}_{j/2}(\tx,\tt,\eta) = 
\frac{d_{j/2}(\tt,\eta)-t^{\rm pre}_{j/2}(\tt,\eta)}
{2(x_0(\tx,\tt)-\lambda_0(t_0(\tt)))} + 
\text{(regular function at $\tx=\lambda_0(\tt)$)}.
\end{equation}
Here $d_{j/2}(\tt,\eta)$ is determined by 
$x^{\rm pre}_{j'/2}$ and $t^{\rm pre}_{j'/2}$ with 
$j'\le j-1$ and, in particular, independent of 
$t^{\rm pre}_{j/2}$. 
Substituting \eqref{eq:singularity-of-yj/2} into 
\eqref{eq:Cj/2} and taking the limit 
$\tx \rightarrow \tlambda_0(\tt)$, we obtain 
\begin{equation} \label{eq:free-parameter-and-t}
\frac{1}{2} \sqrt{F^{(1)}_{\rm II}(t_0(\tt))}~ 
\left( t^{\rm pre}_{j/2}(\tt,\eta)-d_{j/2}(\tt,\eta) \right) 
 = {\cal I}_{j/2}.
\end{equation}
Here we have used the equalities \eqref{eq:RJ-at-lambda0},
\eqref{eq:boundary-condition}, 
\[
S_{-1}(x_0,t_0) = (x_0(\tx,\tt) - \lambda_0(t_0(\tt))) 
\sqrt{R_{\rm II}(x_0(\tx,\tt),t_0(\tt))},
\]
and the fact that $x^{\rm pre}_{j/2}(\tx,\tt,\eta)$ 
is holomorphic at $\tx=\tlambda_0(\tt)$. 
Again we emphasize that ${\cal I}_{j/2}$ is 
independent of $\tt$.

Here we suppose that $j$ is even, and write $j=2n$ 
($n \ge 1$). Then, in view of \eqref{eq:C-and-tpre}, 
we can conclude that the free parameter $C_n$ 
remains in $t^{\rm pre}_{j/2}(\tt,\eta) = 
t^{\rm pre}_{n}(\tt,\eta)$ in the form 
\begin{equation} \label{eq:free-parameter-and-tj/2}
t^{\rm pre}_{n}(\tt,\eta) = 
\left(\frac{d s_0}{d t}(t_0(\tt)) \right)^{-1} C_n + 
N(\tt,\eta).
\end{equation}
Here $s_0(t)$ is the top term \eqref{eq:sJ0} of 
the formal series $s_{\rm II}(t,\eta)$ and hence 
\[
\frac{d s_0}{d t}(t_0(\tt)) = \frac{1}{2}
\sqrt{F^{(1)}_{\rm II}(t_0(\tt))} 
\]
is non-zero, at least when $\tt = \tt_{\ast}$. 
The term $N(\tt,\eta)$ in \eqref{eq:free-parameter-and-tj/2}
consists of terms which are independent of $C_n$.
Thus, \eqref{eq:free-parameter-and-t} and 
\eqref{eq:free-parameter-and-tj/2} show that 
a suitable choice of the free parameter 
$C_n$ makes ${\cal I}_{j/2} = {\cal I}_{n}$ vanish. 
Hence \eqref{eq:Cj/2} implies 
\begin{equation} \label{eq:xj/2=yj/2}
x^{\rm pre}_{j/2}(\tx,\tt,\eta) = 
y^{\rm pre}_{j/2}(\tx,\tt,\eta),
\end{equation}
that is, the claim $(\ast)_j$. 

Next we consider the case that $j$ is odd. In this case, 
due to the property of alternating parity, 
${\cal I}_{j/2}$ must contain only odd instanton terms, 
and hence it never contains constant terms. Thus 
${\cal I}_{j/2}$ must vanish, and \eqref{eq:Cj/2} 
implies \eqref{eq:xj/2=yj/2}. 

Thus the induction proceeds and the claim $(\ast)_{j}$ 
is valid for every $j$. Otherwise stated, the formal series 
$x^{\rm pre}(\tx,\tt,\eta)$ and $y^{\rm pre}(\tx,\tt,\eta)$ 
coincide after the correct choice of free parameters:
\begin{equation} \label{eq:xpre-equal-ypre}
x^{\rm pre}(\tx,\tt,\eta) = y^{\rm pre}(\tx,\tt,\eta) ~
\left( = y^{(1)}(\tx,\tt,\eta) \right).
\end{equation}
Since the all free parameters in $t^{\rm pre}(\tt,\eta)$ 
have been fixed, the correspondence of parameters 
between $(\talpha(\eta),\tbeta(\eta))$ and 
$(\alpha(\eta),\beta(\eta))$ is also fixed. 
In what follows we always assume 
that parameters are chosen so that 
\eqref{eq:xpre-equal-ypre} holds, and denote by 
\begin{equation} \label{eq:tpre-is-t}
t(\tt,\eta) = \sum_{j=0}^{\infty}\eta^{-j/2}t_{j/2}(\tt,\eta)
\end{equation}
the formal series 
$t^{\rm pre}$ after the correct choice of free parameters. 
In the next subsection, we will show that 
the formal series \eqref{eq:xpre-equal-ypre}
also coincides with $y^{(2)}(\tx,\tt,\eta)$ and 
consequently the coefficients of \eqref{eq:xpre-equal-ypre}
are also holomorphic near the simple turning point 
$\tx=\ta_2(\tt)$. 

\subsection{Transformation near the pair of two
simple turning points and the 
transformation of 2-parameter solutions}
\label{section:transformation-near-two-stps}

Finally, in this subsection we show that the 
formal series $y^{(1)}(\tx,\tt,\eta)$ and 
$y^{(2)}(\tx,\tt,\eta)$ constructed in Lemma 
\ref{lemma:transformation-at-simple-TP} coincide.
Our choice \eqref{eq:c-P2} of the constant 
$c$ in $\PII$ and $\SLII$ enables 
us to show the following claim.

\begin{prop}\label{prop:y1=y2}
The formal series $y^{(1)}(\tx,\tt,\eta)$ and 
$y^{(2)}(\tx,\tt,\eta)$ constructed in 
Lemma \ref{lemma:transformation-at-simple-TP} coincide:
\begin{equation} \label{eq:Y1-is-Y2}
y^{(1)}(\tx,\tt,\eta) = y^{(2)}(\tx,\tt,\eta). 
\end{equation} 
Consequently, the coefficients of 
$y^{(1)}(\tx,\tt,\eta)$ and $y^{(2)}(\tx,\tt,\eta)$ 
are holomorphic in $\tx$ on a domain containing 
the pair of two simple turning points 
$\tx=\ta_1(\tt)$ and $\ta_2(\tt)$ of $\SLJ$.
\end{prop}
\begin{proof}
We assume that the case (A) in Figure 
\ref{fig:P2-transformable-cases} happens. The discussion 
given here is applicable to the case (B) in Figure 
\ref{fig:P2-transformable-cases}.
Moreover, we will show the equality 
\eqref{eq:Y1-is-Y2} with $\tt$ being fixed at 
$\tt_{\ast}$. This is just for the sake of clarity, 
and our proof is also valid for any $\tt$ 
in a neighborhood $\tV$ of $\tt_{\ast}$. 
(We may take a smaller neighborhood $\tV$, if necessary.) 

Due to the equality \eqref{eq:xpre-equal-ypre} 
the coefficients of $y^{(1)}(\tx,\tt,\eta)$
are holomorphic in $\tx$ near $\tlambda_0(\tt)$. 
Therefore, there exists a domain $\tU'$ containing 
a part of the Stokes segment $\tgamma_2$ on which the 
both coefficients $y^{(1)}(\tx,\tt,\eta)$ and 
$y^{(2)}(\tx,\tt,\eta)$ are holomorphic 
because $\tgamma_2$ connects $\ta_2(\tt_{\ast})$
and $\tlambda_0(\tt_{\ast})$.
In the proof of Proposition \ref{prop:y1=y2} we 
assume that $\tx$ lies on the domain $\tU'$. 
Note that the top terms $y^{(1)}_0(\tx,\tt)$ and 
$y^{(2)}_0(\tx,\tt)$ coincide and are holomorphic 
in the domain $\tU'$ as we have seen in 
\eqref{eq:coincidence-of-top-terms}.

\begin{figure}
\begin{center}
\begin{pspicture}(0,1)(7,6)
%
\psset{fillstyle=solid, fillcolor=black}
\pscircle(3.5,4.4){0.09} 
\pscircle(1.5,2.5){0.08}
\pscircle(5.5,2.5){0.08}
\psset{fillstyle=none}
\psset{linewidth=1.5pt}
\pscurve(5.5,2.5)(5.45,2.4)(5.3,2.5)(5.2,2.6)(5.1,2.5)
(5.0,2.4)(4.9,2.5)(4.8,2.6)(4.7,2.5)(4.6,2.4)(4.5,2.5)
(4.4,2.6)(4.3,2.5)(4.2,2.4)(4.1,2.5)(4,2.6)(3.9,2.5)
(3.8,2.4)(3.7,2.5)(3.6,2.6)(3.5,2.5)(3.4,2.4)(3.3,2.5)
(3.2,2.6)(3.1,2.5)(3,2.4)(2.9,2.5)(2.8,2.6)(2.7,2.5)
(2.6,2.4)(2.5,2.5)(2.4,2.6)(2.3,2.5)(2.2,2.4)(2.1,2.5)
(2,2.6)(1.9,2.5)(1.8,2.4)(1.7,2.5)(1.6,2.55)(1.5,2.5)
\rput[c]{0}(3.5,5.5){\large $\tlambda_0(\tt_{\ast})$}
\rput[c]{0}(1.0,1.5){\large $\ta_{1}(\tt_{\ast})$}
\rput[c]{0}(6.0,1.5){\large $\ta_{2}(\tt_{\ast})$} 
\rput[c]{0}(3.5,3.0){\large $\tx$} 
\rput[c]{0}(3.5,3.5){\large $\times$} 
\rput[c](1.5,3.5){\Large $\tdelta^{(1)}_{\tx}$} 
\rput[c](5.5,3.5){\Large $\tdelta^{(2)}_{\tx}$} 
\psset{linewidth=0.2pt}
\pscurve(1.5,2.5)(3,4)(5,5.5)
\pscurve(5.5,2.5)(4,4)(2,5.5)
\pscurve(1.5,2.5)(1,2.7)(0,3) 
\pscurve(5.5,2.5)(6,2.7)(7,3)
\pscurve(1.5,2.5)(1.6,2)(1.8,1)
\pscurve(5.5,2.5)(5.4,2)(5.2,1)
\psset{linewidth=2.0pt}
\psline(3.5,3.5)(1.5,2.8)\psline(1.7,2.25)(2.1,2.55)
\pscurve(1.5,2.8)(1.3,2.7)(1.2,2.4)(1.5,2.2)(1.7,2.25)
\psline(2.2,3.05)(2.05,2.8)\psline(2.2,3.05)(1.96,3.15)
\psset{linewidth=2.0pt, linestyle=dashed}
\psline(2.15,2.57)(3.5,3.5)
\psset{linewidth=2.0pt, linestyle=solid}
\psline(3.5,3.5)(5.5,2.8)\psline(5.3,2.25)(5,2.5)
\pscurve(5.5,2.8)(5.7,2.7)(5.8,2.4)(5.5,2.2)(5.3,2.25)
\psline(4.8,3.05)(4.96,2.75)\psline(4.8,3.05)(5.1,3.15)
\psset{linewidth=2.0pt, linestyle=dashed}
\psline(4.95,2.57)(3.5,3.5)
\end{pspicture}
\end{center}
\caption{The paths $\tdelta^{(k)}_{\tx}$.} 
\label{fig:delta-kx}
\end{figure}

Using the equality \eqref{eq:yk-relation}, we have
\begin{equation} \label{eq:Sodd-relations-of-yk}
\int_{\tdelta^{(k)}_{\tx}}
\tS_{J,{\rm odd}}(\tx,\tt,\eta)~d\tx =
\int_{\delta^{(k)}_x} S_{\rm II, odd}
(x,t(\tt,\eta),\eta)~dx 
\Bigl|_{x=y^{(k)}(\tx,\tt,\eta)}
\end{equation}
for each $k=1,2$ (cf.~\cite[Section 2]{KT iwanami}). 
Here the integration path $\tdelta^{(k)}_{\tx}$ 
is a contour in the domain $\tU'_k$
depicted in Figure \ref{fig:delta-kx}.
That is, $\tdelta^{(k)}_{\tx}$ starts from the point 
on the second sheet of the Riemann surface of 
$\sqrt{Q_{J,0}(\tx,\tt)}$ corresponding 
to $\tx$, encircles the simple turning point 
$\ta_{k}(\tt_{\ast})$ and ends at the point corresponding 
to $\tx$ on the first sheet. (The wiggly line 
designates the branch cut for $\sqrt{Q_{J,0}(\tx,\tt)}$.) 
The path $\delta^{(k)}_{x}$ is defined 
in the same manner for $J={\rm II}$.
The right-hand side of 
\eqref{eq:Sodd-relations-of-yk}
is written as 
\begin{eqnarray}
\int_{\delta^{(k)}_x} S_{\rm II, odd}
(x,t(\tt,\eta),\eta)~dx 
\Bigl|_{x=y^{(k)}} & = & 
\int_{\delta^{(k)}_x} 
S_{\rm II, odd}(x,t(\tt,\eta),\eta)~dx 
\Bigl|_{x=y^{(k)}_0} \\[+.3em] 
& & \hspace{-6.em}
+ \sum_{n=0}^{\infty} \frac{\p^n S_{\rm II,odd}}{\p x^n}
(y^{(k)}_0,t(\tt,\eta),\eta) 
\frac{(y^{(k)}-y^{(k)}_0)^{n+1}}
{(n+1) \hspace{+.1em} !}
\nonumber
\end{eqnarray}
by the (formal) Taylor expansion. Taking the difference 
of both sides of \eqref{eq:Sodd-relations-of-yk} 
for $k=1$ and $k=2$, we have 
\begin{eqnarray} \label{eq:pre-key-equality}
\int_{\tdelta'} \tS_{J, {\rm odd}}
(\tx,\tt,\eta)~d\tx  & = & 
\int_{\delta'} S_{\rm II, odd}(x,t(\tt,\eta),\eta)~dx \\[+.3em] 
& & \hspace{-10.em}
+ \sum_{n=0}^{\infty} \frac{\p^n S_{\rm II,odd}}{\p x^n}
(x_0,t(\tt,\eta),\eta)~ 
\frac{(y^{(2)}-x_0)^{n+1}-(y^{(1)}-x_0)^{n+1}}
{(n+1) \hspace{+.1em} !}, 
\nonumber
\end{eqnarray}
where $\tdelta'$ is a closed path in the domain 
$\tU'_1 \cup \tU'_2$ 
which encircles the pair of simple turning points 
$\ta_1(\tt_{\ast})$ and $\ta_2(\tt_{\ast})$ 
as indicated in Figure \ref{fig:integral-cycles} 
($\delta'$ is defined in the same manner  
for $J={\rm II}$). Here we have used the equality 
\eqref{eq:coincidence-of-top-terms}. 

\begin{figure}
\begin{center}
\begin{pspicture}(1.,1)(13,5.8)
%
\psset{fillstyle=solid, fillcolor=black}
\pscircle(3.5,4.4){0.09} 
\pscircle(1.5,2.5){0.08}
\pscircle(5.5,2.5){0.08}
\pscircle(10.0,4.4){0.09} 
\pscircle(8.0,2.5){0.08}
\pscircle(12.0,2.5){0.08}
\psset{fillstyle=none}
\psset{linewidth=1.5pt}
\pscurve(5.5,2.5)(5.45,2.4)(5.3,2.5)(5.2,2.6)(5.1,2.5)
(5.0,2.4)(4.9,2.5)(4.8,2.6)(4.7,2.5)(4.6,2.4)(4.5,2.5)
(4.4,2.6)(4.3,2.5)(4.2,2.4)(4.1,2.5)(4,2.6)(3.9,2.5)
(3.8,2.4)(3.7,2.5)(3.6,2.6)(3.5,2.5)(3.4,2.4)(3.3,2.5)
(3.2,2.6)(3.1,2.5)(3,2.4)(2.9,2.5)(2.8,2.6)(2.7,2.5)
(2.6,2.4)(2.5,2.5)(2.4,2.6)(2.3,2.5)(2.2,2.4)(2.1,2.5)
(2,2.6)(1.9,2.5)(1.8,2.4)(1.7,2.5)(1.6,2.55)(1.5,2.5) %
\pscurve(12,2.5)(11.95,2.4)(11.8,2.5)(11.7,2.6)(11.6,2.5)
(11.5,2.4)(11.4,2.5)(11.3,2.6)(11.2,2.5)(11.1,2.4)(11.0,2.5)
(10.9,2.6)(10.8,2.5)(10.7,2.4)(10.6,2.5)(10.5,2.6)(10.4,2.5)
(10.3,2.4)(10.2,2.5)(10.1,2.6)(10,2.5)(9.9,2.4)(9.8,2.5)
(9.7,2.6)(9.6,2.5)(9.5,2.4)(9.4,2.5)(9.3,2.6)(9.2,2.5)
(9.1,2.4)(9,2.5)(8.9,2.6)(8.8,2.5)(8.7,2.4)(8.6,2.5)
(8.5,2.6)(8.4,2.5)(8.3,2.4)(8.2,2.5)(8.1,2.55)(8,2.5) %
\rput[c]{0}(3.5,5.5){\large $\tlambda_0(\tt_{\ast})$}
\rput[c]{0}(1.1,1.5){$\ta_{1}(\tt_{\ast})$}
\rput[c]{0}(5.9,1.5){$\ta_{2}(\tt_{\ast})$} 
\rput[c](3.5,1.4){\Large $\tdelta'$} 
\rput[c](4.7,4.2){\Large $\tdelta_0$} 
\rput[c]{0}(10,5.5){\large $\tlambda_0(\tt_{\ast})$}
\rput[c]{0}(7.6,1.5){$\ta_{1}(\tt_{\ast})$}
\rput[c]{0}(12.4,1.5){$\ta_{2}(\tt_{\ast})$} 
\rput[c](8.2,4.2){\Large $\tdelta$} 
\psset{linewidth=0.2pt}
\pscurve(1.5,2.5)(3,4)(5,5.5)
\pscurve(5.5,2.5)(4,4)(2,5.5)
\pscurve(1.5,2.5)(1,2.7)(0.5,2.9) 
\pscurve(5.5,2.5)(6,2.7)(6.5,2.9)
\pscurve(1.5,2.5)(1.6,2)(1.8,1)
\pscurve(5.5,2.5)(5.4,2)(5.2,1)
\pscurve(8.0,2.5)(9.5,4)(11.5,5.5)
\pscurve(12.0,2.5)(10.5,4)(8.5,5.5)
\pscurve(8.0,2.5)(7.5,2.7)(7.0,2.9) 
\pscurve(12.0,2.5)(12.5,2.7)(13.0,2.9)
\pscurve(8.0,2.5)(8.1,2)(8.3,1)
\pscurve(12.0,2.5)(11.9,2)(11.7,1)
\psset{linewidth=2.0pt}
\pscurve(2,3)(3.5,3)(5,3)
\pscurve(2,2)(3.5,2)(5,2)
\pscurve(2,3)(1.2,2.85)(1,2.5)(1.2,2.15)(2,2)
\pscurve(5,3)(5.8,2.85)(6,2.5)(5.8,2.15)(5,2)
\psline(3.35,3)(3.65,3.20)
\psline(3.35,3)(3.65,2.80)
\psline(3.6,2)(3.35,2.20)
\psline(3.6,2)(3.35,1.80)
\pscircle(3.5,4.4){0.7} 
\psline(4.15,4.5)(4.35,4.2)
\psline(4.15,4.5)(3.95,4.2)
\psline(2.85,4.25)(2.65,4.55)
\psline(2.85,4.25)(3.05,4.55)
\psset{linewidth=2.0pt}
\pscurve(9.7,4.6)(10.0,4.75)(10.3,4.6)
\pscurve(7.5,2.5)(9.0,4)(9.7,4.6)
\pscurve(12.5,2.5)(11.0,4)(10.3,4.6)
\pscurve(7.5,2.5)(7.4,2.2)(7.6,2)
\pscurve(12.5,2.5)(12.6,2.2)(12.4,2)
\pscurve(7.55,2.03)(7.78,1.95)(8.12,2.22)
\pscurve(12.45,2.03)(12.22,1.96)(11.88,2.22)
\pscurve(8.1,2.2)(8.2,2.3)(8.5,2.6)
\pscurve(11.9,2.2)(11.8,2.3)(11.6,2.5)
\psline(11.2,3.8)(11.25,3.50)
\psline(11.2,3.8)(11.47,3.8)
\psset{linewidth=2.0pt,linestyle=dashed}
\pscurve(9.5,3.6)(10.0,3.95)(10.5,3.6) 
\pscurve(8.5,2.6)(9.0,3.1)(9.4,3.5) 
\pscurve(11.5,2.6)(11.0,3.1)(10.6,3.5) 
\end{pspicture}
\end{center}
\caption{The cycles $\tdelta'$, $\tdelta_0$ and $\tdelta$.} 
\label{fig:integral-cycles}
\end{figure}

Now we prove the following key lemma.
\begin{lem} \label{lem:key-lamma}
If the constant $c$ in $\PII$ and $\SLII$ is chosen 
by \eqref{eq:c-P2} and the free parameters satisfy 
\eqref{eq:invariance-of-residues}, 
then the following equality holds: 
\begin{equation} \label{eq:SoddJ-is-SoddII}
\int_{\tdelta'} \tS_{J, {\rm odd}}
(\tx,\tt,\eta)~d\tx  =
\int_{\delta'} S_{\rm II, odd}
(x,t,\eta)~dx.
\end{equation}
The left-hand side (resp., right-hand side) 
of \eqref{eq:SoddJ-is-SoddII} does not depend
on $\tt$ (resp., $t$), and hence \eqref{eq:SoddJ-is-SoddII}
is an equality for constants.
\end{lem}
\begin{proof}[Proof of Lemma \ref{lem:key-lamma}]
Let $\tdelta$ be a closed cycle encircling 
two Stokes segments $\tgamma_1$ and $\tgamma_2$
as indicated in Figure \ref{fig:integral-cycles}, 
and $\delta$ be a similar cycle for $J={\rm II}$. 
Then, the cycles can be decomposed as 
$\tdelta' = \tdelta - \tdelta_0$ and 
$\delta' = \delta - \delta_0$, where 
$\tdelta_0$ is a closed cycle encircling 
the double turning point $\tlambda_0(\tt_{\ast})$ 
as in Figure \ref{fig:integral-cycles}, and 
$\delta_0$ is defined in the same manner for $J={\rm II}$. 
Then, \eqref{eq:invariance-of-residues} implies that 
\begin{equation} \label{eq:invariance-delta0}
\int_{\tdelta_0} \tS_{J,{\rm odd}}(\tx,\tt,\eta)~d\tx =
\int_{\delta_0} S_{{\rm II},{\rm odd}}(x,t,\eta)~dx
\end{equation}
since $\tE_J/4$ and $E_{\rm II}/4$ are residues of 
$\tS_{J, \rm odd}(\tx,\tt,\eta)~d\tx$ and 
$S_{\rm II, odd}(x,t,\eta)~dx$ at the double turning points
$\tlambda_0(\tt)$ and $\lambda_0(t_0(\tt))$, respectively.
In particular, both sides of \eqref{eq:invariance-delta0}
are independent of $\tt$ and $t$.

Furthermore, our choice \eqref{eq:c-P2} 
of the constant $c$ in $\PII$ and $\SLII$ entails that
\begin{equation} \label{eq:invariance-delta}
\int_{\tdelta} \tS_{J,{\rm odd}}(\tx,\tt,\eta)~d\tx =
\int_{\delta} S_{{\rm II},{\rm odd}}(x,t,\eta)~dx
\end{equation}
by the following reason. 

First, since we assume that all singular points 
of $\tQ_{J,0}(\tx,\tt)$ are poles of even order 
in Assumption \ref{ass:transform-to-P2} (5),
$\sqrt{\tQ_{J,0}(\tx,\tt)}$ 
(and hence $\tS_{J,{\rm odd}}(\tx,\tt,\eta)$)
does not have branch points except for 
$\ta_1(\tt)$ and $\ta_2(\tt)$. Therefore, 
the left-hand side of \eqref{eq:invariance-delta} 
is reduced to the sum of residues of 
$\tS_{J,{\rm odd}}(\tx,\tt,\eta)~d\tx$ at 
singular points of $\SLJ$. As is noted in 
\eqref{eq:resSodd=res-1}, the residues of 
$\tS_{J,\rm odd}(\tx,\tt,\eta)~d\tx$ at  
singular points coincide with those of 
$\eta \sqrt{\tQ_{J,0}(\tx,\tt)}~d\tx$. Thus we have
\begin{equation} \label{eq:Sodd-is-sqrtQ0}
\int_{\tdelta} \tS_{J,{\rm odd}}(\tx,\tt,\eta)~d\tx =
\eta \int_{\tdelta} \sqrt{\tQ_{J,0}(\tx,\tt)}~d\tx.
\end{equation}
On the other hand, the equality 
\eqref{eq:branches-P2SL2}
shows that 
\begin{eqnarray} \label{eq:Sodd-is-sqrtQ0-2}
\oint_{\tdelta} \sqrt{\tQ_{J,0}(\tx,\tt)}~d\tx & = & 
2 \left(
\int_{\tgamma_1} \sqrt{\tQ_{J,0}(\tx,\tt)}~d\tx -
\int_{\tgamma_2} \sqrt{\tQ_{J,0}(\tx,\tt)}~d\tx \right) \\ 
& = & \nonumber
\int_{\tr_1}^{\tr_2} \sqrt{\tF_J^{(1)}(\tt)}~d\tt = 
2 \pi i c. 
\end{eqnarray}
Here we have used 
\eqref{eq:c-P2}. 
Since the equalities \eqref{eq:Sodd-is-sqrtQ0} 
and \eqref{eq:Sodd-is-sqrtQ0-2} 
also hold for $J={\rm II}$, we have 
\begin{equation} \label{eq:Sodd-is-sqrtQ0-II}
\int_{\delta} S_{{\rm II},{\rm odd}}(x,t,\eta)~dx =
2 \pi i c \eta.
\end{equation}
Combining \eqref{eq:Sodd-is-sqrtQ0},  
\eqref{eq:Sodd-is-sqrtQ0-2} and 
\eqref{eq:Sodd-is-sqrtQ0-II}, we obtain 
\begin{equation} 
\int_{\tdelta} \tS_{J,{\rm odd}}(\tx,\tt,\eta)~d\tx = 
2 \pi i c \eta =
\int_{\delta} S_{{\rm II},{\rm odd}}(x,t,\eta)~dx, 
\end{equation}
which proves \eqref{eq:invariance-delta}.

As is explained above, we have 
\begin{eqnarray*}
\int_{\tdelta'} \tS_{J, {\rm odd}}
(\tx,\tt,\eta)~d\tx  & = & \int_{\tdelta} 
\tS_{J, {\rm odd}}(\tx,\tt,\eta)~d\tx - 
\int_{\tdelta_0} \tS_{J, {\rm odd}}
(\tx,\tt,\eta)~d\tx,  \\[+.2em]
\int_{\delta'} S_{\rm II, odd}(x,t,\eta)~dx & = & 
\int_{\delta} S_{\rm II, odd}(x,t,\eta)~dx - 
\int_{\delta_0} S_{\rm II, odd}(x,t,\eta)~dx.
\end{eqnarray*}
Therefore, the desired equality 
\eqref{eq:SoddJ-is-SoddII} follows from 
\eqref{eq:invariance-delta0} and 
\eqref{eq:invariance-delta}.
\end{proof}
Due to Lemma \ref{lem:key-lamma}, 
the equality \eqref{eq:pre-key-equality} implies that 
\begin{equation} \label{eq:y2-y1}
\sum_{n=0}^{\infty} \frac{\p^n S_{\rm II,odd}}{\p x^n}
(x_0,t(\tt,\eta),\eta) ~
\frac{(y^{(2)}-x_0)^{n+1}-(y^{(1)}-x_0)^{n+1}}
{(n+1) \hspace{+.1em} !} = 0.
\end{equation}
The coefficient of $\eta^{-(j-2)/2}$ in the left-hand side 
of \eqref{eq:y2-y1} is written as 
\[
S_{-1}(x_0,t_0) (y^{(2)}_{j/2} - y^{(1)}_{j/2}) + 
\bigl( T_{j/2}(y^{(2)}_{0},\dots,y^{(2)}_{(j-1)/2}) - 
T_{j/2}(y^{(1)}_{0},\dots,y^{(1)}_{(j-1)/2}) \bigr),
\]
where $S_{-1}(x,t)$ is the top term of 
$S_{\rm II, odd}(x,t,\eta)$ and the term 
$T_{j/2}(y^{(2)}_{0},\dots,y^{(2)}_{(j-1)/2})$ 
(resp., $T_{j/2}(y^{(1)}_{0},\dots,y^{(1)}_{(j-1)/2})$)
consists of the terms given by 
$y^{(2)}_{0},\dots,y^{(2)}_{(j-1)/2}$
(resp., $y^{(1)}_{0},\dots,y^{(1)}_{(j-1)/2}$).
Hence we can prove $y^{(1)}_{j/2}(\tx,\tt,\eta) = 
y^{(2)}_{j/2}(\tx,\tt,\eta)$ for all $j \ge 0$ 
by using the induction.
\end{proof}

Set
\begin{equation} 
x(\tx,\tt,\eta)~\left( = 
x^{\rm pre}(\tx,\tt,\eta) = y^{(1)}(\tx,\tt,\eta) 
 = y^{(2)}(\tx,\tt,\eta) \right).
\end{equation}
Then we have proved that the coefficients of the 
formal series $x(\tx,\tt,\eta)$ are holomorphic 
in a domain $\tU$ containing the double turning point 
$\tlambda_0(\tt)$ and the pair of the two simple turning 
points $\ta_1(\tt)$ and $\ta_2(\tt)$. 
The formal series $x(\tx,\tt,\eta)$ and 
$t(\tt,\eta)$ have almost all the desired properties 
in Theorem \ref{thm:main-theorem1}. 

Now what remains to be proved is the equality 
\eqref{eq:equivalece-of-PJ} in 
Theorem \ref{thm:main-theorem1}. 
This is a consequence of Proposition 
\ref{prop:transformation-at-double-TP}; in fact, 
\eqref{eq:xpre-relation2} reads as follows:
\begin{equation} \label{eq:x-relation2}
2\tilde{B}_J(\tx,\tt,\eta)
\frac{x(\tx,\tt,\eta) - \lambda_{\rm II}(t(\tt,\eta),\eta)}
{\tx - \tlambda_J(\tt,\eta)} 
\frac{\p x}{\p \tx} = \frac{\p t}{\p \tt} +
2 (x(\tx,\tt,\eta) - \lambda_{\rm II}(t(\tt,\eta),\eta)
\frac{\p x}{\p \tt}.
\end{equation}
Here $\tilde{B}_J(\tx,\tt,\eta)$ is defined by 
$(\tx-\tlambda_J(\tt,\eta))\tA_J(\tx,\tt,\eta)$, which is 
holomorphic at $\tx = \tlambda_J(\tt,\eta)$ 
in view of Table \ref{table:AJ}.
Since the right-hand side of \eqref{eq:x-relation2} is 
non-singular at $\tx = \tlambda_J(\tt,\eta)$, we find 
\begin{equation}
x(\tlambda_J(\tt,\eta),\tt,\eta) = 
\lambda_{\rm II}(t(\tt,\eta), \eta).
\end{equation}
Thus we have proved all claims 
in Theorem \ref{thm:main-theorem1}.

\section{Transformation to $\PIII$ on loop-type 
$P$-Stokes segments}
\label{section:main-theorem2}

In this section we show our second main claim 
concerning with WKB theoretic transformation 
of Painlev\'e transcendents on a loop-type 
$P$-Stokes segment. We put symbol $\sim$ over variables 
or functions relevant to $\PJ$ and $\SLJ$ 
as in the previous section.


\subsection{Assumptions and statements}

Let $(\tilde{\lambda}_J, \tilde{\nu}_J) = $ 
$(\tilde{\lambda}_J(\tt,\eta;\talpha,\tbeta), 
\tilde{\nu}_J(\tt,\eta;\talpha,\tbeta))$ 
be a 2-parameter solution of $(H_J)$ defined in 
a neighborhood of a point $\tt_{\ast} \in \Omega_J$, 
and consider $\SLJ$ and $\DJ$ with 
$(\tilde{\lambda}_J, \tilde{\nu}_J)$ 
substituted into their coefficients. 
In this section we impose the following conditions.
\begin{ass} \label{ass:transform-to-P3}
\begin{enumerate}[\upshape (1)]
\item %
$J \in \{{{\rm III'}(D_7)}, 
{{\rm III'}(D_6)}, {\rm IV}, {\rm V}, {\rm VI} \}$. %
\item %
There is a $P$-Stokes segment of {\em loop-type} 
$\tilde{\Gamma}$ in the $P$-Stokes geometry of $\PJ$ 
which emanates from and returns to a simple $P$-turning 
point $\tr$ of $\tlambda_J$ 
(which is not simple-pole type), and the point $\tt_{\ast}$ 
in question lies on $\tilde{\Gamma}$ as 
indicated in Figure \ref{fig:D7-transformable} (a). %
\item %
The function \eqref{eq:phaseJ} appearing 
in the instanton $\tPhi_J(\tt,\eta)$ of the 2-parameter 
solution $(\tilde{\lambda}_J, \tilde{\nu}_J)$ 
is normalized at the simple $P$-turning point $\tr$: 
\begin{equation} \label{eq:tlidephiJ2}
\tphi_J(\tt) = \int_{\tr}^{\tt}\sqrt{\tF_J^{(1)}(\tt)}~d\tt.
\end{equation}
Here the path of integral is taken along 
one of the paths $\tGamma_{\tt,1}$ or $\tGamma_{\tt,2}$ 
shown Figure \ref{fig:D7-transformable} (a). 
(Since there are singular points inside a 
loop-type $P$-Stokes segment in general, the two 
paths $\tGamma_{\tt,1}$ and $\tGamma_{\tt,2}$ are 
not homotopic in general.)
\item %
The Stokes geometry of $\SLJ$ at $\tt=\tt_{\ast}$ 
contains the same configuration as in 
Figure \ref{fig:D7-transformable} (b). 
That is, the following conditions hold.
\begin{itemize}
\item %
The double turning point $\tlambda_0(\tt_{\ast})$ 
is connected to the {\em same} simple turning point 
$\tilde{a}(\tt_{\ast})$ by two Stokes segments 
$\tilde{\gamma}_{1}$ and 
$\tilde{\gamma}_{2}$. 
Here labels of the Stokes segments are 
given as follows: When $\tt$ tends to 
$\tr$ along the path $\tGamma_{\tt,1}$ 
(resp., $\tGamma_{\tt,2}$) depicted in 
Figure \ref{fig:D7-transformable} (a), 
the Stokes segment $\tgamma_1$ 
(resp., $\tgamma_2$) shrinks to a point
(cf.~Proposition \ref{prop:PStokes-and-Stokes}).
\item %
The union of the Stokes segments $\tilde{\gamma}_{1}$ 
and $\tilde{\gamma}_{2}$ divide the $\tx$-plane 
into two domains. Let $\tilde{W}$ be one of them 
which contains both of the end-points 
$\tilde{p}_1$ and $\tilde{p}_2$
of two Stokes curves of $\SLJ$ emanating from 
$\tlambda_0(\tt)$ other than 
$\tilde{\gamma}_{1}$ or $\tilde{\gamma}_{2}$.
Then, the end-point of the Stokes curve emanating from 
$\ta(\tt_{\ast})$ is {\em not} contained in the domain $\tilde{W}$.
(Unlike Figure \ref{fig:D7-transformable} (b), 
the domain $\tilde{W}$ may contain $\tx = \infty$. 
Also, the points $\tilde{p}_1$ 
and $\tilde{p}_2$ may coincide.)
\end{itemize} 
\item %
The domain $\tilde{W}$ defined above does {\em not} 
contain the other turning point of $\SLJ$ than $\ta(\tt)$ 
and $\tlambda_0(\tt)$. 
All singular points of $\tQ_{J,0}(\tx,\tt)$ 
(as a function of $\tx$) contained in $\tilde{W}$ 
are poles of {\em even} order. %
\end{enumerate}
\end{ass} %

  \begin{figure}[h]
  \begin{center}
\begin{pspicture}(0,0)(0,0)
\psset{fillstyle=none}
\end{pspicture}
\end{center}
  \begin{minipage}{0.5\hsize}
  \begin{center}
  \includegraphics[width=34mm]
  {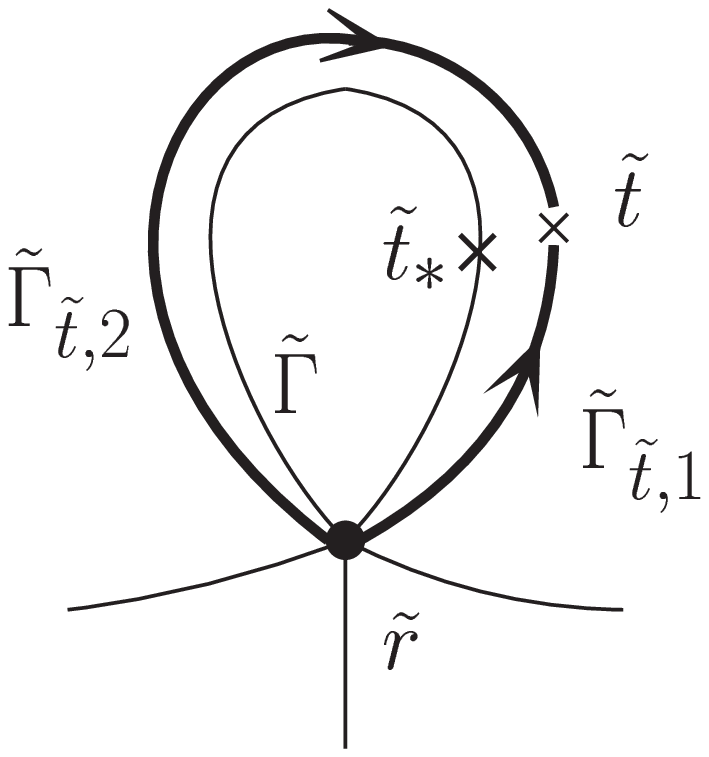} \\
  {(a): The loop-type $P$-Stokes segment $\tGamma$.}
  \end{center}
  \end{minipage} \hspace{+1.0em}
  \begin{minipage}{0.45\hsize}
  \begin{center}
  \includegraphics[width=49mm]
  {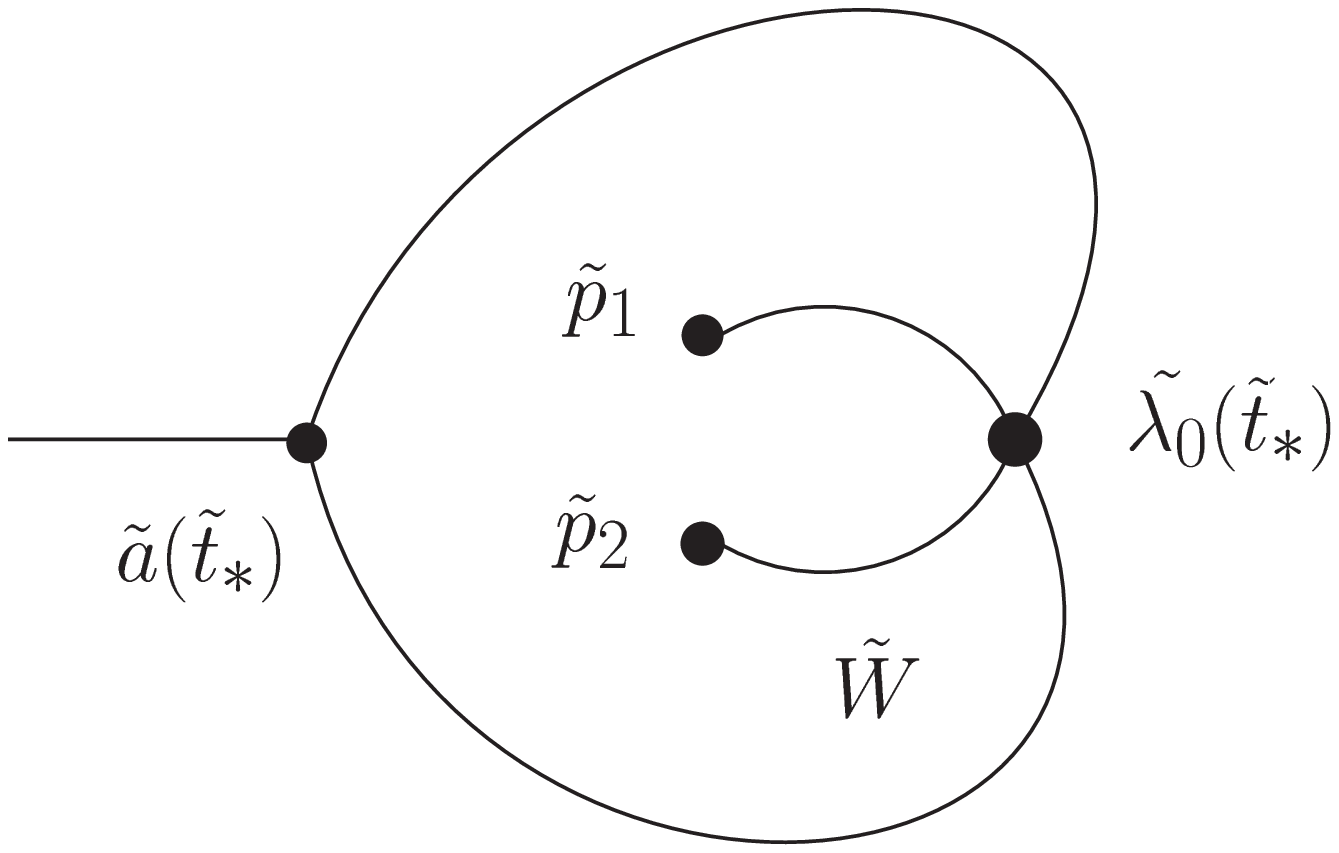} \\[+1.5em]
  {(b): The Stokes geometry of $\SLJ$ at $\tt = \tt_{\ast}$.}
  \end{center} 
  \end{minipage} 
  \caption{The $P$-Stokes geometry of $\PJ$ and an example 
  of the Stokes geometry of $\SLJ$ 
  satisfying Assumption \ref{ass:transform-to-P3}.}
  \label{fig:D7-transformable}
  \end{figure}

Since the $P$-Stokes geometry for $J = {\rm I}$, 
${\rm II}$, and ${\rm III'}(D_8)$ never contains 
a $P$-Stokes segment of loop-type, we have excluded 
these cases. Similarly to Theorem \ref{thm:main-theorem1},
under Assumption \ref{ass:transform-to-P3}
we will construct a formal transformation series 
to the third Painlev\'e equation $\PIII$ of type $D_7$. 

We fix the constant $c$ 
contained in $\PIII$ and $\SLIII$ by 
\begin{equation} \label{eq:c-P3}
c = \frac{1}{2\pi i} \int_{\tGamma}
\sqrt{\tF_{J}^{(1)}(\tt)}~d\tt, 
\end{equation}
where the path of integral is taken along the 
loop-type $P$-Stokes segment $\tGamma$ 
in the same direction as the integral 
\eqref{eq:tlidephiJ2}; that is, when 
\eqref{eq:tlidephiJ2} is defined along the path 
$\tGamma_{\tt,1}$ (resp., $\tGamma_{\tt,2}$) 
in Figure \ref{fig:D7-transformable} (a), 
then the path of integral in \eqref{eq:c-P3} 
is taken counter-clockwise (resp., clockwise)
direction along $\tGamma$. Here we assume that 
the imaginary part of $c$ is positive;
$c \in i \hspace{+.1em} {\mathbb R}_{>0}$. 
Then, the geometric configuration of $P$-Stokes geometry 
of $\PIII$ (described in the variable $u$ given by 
\eqref{eq:u-P3}) is the same as in 
Figure \ref{fig:P3-SL3-D7-Stokes} (P) 
when $c$ is given by \eqref{eq:c-P3}. 
Thus, the $P$-Stokes geometry of $\PIII$ 
has a loop-type $P$-Stokes segment $\Gamma$ 
starting from and returns to the same 
$P$-simple turning point $r$. 
(As is remarked in Section 
\ref{section:degeneration-of-PStokes-geometry}, 
when $c \in i \hspace{+.1em} {\mathbb R}_{<0}$, 
the $P$-Stokes geometry of $\PIII$ is the reflection 
$u \mapsto -u$ of Figure \ref{fig:P3-SL3-D7-Stokes} (P), 
and our discussion below is also applicable 
to the case $c \in i \hspace{+.1em} {\mathbb R}_{<0}$.) 
Furthermore, we can verify that the corresponding 
Stokes geometry of $\SLIII$ on the loop type 
$P$-Stokes segment $\Gamma$ is the same as 
the Stokes geometry depicted in 
Figure \ref{fig:P3-SL3-D7-Stokes} (SL).
That is, when a point $t$ lies on $\Gamma$, 
the corresponding Stokes geometry of $\SLIII$ 
has a double turning point $x=\lambda_0(t)$ 
and a simple turning point $x=a(t)$, 
and two Stokes segments $\gamma$ and 
$\gamma'$ both of which connect 
$\lambda_0(t)$ and $a(t)$. 
These Stokes segments are labeled as follows: 
When $t$ tends to $r$ along the path $\Gamma_{t}$ 
(resp., $\Gamma_{t}'$) depicted in 
Figure \ref{fig:P3-SL3-D7-Stokes} (P), 
the Stokes segment $\gamma$ (resp., $\gamma'$) 
shrinks to a point 
(cf.~Proposition \ref{prop:PStokes-and-Stokes}).

  \begin{figure}[h]
  \begin{center}
\begin{pspicture}(0,0)(0,0)
\psset{fillstyle=none}
\rput[c]{0}(-3.05,-1.7){$\times$} 
\rput[c]{0}(-2.6,-1.6){\large $u_{\ast}$}
\rput[c]{0}(-3.9,-1.5){\large $\Gamma$}
\rput[c]{0}(-3.7,-2.9){\large $r$}
\rput[c]{0}(3.3,-2.3){\large $0$}
\rput[c]{0}(1.0,-1.9){\large $a(t_{\ast})$}
\rput[c]{0}(4.5,-3.2){\large $\lambda_0(t_{\ast})$} 
\rput[c]{0}(2.2,-3.7){\large $\gamma$}
\rput[c]{0}(3.2,-0.8){\large $\gamma'$}
\end{pspicture}
\end{center}
  \begin{minipage}{0.47\hsize}
  \begin{center}
  \includegraphics[width=45mm]
  {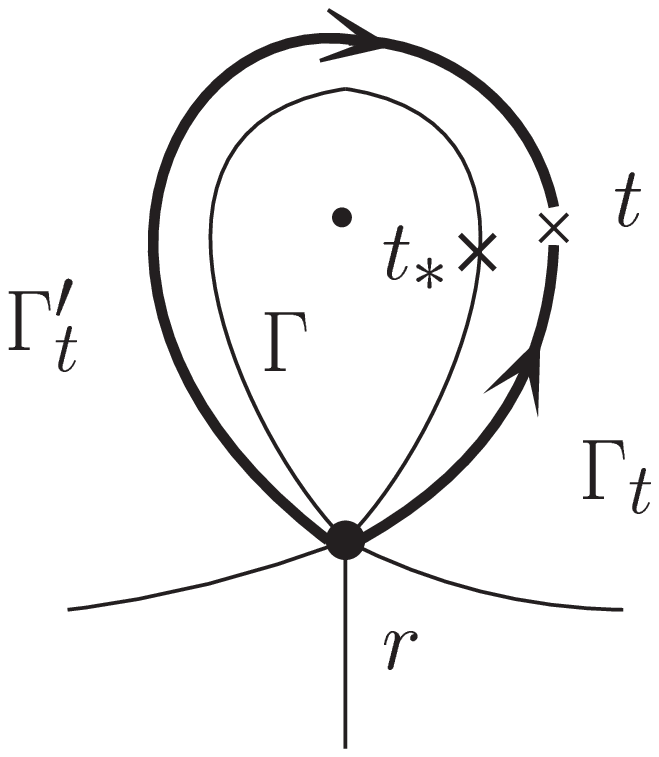} \\[-.5em]
  {(P): The P-Stokes geometry of $\PIII$ 
  (described on the u-plane).}
  \end{center}
  \end{minipage} \hspace{+.3em} 
  \begin{minipage}{0.45\hsize}
  \begin{center}
  \includegraphics[width=45mm]
  {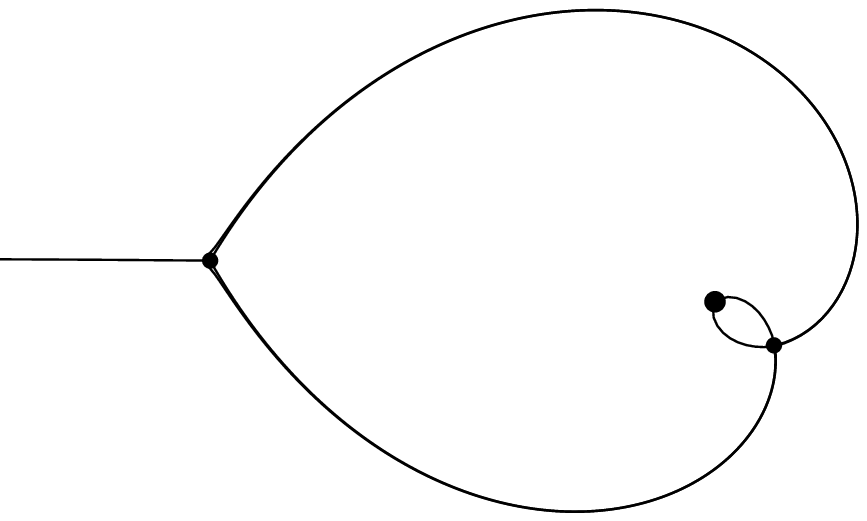} \\[-.5em]
  {(SL): The Stokes geomtry of $\SLIII$ 
  at $t = t_{\ast}$.}
  \end{center} 
  \end{minipage}
  \caption{The $P$-Stokes geometry of $\PIII$ 
  and the Stokes geometry of $\SLIII$.}
  \label{fig:P3-SL3-D7-Stokes}
  \end{figure}

Having the above geometric properties in mind, 
we formulate our second main result as follows. 
\begin{thm} \label{thm:main-theorem2}
Under Assumption \ref{ass:transform-to-P3}, 
for any 2-parameter solution 
$(\tilde{\lambda}_J, \tilde{\nu}_J) = $
$(\tilde{\lambda}_J(\tilde{t},\eta; \tilde{\alpha},
\tilde{\beta}), \tilde{\nu}_J(\tilde{t},\eta; 
\tilde{\alpha},\tilde{\beta}))$ of $(H_{J})$, 
there exist
\begin{itemize} %
\item an annular domain $\tilde{U}$ which contains 
the union $\tilde{\gamma}_1 \cup \tilde{\gamma}_2$
of two Stokes segments, %
\item a neighborhood $\tilde{V}$ of $\tilde{t}_{\ast}$, %
\item formal series 
\[
x(\tx,\tt,\eta)=\sum_{j\ge0}\eta^{-j/2}x_{j/2}(\tx,\tt,\eta), 
\quad 
t(\tt,\eta)=\sum_{j\ge0}\eta^{-j/2}t_{j/2}(\tt,\eta)
\]
whose coefficients 
$\{x_{j/2}(\tilde{x},\tilde{t},\eta)\}_{j=0}^{\infty}$ 
$\{t_{j/2}(\tilde{t},\eta)\}_{j=0}^{\infty}$ 
are functions defined on $\tilde{U}\times\tilde{V}$ 
and $\tV$, respectively, and may depend on $\eta$, %
\item a 2-parameter solution 
\begin{eqnarray*}
(\lambda_{\III},\nu_{\III}) & = &
(\lambda_{\III}(t,\eta;\alpha,\beta),
\nu_{\III}(t,\eta;\alpha,\beta)), \\
(\alpha,\beta) & = & (\sum_{n=0}^{\infty}\eta^{-n}\alpha_n, 
\sum_{n=0}^{\infty}\eta^{-n}\beta_n), 
\end{eqnarray*}
of $\HIII$ with the constant $c$ being determined by 
\eqref{eq:c-P3}, and the function \eqref{eq:phaseJ}
appearing in the instanton $\Phi_{\III}(t,\eta)$ that
is normalized at a simple $P$-turning point $r$ of 
$\PIII$ as 
\begin{equation} \label{eq:phi3D7}
\phi_{\III}(\tt) = \int_{r}^{t}
\sqrt{F_{\III}^{(1)}(t)}~dt,
\end{equation}
\end{itemize}
which satisfy the relations below: %
\begin{enumerate}[\upshape (i)]
\item The function $t_0(\tt)$ is independent of 
$\eta$ and satisfies 
\begin{equation}  
\label{eq:correspondence-of-instanton}
\tphi_{J}(\tt) = 
\tphi_{\III}(t_0(\tt)).
\end{equation}
\item $dt_0/d\tt$ never vanishes on $\tV$. %
\item The function $x_0(\tx,\tt)$ is also 
independent of $\eta$ and satisfies 
\begin{eqnarray} \label{eq:x0D7-1}
x_0(\tlambda_0(\tt),\tt) & = & \lambda_0(t_{0}(\tt),c), \\
x_0(\ta(\tt),\tt) & = & a(t_{0}(\tt),c).
\label{eq:x0D7-2}
\end{eqnarray} 
Here $\lambda_0(t)$ and $a(t)$ are double and simple 
turning points of $\SLIII$.
\item $\p x_0/\p\tx$ never vanishes on $\tU\times\tV$. %
\item $x_{1/2}$ and $t_{1/2}$ vanish identically. %
\item The functions 
$\{ x_{j/2}(\tx,\tt,\eta) \}_{j=0}^{\infty}$ 
are single-valued in the annular domain $\tU$ 
as functions of $\tx$. %
\item The $\eta$-dependence of $x_{j/2}$ and $t_{j/2}$ ($j \ge 2$)
is only through instanton terms $\exp(\ell \tPhi_{J}(\tt,\eta))$
($\ell = j-2-2m$ with $0 \le m \le j-2$) that appears 
in the 2-parameter solution $(\tlambda_J,\tnu_{J})$ of $(H_J)$. %
\item The following relations hold:
\begin{eqnarray} \label{eq:equivalece-of-PJD7}
\hspace{+1.em}
x(\tlambda_J(\tt,\eta;\talpha,\tbeta),\tt,\eta) & = &
\lambda_{{\rm III'}(D_7)}(t(\tt,\eta),c,\eta;\alpha,\beta), \\%
\label{eq:equivalece-of-SLJ}
\tQ_{J}(\tx,\tt,\eta) & = & 
\left(\frac{\p x(\tx,\tt,\eta)}{\p\tx}\right)^2
Q_{{\rm III'}(D_7)}(x(\tx,\tt,\eta),t(\tt,\eta),c,\eta) \\
& & 
-\frac{1}{2}\eta^{-2}\{ x(\tx,\tt,\eta);\tx\}, \nonumber
\end{eqnarray}%
where the 2-parameter solution of $(H_J)$ and $\HIII$ 
are substituted into $(\lambda,\nu)$ in the coefficients 
of $\tQ_J$ and $Q_{{\rm III'}(D_7)}$, respectively, and
$\{x(\tx,\tt,\eta);\tx\}$ denotes the Schwarzian 
derivative \eqref{eq:Schwarzian-derivative}.
\end{enumerate} %
\end{thm}
%

\subsection{Construction of the top 
term of the transformation}
\label{section:top-terms-D7}

First we explain the construction of $t_0(\tt)$ and 
$x_0(\tx,\tt)$. In the proof we consider the case that 
the path of integral \eqref{eq:tlidephiJ2} 
is taken along a path $\tGamma_{\tt,1}$ 
shown in Figure \ref{fig:D7-transformable} (a). 
(This additional assumption is imposed just 
in order to fix the situation, and our discussion 
below is also applicable to the case where
the path of integral \eqref{eq:tlidephiJ2} 
is taken along a path $\tGamma_{\tt,2}$.) 
Then, it follows from the definition \eqref{eq:c-P3}
of the constant $c$ that we have
\begin{equation} \label{eq:c-P3-2}
\int_{\tGamma_{\tt,1}} \sqrt{\tF^{(1)}_J(\tt)}~d\tt - 
\int_{\tGamma_{\tt,2}} \sqrt{\tF^{(1)}_J(\tt)}~d\tt = 
\int_{\tGamma} \sqrt{\tF^{(1)}_J(\tt)}~d\tt = 
2\pi i c.
\end{equation}

Let us construct $t_0(\tt)$. Similarly to Section 
\ref{section:top-terms}, under the assumption 
that $\tt_{\ast}$ lies on the $P$-Stokes segment 
$\tGamma$, we can construct $t^{(k)}_0(\tt)$ so that 
\begin{equation} 
\label{eq:correspondence-of-instanton2} 
\tphi_{J,k}(\tt) = \phi_{\III,k}(t^{(k)}_0(\tt))
\end{equation}
holds for each $k = 1$ and $2$, where
\begin{equation} \label{eq:phiJ-and-phiD7P3}
\tphi_{J,k}(\tt) = \int_{\tGamma_{\tt, k}}
\sqrt{\tF_{J}^{(1)}(\tt)}~d\tt,\quad 
\phi_{\III,k}(t) = \int_{\Gamma_{t, k}}
\sqrt{{F^{(1)}_{\III}}(t)}~dt.
\end{equation} 
Here the path $\Gamma_{t, k}$ 
for $\phi_{\III,k}(t)$ is a path from the $P$-turning point 
$r$ of $\PIII$ to $t$ defined by the following rule. 
Note that, under Assumption \ref{ass:transform-to-P3} (4),
we have the following two possibilities for 
the configuration of the Stokes geometry 
of $\SLJ$ at $\tt_{\ast}$ 
(see Figure \ref{fig:P3-transformable-posibilities}):
\begin{itemize}
\item[(A)] The Stokes segment $\tgamma_2$ 
comes next to the Stokes segment $\tgamma_{1}$ 
in the {\em counter-clockwise} order 
near $\tlambda_0(\tt_{\ast})$. %
\item[(B)] The Stokes segment $\tgamma_2$ 
comes next to the Stokes segment $\tgamma_1$ 
in the {\em clockwise} order 
near $\tlambda_0(\tt_{\ast})$. %
\end{itemize}
Then, we set 
\begin{eqnarray} 
\label{eq:choice-of-r1-and-r2-D7}
(\Gamma_{t, 1},\Gamma_{t, 2}) = 
\begin{cases}
(\Gamma_{t}, \Gamma_{t}') & 
\text{when the case (A) in Figure 
\ref{fig:P3-transformable-posibilities} happens}, \\
(\Gamma_{t}', \Gamma_{t}) & 
\text{when the case (B) in Figure 
\ref{fig:P3-transformable-posibilities} happens},
\end{cases}
\end{eqnarray}
where $\Gamma_{t}$ and $\Gamma_{t}'$ 
are the paths depicted 
in Figure \ref{fig:P3-SL3-D7-Stokes} (P). 
Moreover, the branch of $\sqrt{F_{\III}^{(1)}(t)}$ 
in \eqref{eq:phiJ-and-phiD7P3} is chosen 
so that the sign appearing in the right-hand side 
of \eqref{eq:integral-D7} is $+$ (the orientation of 
$\Gamma$ is given appropriately):
\begin{equation} \label{eq:c-P3-3}
\int_{\Gamma_{t,1}} \sqrt{F^{(1)}_{\III}(t)}~dt - 
\int_{\Gamma_{t,2}} \sqrt{F^{(1)}_{\III}(t)}~dt = 
\int_{\Gamma} \sqrt{F_{\III}^{(1)}(t)}dt = 
+ 2\pi i c.
\end{equation}
This choice \eqref{eq:choice-of-r1-and-r2-D7}
of $\Gamma_{t, 1}$ and $\Gamma_{t, 2}$ is essential 
in the construction of $x_0(\tx,\tt)$. 

\begin{figure}[h]
  \begin{minipage}{0.5\hsize}
  \begin{center}
  \includegraphics[width=50mm]
  {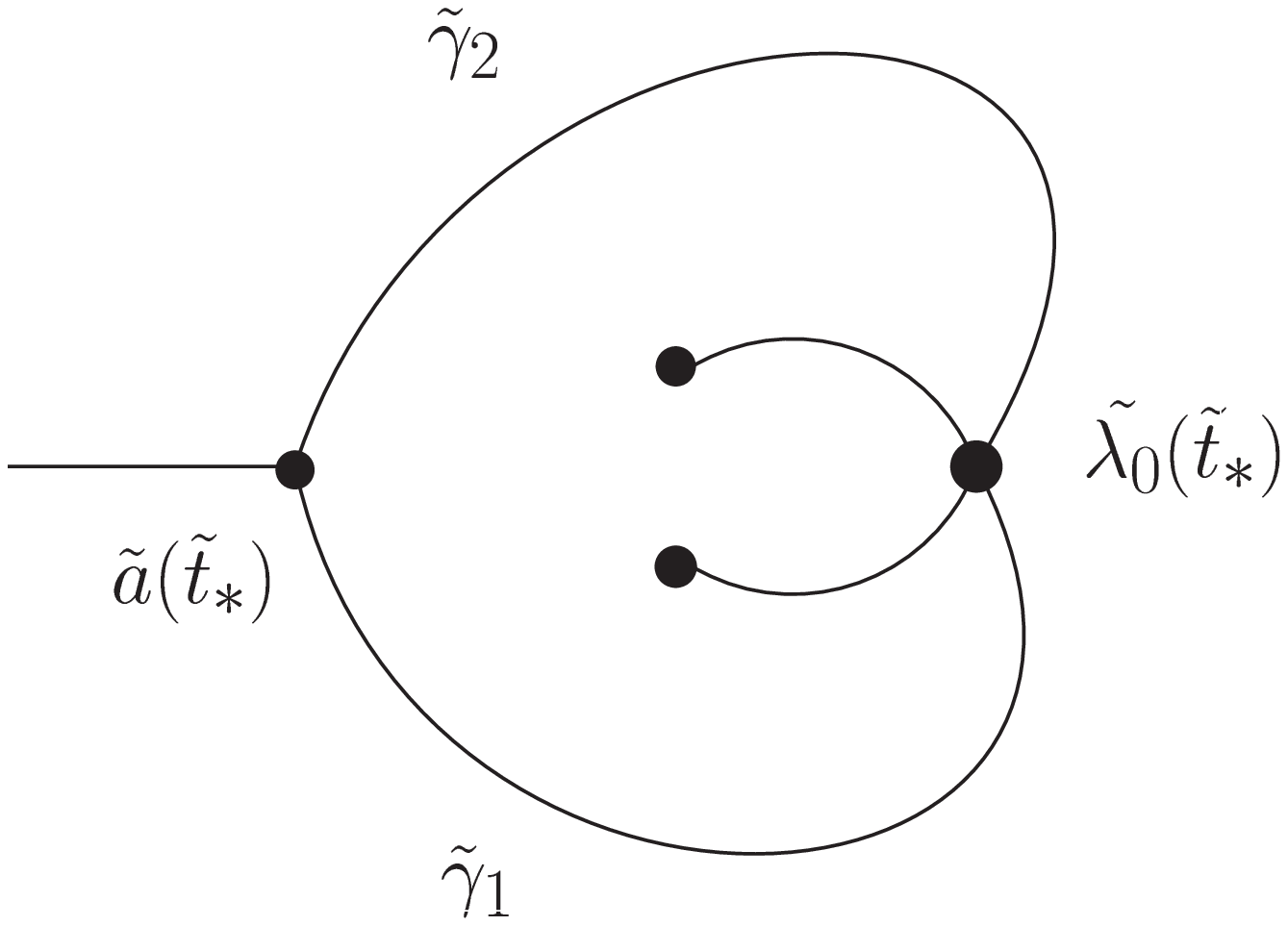} \\
  {(A)}
  \end{center}
  \end{minipage} \hspace{0.0em}
  \begin{minipage}{0.45\hsize}
  \begin{center}
  \includegraphics[width=50mm]
  {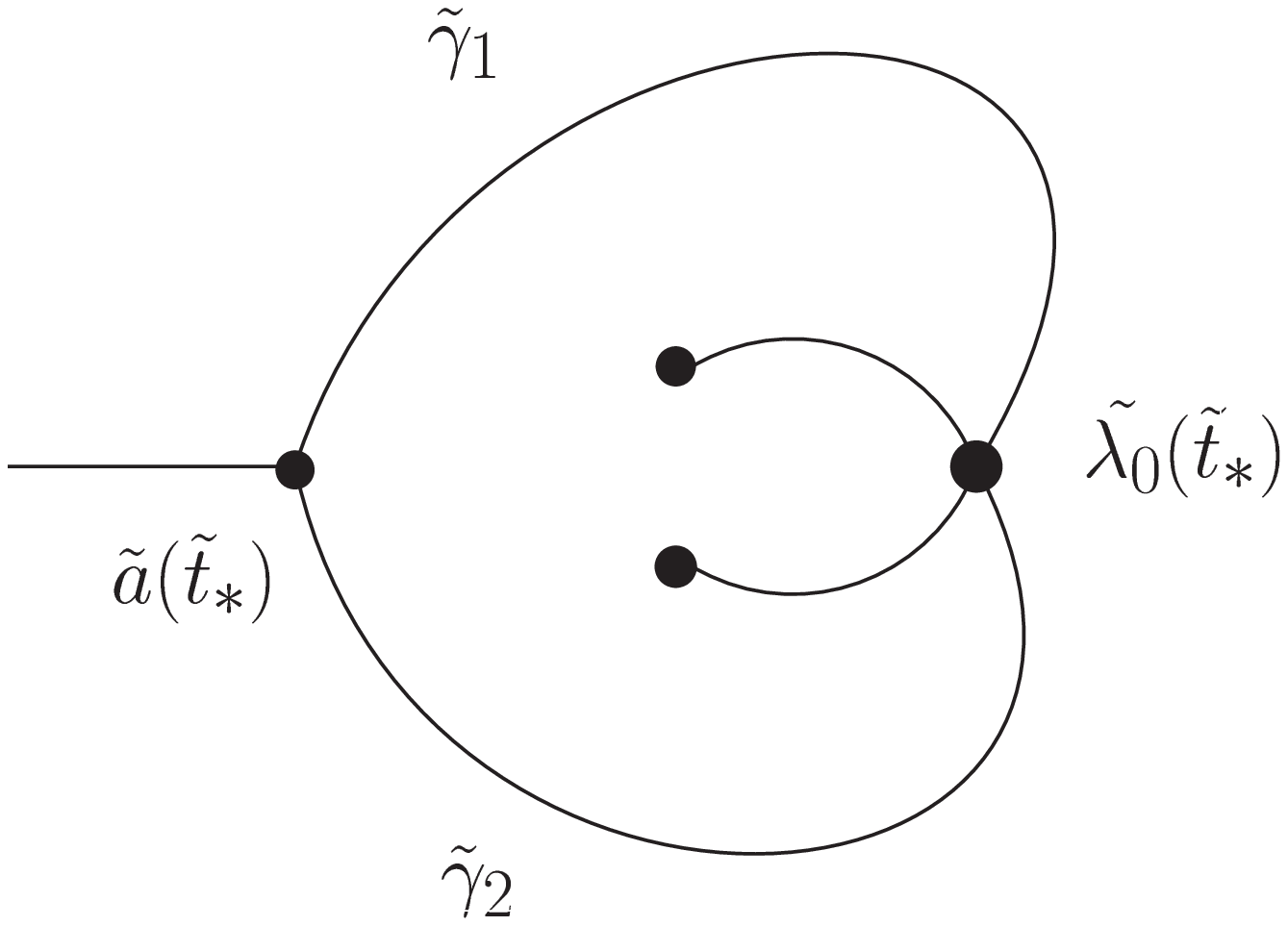} \\
  {(B)}
  \end{center} 
  \end{minipage} 
\caption{Two possibilities for adjacent 
Stokes segments of $\SLJ$.} 
\label{fig:P3-transformable-posibilities}
\end{figure}

Since the right-hand sides of \eqref{eq:c-P3-2} 
and \eqref{eq:c-P3-3} coincide, by the same 
discussion of Section \ref{section:top-terms}
we can show that $t^{(1)}_0(\tt) = t^{(2)}_0(\tt)$. 
We define $t_0(\tt) = t^{(1)}_0(\tt) = t^{(2)}_0(\tt)$.
Then, taking the path in \eqref{eq:phi3D7} 
along $\Gamma_{1,t}$, we have 
\eqref{eq:correspondence-of-instanton}.

Next we construct $x_0(\tx,\tt)$. Set
\begin{eqnarray} 
\label{eq:tw0-Stokes-segments-gamma1-and-gamma2-P3D7}
(\gamma_1,\gamma_2) & = &
\begin{cases}
(\gamma, \gamma') & \text{when the case (A) in Figure 
\ref{fig:P3-transformable-posibilities} happens}, \\
(\gamma', \gamma) & \text{when the case (B) in Figure 
\ref{fig:P3-transformable-posibilities} happens},
\end{cases} 
\end{eqnarray}
where $\gamma$ and $\gamma'$ are the Stokes segments 
of $\SLIII$ depicted in 
Figure \ref{fig:P3-SL3-D7-Stokes} (SL). 
Then, due to the equality 
\eqref{eq:correspondence-of-instanton2} 
and our choice \eqref{eq:choice-of-r1-and-r2-D7}
of paths in \eqref{eq:phiJ-and-phiD7P3},
the discussion of Section \ref{section:top-terms} 
is also valid in this case because the 
relative locations 
$(\ta,\tlambda_0,\tgamma_1,\tgamma_2)$ 
of the simple turning point, the 
double turning point and the two Stokes segments 
of $\SLJ$ completely coincide with those 
$(a,\lambda_0,\gamma_1,\gamma_2)$ of $\SLIII$. 
Thus we can construct $x_0(\tx,\tt)$ satisfying 
\eqref{eq:x0D7-1} and \eqref{eq:x0D7-2} and mapping 
the Stokes segments $\tgamma_1$ and $\tgamma_2$ 
to $\gamma_1$ and $\gamma_2$, respectively.
Furthermore, $x_0(\tx,\tt)$ becomes single-valued 
in an annular domain $\tU$ containing the union 
of two Stokes segments $\tgamma_1 \cup \tgamma_2$ 
due to the following fact:
At each turning point $\ta(\tt)$ and $\tlambda_0(\tt)$, 
a holomorphic function which maps $\tgamma_1$ 
to $\gamma_1$ {\em uniquely} exists and it must 
coincide with $x_0(\tx,\tt)$. 

In what follows we choose the branch of 
$\sqrt{\tQ_{J,0}(\tx,\tt)}$ and 
$\sqrt{Q_{\III,0}(x,t)}$ appearing in the proof so that 
\begin{eqnarray} \label{eq:branches-P3D7}
\int_{\tgamma_k} \sqrt{\tQ_{J}(\tx,\tt)}~d\tx & = & 
\frac{1}{2} \int_{\tGamma_{\tt,k}} \sqrt{\tF^{(1)}_J(\tt)}~d\tt \\
\int_{\gamma_k} \sqrt{Q_{\III}(x,t)}~dx & = & 
\frac{1}{2} \int_{\Gamma_{t,k}} \sqrt{F^{(1)}_{\III}(t)}~dt
\label{eq:branches-P3D7-sono2}
\end{eqnarray}
hold for $k=1,2$. 
In \eqref{eq:branches-P3D7} and 
\eqref{eq:branches-P3D7-sono2} Stokes segments 
are directed from the simple turning point 
to the double turning point.

\subsection{Construction of higher order
terms of the transformation series 
and the transformation of 
the 2-parameter solutions}
\label{section:construction-higher-D7}

Here we explain the construction of higher order 
terms of the transformation series. We note that 
most of the discussion given in Section \ref{section:main-results} 
are applicable also to this case. The transformation series 
$x^{\rm pre}(\tx,\tt,\eta)$ near the double turning point 
is constructed in the same manner as in Section 
\ref{section:transformation-at-double},
and the matching procedure given in Section 
\ref{section:matching-P2} is valid in our case since 
we have only used the fact that 
``there is a Stokes segment of $\SLJ$ connecting 
a simple turning point and the double turning point 
$\tlambda_0(\tt_{\ast})$" in the proof. 
What we have to prove here is the single-valuedness 
of the higher order coefficients of formal series
in the annular domain $\tU$ containing the 
union of two Stokes segments $\tgamma_1 \cup \tgamma_2$.

Define
\begin{eqnarray}
x^{\rm pre}(\tx,\tt,\eta) & = & z_{\III}^{-1}
(\tz_{J}(\tx,\tt,\eta), s_{J}(\tt,\eta),\eta),  \\
t^{\rm pre}(\tt,\eta) & = & s_{\III}^{-1}
(\ts_{J}(\tt,\eta),\eta)
\end{eqnarray}
in the same manner as \eqref{eq:x---pre} and 
\eqref{eq:t---pre} in Section 
\ref{section:transformation-at-double}. 
Here we have fixed the correspondence of free parameters 
$(\talpha,\tbeta)$ of $(\tlambda_J,\tnu_J)$ 
and $(\alpha,\beta)$ of 
$(\lambda_{\III},\nu_{\III})$ 
so that 
\begin{equation} \label{eq:invariance-of-residuesD7}
\tE_{J}(\talpha,\tbeta) = E_{\III}(\alpha,\beta)
\end{equation}
holds similarly to \eqref{eq:invariance-of-residues}. 
Let $y^{(1)}(\tx,\tt,\eta)$ and $y^{(2)}(\tx,\tt,\eta)$ 
be formal series which transform $\SLJ$ to $\SLIII$ 
near the simple turning point $\ta(\tt)$ 
in the sense of Lemma \ref{lemma:transformation-at-simple-TP}. 
In our geometric assumption these two formal series 
coincide near $\tx=\ta{(\tt)}$ due to the following reason.
Since the top term of $y^{(1)}_0(\tx,\tt)$ of 
$y^{(1)}(\tx,\tt,\eta)$ maps the Stokes segment 
$\tgamma_1$ of $\SLJ$ to the Stokes segment 
$\gamma_1$ of $\SLIII$ by definition, it also 
maps the other Stokes segment $\tgamma_2$ to 
$\gamma_2$ simultaneously. 
Thus $y^{(1)}_0(\tx,\tt)$ must coincide with 
$y^{(2)}_0(\tx,\tt)$ near $\tx=\ta(\tt)$, 
and hence the higher order terms also coincide, 
at least near $\tx = \ta(\tt)$. 
Especially, the top terms of them 
also coincide with $x_0(\tx,\tt)$ constructed 
in Section \ref{section:top-terms-D7}.

Furthermore, by the same argument as in 
Section \ref{section:matching-P2} we can prove that  
all coefficients of $y^{(1)}(\tx,\tt,\eta)$ become
holomorphic also at $\tx = \tlambda_0(\tt)$ and we have
\begin{equation}
x^{\rm pre}(\tx,\tt,\eta) = y^{(1)}(\tx,\tt,\eta),
\end{equation}
after we chose the free parameters contained in 
$t^{\rm pre}(\tt,\eta)$ appropriately. 
We denote by $t(\tt,\eta)$ the formal series 
$t^{\rm pre}(\tt,\eta)$ with parameters contained 
in it being chosen appropriately in the above sense. 
Then, there exists a domain $\tU'$ near $\tlambda_0(\tt)$ 
on which all the coefficients of two formal series 
$y^{(1)}(\tx,\tt,\eta)$ and $y^{(2)}(\tx,\tt,\eta)$
are holomorphic. Then, what we have to show here is 
that the analytic continuation of the coefficients 
of $y^{(1)}(\tx,\tt,\eta)$ along the Stokes segment 
$\tgamma_1$ coincides with the analytic continuation 
of the coefficients of $y^{(2)}(\tx,\tt,\eta)$ 
along the Stokes segment $\tgamma_2$ 
on the domain $\tU'$. 
Here we show that the single-valuedness 
is guaranteed by our choice \eqref{eq:c-P3} 
of the constant $c$ in $\PIII$ and $\SLIII$.

\begin{prop} \label{prop:single-valuedness} 
The analytic continuation of the coefficients of 
$y^{(1)}(\tx,\tt,\eta)$ along the Stokes segment 
$\tgamma_1$ coincide with the analytic continuation 
of the coefficients of $y^{(2)}(\tx,\tt,\eta)$ 
along the Stokes segment $\tgamma_2$ 
on the domain $\tU'$: 
\begin{equation} 
\label{eq:single-valuedness-of-transformation-y1-y2}
y^{(1)}(\tx,\tt,\eta) = y^{(2)}(\tx,\tt,\eta).
\end{equation}
Consequently, the coefficients of $y^{(1)}(\tx,\tt,\eta)$
and $y^{(2)}(\tx,\tt,\eta)$ are holomorphic and single-valued 
in $\tx$ on an annular domain $\tU$ containing 
the union of two Stokes segments $\tgamma_1 \cup \tgamma_2$.
\end{prop}
\begin{proof}
In the proof of Proposition \ref{prop:single-valuedness} 
we assume that the case (A) in Figure 
\ref{fig:P3-transformable-posibilities} happens. 
(The discussion given here is also applicable 
to the case (B) in Figure 
\ref{fig:P3-transformable-posibilities}.)
Moreover, we will prove the equality 
\eqref{eq:single-valuedness-of-transformation-y1-y2} 
when $\tt$ is fixed at $\tt_{\ast}$. 
This is just for the sake of clarity, 
and our proof is also valid in a 
neighborhood $\tV$ of $\tt_{\ast}$. 
(We may take a smaller neighborhood $\tV$
of $\tt_{\ast}$.)

  \begin{figure}[h]
  \begin{minipage}{0.5\hsize}
  \begin{center}
  \includegraphics[width=58mm]
  {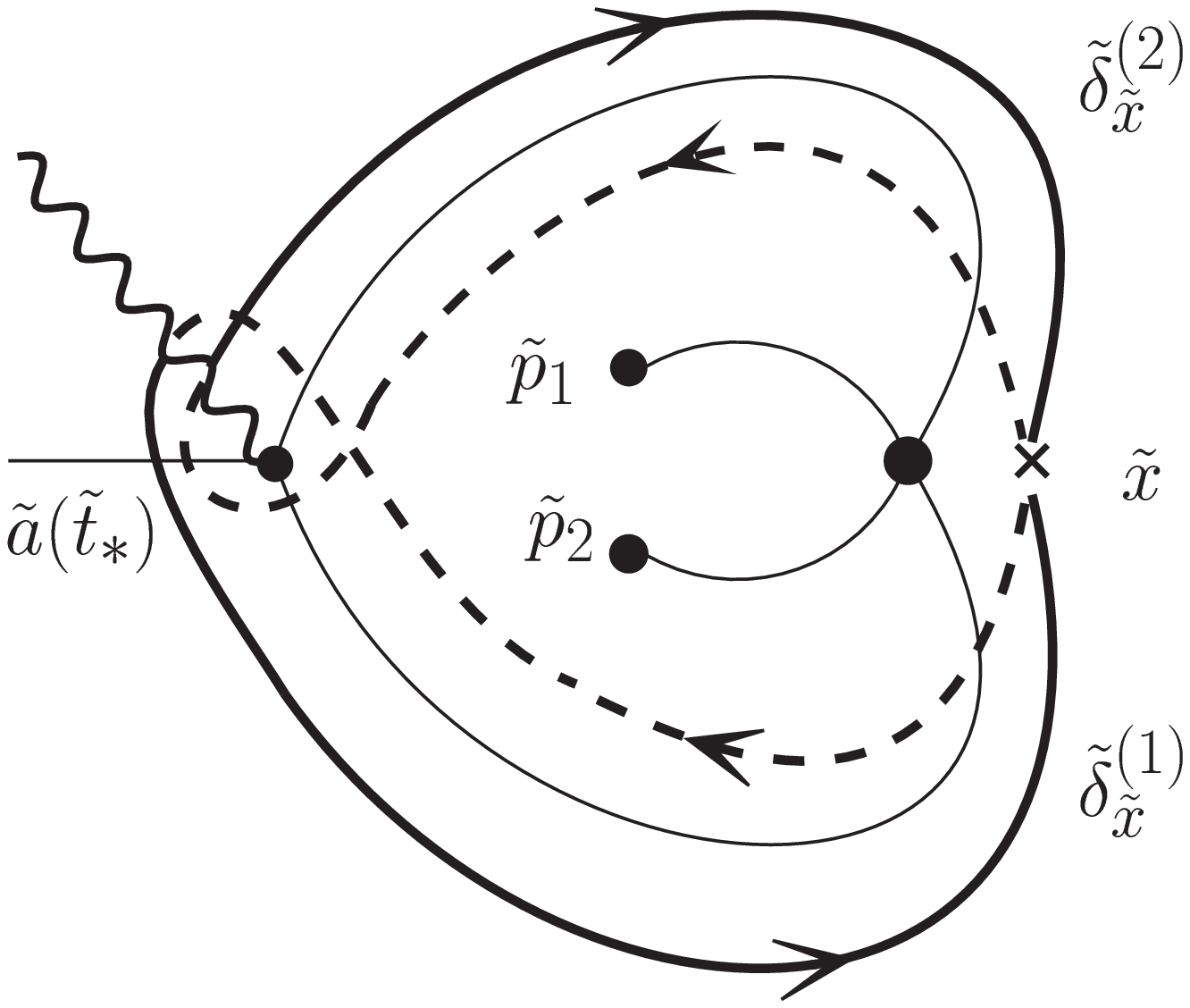} \\
  \end{center}
  \end{minipage} \hspace{-0.2em}
  \begin{minipage}{0.5\hsize}
  \begin{center}
  \includegraphics[width=56mm]
  {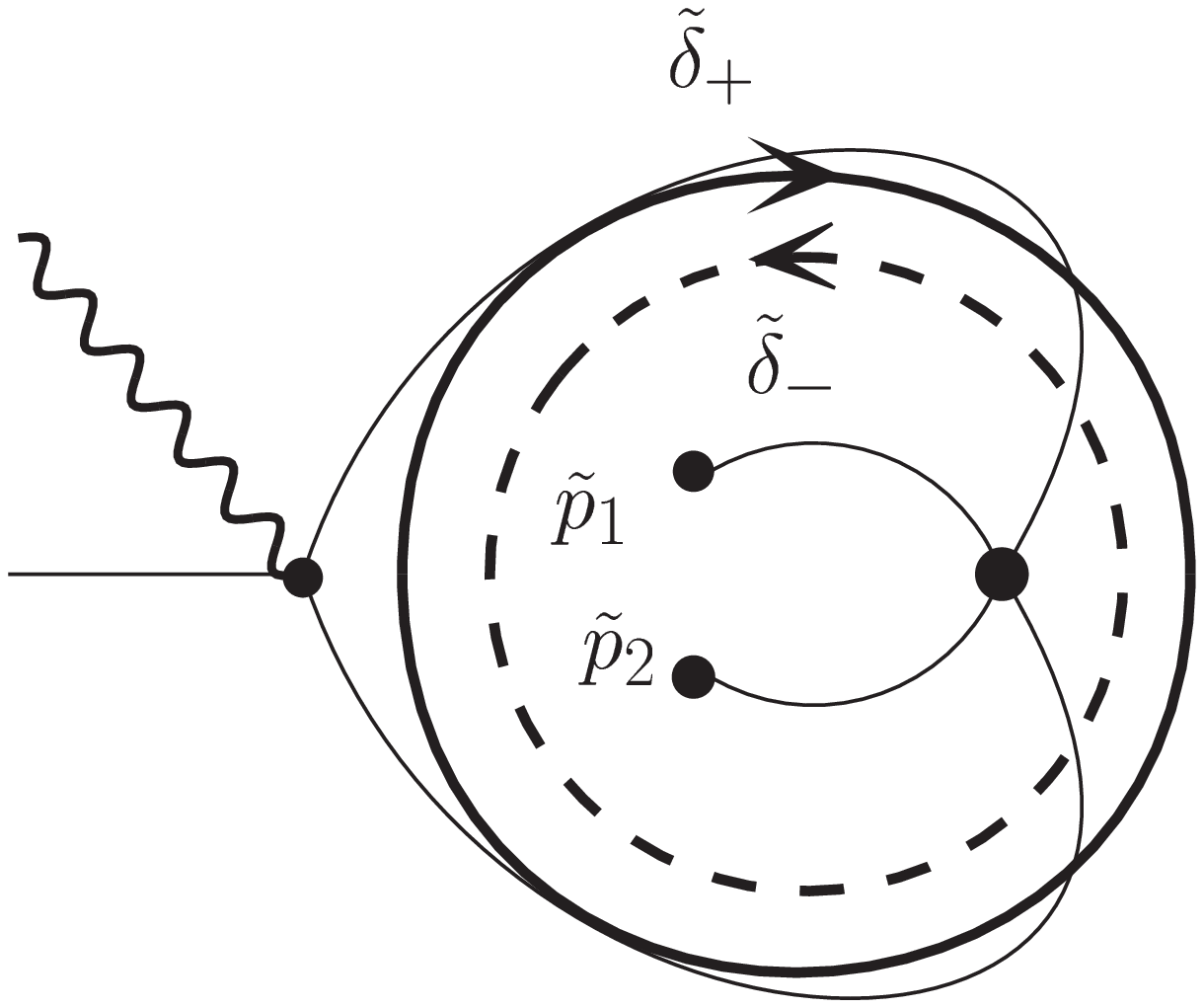} \\[+.5em]
  {The thick solid line (resp., thick dashed line) 
  designates the cycle $\tdelta_{+}$ (resp., $\tdelta_{-}$).}
  \end{center} 
  \end{minipage} 
  \caption{The cycles $\tdelta^{(k)}_{\tx}$ 
  ($k=1,2$) and $\tdelta = \tdelta_{+} + \tdelta_{-}$.}
 \label{fig:delta-kxD7}
  \end{figure}

Let $\tx$ be a point on the domain $\tU'$. 
Similarly to \eqref{eq:Sodd-relations-of-yk}, we have
\begin{equation} \label{eq:Sodd-relations-of-ykD7}
\int_{\tdelta^{(k)}_{\tx}}
\tS_{J,{\rm odd}}(\tx,\tt,\eta)~d\tx =
\int_{\delta^{(k)}_x} S_{\III, {\rm odd}}
(x,t(\tt,\eta),\eta)~dx 
\Bigl|_{x=y^{(k)}(\tx,\tt,\eta)}
\end{equation}
for each $k=1,2$. 
Here the integration path $\tdelta^{(k)}_{\tx}$ 
is a contour depicted in Figure \ref{fig:delta-kxD7}.
That is, $\tdelta^{(k)}_{\tx}$ starts from the point 
on the second sheet of the Riemann surface of 
$\sqrt{Q_{J,0}(\tx,\tt)}$ corresponding to $\tx$, 
goes to the simple turning point $\ta(\tt_{\ast})$ 
along the Stokes segment $\tgamma_k$, 
encircles the simple turning point 
$\ta(\tt_{\ast})$ and returns to the 
point corresponding to $\tx$ on the first sheet
along the Stokes segment $\tgamma_k$. 
(The wiggly line designates the branch cut for 
$\sqrt{Q_{J,0}(\tx,\tt)}$.) 
The path $\delta^{(k)}_{x}$ 
is defined in the same manner for $J={\III}$.
As well as \eqref{eq:pre-key-equality}, 
taking the difference of both sides of
\eqref{eq:Sodd-relations-of-ykD7} for $k=1$ and $k=2$, 
we obtain
\begin{eqnarray} \label{eq:pre-key-equalityD7}
\int_{\tdelta} \tS_{J, {\rm odd}}
(\tx,\tt,\eta)~d\tx  & = & 
\int_{\delta} S_{\III, {\rm odd}}
(x,t(\tt,\eta),\eta)~dx \\[+.3em] 
& & \hspace{-10.em}
+ \sum_{n=0}^{\infty} 
\frac{\p^n S_{\III, {\rm odd}}}{\p x^n}
(x_0,t(\tt,\eta),\eta)~ 
\frac{(y^{(2)}-x_0)^{n+1}-(y^{(1)}-x_0)^{n+1}}
{(n+1) \hspace{+.1em} !}. 
\nonumber
\end{eqnarray}
Here $\tdelta = \tdelta_{+} + \tdelta_{-}$ 
is the sum of two closed cycles 
$\tdelta_{+}$ and $\tdelta_{-}$, where 
$\tdelta_{+}$ (resp., $\tdelta_{-}$)
encircles the double turning point $\tlambda_0(\tt)$ 
and all singular points contained in the domain $\tilde{W}$ 
(cf.~Assumption \ref{ass:transform-to-P3} (4)) 
in clockwise (resp., counter-clockwise) direction 
on the first (resp., the second) sheet of the 
Riemann surface of $\sqrt{\tQ_{J,0}(\tx,\tt)}$ 
as indicated in Figure \ref{fig:delta-kxD7}. 
The path $\delta$ is defined in the same manner 
for $J={\III}$.

Under Assumption \ref{ass:transform-to-P3} (5), 
there is no branch points of 
$\tS_{J, {\rm odd}}(\tx,\tt,\eta)$
inside the closed cycle $\tdelta$. Thus the 
left-hand side of \eqref{eq:pre-key-equalityD7} 
is written as
\begin{eqnarray}\label{eq:def-of-R}
\int_{\tdelta} \tS_{J, {\rm odd}}
(\tx,\tt,\eta)~d\tx & = & 
\int_{\tdelta_{+}} \tS_{J, {\rm odd}}
(\tx,\tt,\eta)~d\tx +
\int_{\tdelta_{-}} \tS_{J, {\rm odd}}
(\tx,\tt,\eta)~d\tx \\
& = & 4 \pi i \hspace{-.5em}
\Res_{\tx=\tlambda_0(\tt)}
\tS_{J, {\rm odd}}(\tx,\tt,\eta)~d\tx 
+ 4 \pi i {R} = \pi i \tE_J 
+  4 \pi i {R}, \hspace{-2.em} \nonumber
\end{eqnarray}
where ${R}$ is the sum of the residues of 
$\tS_{J, {\rm odd}}(\tx,\tt,\eta)~d\tx$ at singular points 
of $\tQ_{J,0}(\tx,\tt)$ contained in the domain $\tilde{W}$.
By the same discussion of the proof of 
Lemma \ref{lem:key-lamma} we have the following 
equality (cf.~\eqref{eq:Sodd-is-sqrtQ0}):
\[
4 \pi i {R} =
 \eta \int_{\tdelta} 
\sqrt{\tQ_{J,0}(\tx,\tt)}~d\tx. 
\]
Here note that $\tQ_{J,0}(\tx,\tt)$ is holomorphic at 
$\tx=\tlambda_0(\tt)$. On the other hand, 
using the equalities \eqref{eq:c-P3-2} and 
\eqref{eq:branches-P3D7}, we have 
\begin{eqnarray} 
\oint_{\tdelta} \sqrt{\tQ_{J,0}(\tx,\tt)}~d\tx & = &
2 \left( \int_{\tgamma_2} \sqrt{\tQ_{J,0}(\tx,\tt)} d\tx
- \int_{\tgamma_1} \sqrt{\tQ_{J,0}(\tx,\tt)} d\tx \right) 
= - 2 \pi i c. \hspace{-2.em} 
\end{eqnarray}
Then, it follows from \eqref{eq:def-of-R} that 
\begin{equation} 
\int_{\tdelta} \tS_{J, {\rm odd}}
(\tx,\tt,\eta)~d\tx = \pi i \tE_J - 2 \pi i c \eta.
\end{equation}
The same computation is also valid for $J = {\III}$
and we obtain 
\begin{equation}
\int_{\delta} S_{\III, {\rm odd}}
(x,t,\eta)~dx = \pi i E_{\III} - 2 \pi i c \eta
\end{equation}
from the equalities \eqref{eq:c-P3-3} and 
\eqref{eq:branches-P3D7-sono2}. 
Since the parameters $(\talpha,\tbeta)$ 
and $(\alpha,\beta)$ are chosen as 
\eqref{eq:invariance-of-residuesD7}, 
the equality \eqref{eq:pre-key-equalityD7} implies
\[
\sum_{n=0}^{\infty} 
\frac{\p^n S_{\III, {\rm odd}}}{\p x^n}
(x_0,t(\tt,\eta),\eta)~ 
\frac{(y^{(2)}-x_0)^{n+1}-(y^{(1)}-x_0)^{n+1}}
{(n+1) \hspace{+.1em} !} = 0.
\]
Therefore, by the induction argument we have 
the desired equality 
\eqref{eq:single-valuedness-of-transformation-y1-y2}
on the domain $\tU'$. Since 
$y^{(1)}(\tx,\tt,\eta)$ and $y^{(2)}(\tx,\tt,\eta)$
coincide at the simple turning point $\ta(\tt)$
as is noted above, we have proved the single-valuedness 
of the transformation series.
\end{proof}

Set 
\begin{equation} 
x(\tx,\tt,\eta)~\left( = 
x^{\rm pre}(\tx,\tt,\eta) = y^{(1)}(\tx,\tt,\eta) 
 = y^{(2)}(\tx,\tt,\eta) \right).
\end{equation}
Then, since the equations 
\eqref{eq:x-relation2} etc. also hold
if we replace ${\rm II}$ by ${\III}$, 
we have the equality 
\begin{equation}
x(\tlambda_J(\tt,\eta),\tt,\eta) = 
\lambda_{\III}(t(\tt,\eta),\eta). 
\end{equation}
Thus we have proved all claims in 
Theorem \ref{thm:main-theorem2}.


\subsection*{Acknowledgements}
The author is very grateful to Yoshitsugu Takei, 
Takahiro Kawai, Takashi Aoki, Tatsuya Koike, 
Shingo Kamimoto, Shinji Sasaki 
and Yasuhiro Wakabayashi for helpful advices. 
Many of ideas of the proof of our main results come 
from materials of the paper \cite{KT98} of Kawai and Takei. 
This research is supported by Research Fellowships 
of Japan Society for the Promotion for Young Scientists.

\end{document}